\documentclass[12pt]{amsart}

\usepackage{etex}
\usepackage[usenames,dvipsnames]{pstricks} 
\usepackage{epsfig}
\usepackage{graphicx,color}
\usepackage{geometry}
\usepackage[all]{xy}
\usepackage{amssymb,amscd}
\usepackage{cite}
\usepackage{fullpage}
\usepackage{marvosym}
\xyoption{poly}
\usepackage{url}

\usepackage{tikz}
\usepackage{tikz-cd}
\usetikzlibrary{decorations.pathmorphing}
\newtheorem{introtheorem}{Theorem}

\newtheorem{theorem}{Theorem}[section]
\newtheorem{lemma}[theorem]{Lemma}
\newtheorem{proposition}[theorem]{Proposition}
\newtheorem{corollary}[theorem]{Corollary}
\theoremstyle{definition}
\newtheorem{definition}[theorem]{Definition}

\newtheorem{remark}[theorem]{Remark}

\newtheorem*{question*}{Question}

\newtheorem*{questions*}{Questions}

\newtheorem*{steps*}{Answer/steps}

\newtheorem*{progress*}{Progress}

\newtheorem*{classification*}{Classification}

\newtheorem*{construction*}{Classification}
\newtheorem*{example*}{Example}

\newtheorem*{remark*}{Remark}
\newtheorem*{remarks*}{Remarks}
\newtheorem*{definition*}{Definition}

\usepackage{calrsfs}
\usepackage{url}
\usepackage{longtable}
\usepackage[OT2, T1]{fontenc}
\usepackage{textcomp}
\usepackage{times}
\usepackage[scaled=0.92]{helvet}

\renewcommand{\tilde}{\widetilde}

\newcommand{\Q}{\mathbb{Q}}

\newcommand{\Z}{\mathbb{Z}}

\newcommand{\F}{\mathbb{F}}

\newcommand{\PSL}{\mathrm{PSL}}

\newcommand{\X}{\mathcal{X}}

\DeclareMathOperator{\Res}{Res}
\DeclareMathOperator{\SL}{SL}
\DeclareMathOperator{\GL}{GL}

\DeclareMathOperator{\Aut}{Aut}

\DeclareSymbolFont{cyrletters}{OT2}{wncyr}{m}{n}
\DeclareMathSymbol{\Sha}{\mathalpha}{cyrletters}{"58}

\makeatletter

\def\greekbolds#1{%
 \@for\next:=#1\do{%
    \def\X##1;{%
     \expandafter\def\csname V##1\endcsname{\boldsymbol{\csname##1\endcsname}}
     }
   \expandafter\X\next;
  }
}

\greekbolds{alpha,beta,iota,gamma,lambda,nu,eta,Gamma,varsigma,Lambda}

\def\make@bb#1{\expandafter\def
  \csname bb#1\endcsname{{\mathbb{#1}}}\ignorespaces}

\def\make@bbm#1{\expandafter\def
  \csname bb#1\endcsname{{\mathbbm{#1}}}\ignorespaces}

\def\make@bf#1{\expandafter\def\csname bf#1\endcsname{{\bf
      #1}}\ignorespaces} 

\def\make@gr#1{\expandafter\def
  \csname gr#1\endcsname{{\mathfrak{#1}}}\ignorespaces}

\def\make@scr#1{\expandafter\def
  \csname scr#1\endcsname{{\mathscr{#1}}}\ignorespaces}

\def\make@cal#1{\expandafter\def\csname cal#1\endcsname{{\mathcal
      #1}}\ignorespaces} 

\def\do@Letters#1{#1A #1B #1C #1D #1E #1F #1G #1H #1I #1J #1K #1L #1M
                 #1N #1O #1P #1Q #1R #1S #1T #1U #1V #1W #1X #1Y #1Z}
\def\do@letters#1{#1a #1b #1c #1d #1e #1f #1g #1h #1i #1j #1k #1l #1m
                 #1n #1o #1p #1q #1r #1s #1t #1u #1v #1w #1x #1y #1z}
\do@Letters\make@bb   \do@letters\make@bbm
\do@Letters\make@cal  
\do@Letters\make@scr 
\do@Letters\make@bf \do@letters\make@bf   
\do@Letters\make@gr   \do@letters\make@gr
\makeatother

\def\ol{\overline}
\def\wt{\widetilde}

\def\ul{\underline}
\def\onto{\twoheadrightarrow}

\def\wh{\widehat}

\newcommand{\<}{\langle}  
\renewcommand{\>}{\rangle}

\newcommand{\isoto}{\stackrel{\sim}{\longrightarrow}}
\newcommand{\embed}{\hookrightarrow}
\def\Spec{{\rm Spec}\,}

\def\Fpbar{\overline{\bbF}_p}
\def\Fp{{\bbF}_p}

\def\Qp{{\bbQ}_p}

\def\Zp{{\bbZ}_p}

\def\ch{characteristic\ }

\newcommand{\A}{\mathbb A}    
\newcommand{\G}{\mathbb G}
  
\def\makeop#1{\expandafter\def\csname#1\endcsname
  {\mathop{\rm #1}\nolimits}\ignorespaces}
\makeop{Hom}   \makeop{End}   \makeop{Aut}   \makeop{Isom}  \makeop{Pic} 
\makeop{Gal}   \makeop{ord}   \makeop{Char}  \makeop{Div}   \makeop{Lie} 
\makeop{PGL}   \makeop{Corr}  \makeop{PSL}   \makeop{sgn}   \makeop{Spf}
\makeop{Spec}  \makeop{Tr}    \makeop{Nr}    \makeop{Fr}    \makeop{disc}
\makeop{Proj}  \makeop{supp}  \makeop{ker}   \makeop{im}    \makeop{dom}
\makeop{coker} \makeop{Stab}  \makeop{SO}    \makeop{SL}    \makeop{SL}
\makeop{Cl}    \makeop{cond}  \makeop{Br}    \makeop{inv}   \makeop{rank}
\makeop{id}    \makeop{Fil}   \makeop{Frac}  \makeop{GL}    \makeop{SU}
\makeop{Nrd}   \makeop{Sp}    \makeop{Tr}    \makeop{Trd}   \makeop{diag}
\makeop{Res}   \makeop{ind}   \makeop{depth} \makeop{Tr}    \makeop{st}
\makeop{Ad}    \makeop{Int}   \makeop{tr}    \makeop{Sym}   \makeop{can}
\makeop{length}\makeop{SO}    \makeop{torsion} \makeop{GSp} \makeop{Ker}
\makeop{Adm}   \makeop{Mat}

\DeclareMathOperator{\Mass}{Mass}

\newcommand{\dieu}{Dieudonn\'{e} }

\DeclareMathOperator{\Span}{Span}

\DeclareMathSymbol{\twoheadrightarrow} {\mathrel}{AMSa}{"10}

\DeclareMathOperator{\pr}{pr}

\def\sep{\mathrm{sep}}

\def\sfF{\mathsf{F}}
\def\sfV{\mathsf{V}}

\begin{document}

\title{Mass formula and Oort's conjecture 
for supersingular abelian threefolds}
\author{Valentijn Karemaker}
\address{(Karemaker) Mathematical Institute, Utrecht University, Utrecht, The Netherlands}
\email{V.Z.Karemaker@uu.nl}
\author{Fuetaro Yobuko}
\address{(Yobuko) Graduate School of Mathematics, Nagoya University, Nagoya, Japan}
\email{yobuko@math.nagoya-u.ac.jp}
\author{Chia-Fu Yu}
\address{(Yu) Institute of Mathematics, Academia  Sinica and National Center for Theoretic Sciences, Taipei, Taiwan}
\email{chiafu@math.sinica.edu.tw}

\date{\today}

\begin{abstract}
Katsura and Oort obtained an explicit description of the supersingular
locus $\calS_{3,1}$ of the Siegel modular variety of degree $3$ in terms of
class numbers. In this paper we study an alternative stratification of
$\calS_{3,1}$, the so-called mass stratification. 
We show that when $p\neq 2$, there are eleven 
strata (one of $a$-number $3$, two of $a$-number $2$ and eight of 
$a$-number $1$). We give an explicit mass formula for each stratum and
classify  
possible automorphism groups on each stratum of $a$-number one. 
On the largest open stratum we show that every automorphism group is 
$\{\pm 1\}$ if and only if $p \neq 2$; that is, we prove that Oort's
conjecture on the automorphism groups of generic supersingular
abelian threefolds holds precisely when $p>2.$ 
\end{abstract}

\maketitle
\setcounter{tocdepth}{2}

\section{Introduction}

Throughout this paper, let $p$ be a prime number, and let $k$ be an algebraically closed field of characteristic $p$. An abelian variety $X$ over $k$ is said to be \emph{supersingular} if it is isogenous to a product of
supersingular elliptic curves; it is called \emph{superspecial} if it is isomorphic to a product of supersingular elliptic curves. 
To each polarised supersingular abelian variety $x=(X_0,\lambda_0)$ of $p$-power polarisation degree, we associate a set $\Lambda_x$ of isomorphism classes of 
$p$-power degree polarised abelian varieties $(X,\lambda)$ over $k$, consisting of those whose associated quasi-polarised $p$-divisible groups satisfy $(X,\lambda)[p^\infty]\simeq (X_0,\lambda_0)[p^\infty]$. 
It is known that $\Lambda_x$ is a finite set, and
the \emph{mass} of $\Lambda_x$ is defined to be the weighted sum
\begin{equation}\label{eq:intromass}
  \mathrm{Mass}(\Lambda_x):=\sum_{(X,\lambda)\in \Lambda_x} \frac{1}{\vert \mathrm{Aut}(X,\lambda)\vert }. 
\end{equation}

Let $\calA_g$ be the moduli space over $\Fpbar$ 
of $g$-dimensional principally
polarised abelian varieties. 
If $x=(X_0,\lambda_0)$ is a superspecial point in $\calA_g(k)$, that
is, $X_0$ is superspecial,  
then $\Lambda_x$ coincides with the
superspecial locus $\Lambda_{g,1}$ of $\calA_g$, which
consists of all superspecial points in $\calA_g$, called the
\emph{principal genus}.
The classical mass formula (see Hashimoto--Ibukiyama \cite[Proposition 
9]{hashimotoibukiyama} and Ekedahl \cite[p.~159]{ekedahl}) states that  
\begin{equation}\label{eq:introsspmass}
  \mathrm{Mass}(\Lambda_{g,1})=\frac{(-1)^{g(g+1)/2}}{2^g} \left \{
  \prod_{i=1}^g \zeta(1-2i) \right \}\cdot
  \prod_{i=1}^{g}\left\{(p^i+(-1)^i\right \}, 
\end{equation}
where $\zeta(s)$ denotes the Riemann zeta function. 

More generally, for any integer $c$ with 
$0 \leq c \leq \lfloor g/2 \rfloor$, let $\Lambda_{g,p^c}$ denote the
finite set of isomorphism classes of $g$-dimensional polarised
superspecial abelian varieties $(X,\lambda)$ such that $\ker(\lambda)
\simeq \alpha_p^{2c}$, 
where $\alpha_p$ is the kernel of the Frobenius morphism on the
additive group $\mathbb{G}_a$.  
Then one also has $\Lambda_{g,p^c}=\Lambda_x$ for any
member $x$ in $\Lambda_{g,p^c}$. 
The case $c=\lfloor g/2 \rfloor$ is called the
\emph{non-principal genus}. 
As shown by Li-Oort \cite{lioort}, both the principal and
non-principal genera describe the irreducible components of the 
supersingular locus $\mathcal{S}_{g,1}$ of $\mathcal{A}_g$.
Similarly, the sets 
$\Lambda_{g,p^c}$ describe the irreducible components of 
supersingular Ekedahl-Oort (EO) strata in $\mathcal{A}_g$
cf.~\cite{harashita}. 
The explicit determination of the class number $\vert \Lambda_{g,p^c}\vert$, i.e., the class number problem, is a very difficult task for large $g$, and
is still open for $g=3$ and $c=1$. 
Nevertheless, an explicit calculation of the mass $\mathrm{Mass}(\Lambda_{g,p^c})$ is more accessible
and provides a good estimate for the class number. 
This mass was calculated explicitly by the third author \cite[Theorem 1.4]{yu2} when $g=2c$ and extended to arbitrary $g$ and~$c$ by Harashita \cite[Proposition 3.5.2]{harashita}. 

In \cite{yuyu}, J.-D. Yu and the third author explicitly calculated the
mass formula for $\Mass(\Lambda_x)$ for 
an arbitrary  principally polarised
supersingular abelian surface $x=(X_0,\lambda_0)$. 
In \cite{ibukiyama}, Ibukiyama investigated 
principal polarisations of a given supersingular
non-superspecial abelian surface $X_0$. 
He explicitly computed the number of polarisations and
the mass of the corresponding principally polarised abelian surfaces. 
He also showed the agreement with $\vert \Lambda_x \vert$ and 
$\mathrm{Mass}(\Lambda_x)$ cf.~\cite[Proposition 3.3 and Theorem
3.6]{ibukiyama}, respectively, for a member 
$x=(X_0,\lambda_0)$ in $\calS_{2,1}$. 
As an important arithmetic application, Ibukiyama 
proved Oort's conjecture that the automorphism group of 
any generic member is $\{\pm 1\}$ for 
$p\geq 3$, and he gave a counterexample for 
$p=2$.\\

Inspired by Ibukiyama's work \cite{ibukiyama}, and as 
a continuation of \cite{yuyu}, 
in this paper we 
completely determine the mass formula for $\mathrm{Mass}(\Lambda_x)$
when $g=3$, and prove Oort's conjecture for $p>2$ 
as an arithmetic application.
To describe our results, we introduce some notation; more details will 
be given in Sections~\ref{sec:formulae} and~\ref{sec:sslocus}. 

For any abelian variety $X$ over $k$, the
\emph{$a$-number} of $X$ is $a(X):=\mathrm{dim}_k
\mathrm{Hom}(\alpha_p, X)$. 
For abelian threefolds $X$ we have $a(X) \in \{1,2,3\}$; when computing the mass, we will separate into cases based on the $a$-number. 

Further let $E$ be a supersingular elliptic curve over
$\mathbb{F}_{p^2}$ with Frobenius endomorphism $\pi_E=-p$, and let
$E_k=E\otimes_{\mathbb{F}_{p^2}} k$. 
For each integer $c$ with $0\leq c \leq \lfloor g/2 \rfloor$, we denote by $P_{p^c}({E^g_k})$ the set of polarisations $\mu$ on ${E^g_k}$ such that $\mathrm{ker} \mu \simeq \alpha_p^{2c}$; one has $P_{p^c}({E^g_k})=P_{p^c}(E^g)$. 
As superspecial abelian threefolds are unique up to isomorphism, there
is a natural bijection $P_{p^c}({E^g_k}) \simeq
\Lambda_{g,p^c}$. 

Let $\mu$ be a polarisation in $P_1(E_k^3)$.
As alluded to above, Li and Oort \cite{lioort} show there is a one-to-one natural correspondence between the
set $P_1({E^3_k})$ and the set $\Sigma(\mathcal{S}_{3,1})$ of
(geometrically) irreducible components of $\mathcal{S}_{3,1}$.  
More precisely, they consider the moduli space $\mathcal{P}_{\mu}$
(resp.~$\mathcal{P}'_{\mu}$) over $\mathbb{F}_{p^2}$ of
three-dimensional (resp.~rigid) polarised flag type quotients with
respect to $\mu$. This space is an irreducible scheme which comes with
a proper projection morphism $\mathrm{pr}_0: \mathcal{P}_{\mu} \to
\mathcal{S}_{3,1}$, such that for each principally polarised
supersingular abelian threefold $(X,\lambda)$ there exist a $\mu \in
P_1(E_k^3)$ and a $y \in \mathcal{P}_{\mu}$ such that
$\mathrm{pr}_0(y) = [(X,\lambda)] \in \mathcal{S}_{3,1}$.  

Let $C\subseteq \mathbb{P}^2$ be the Fermat curve of degree $p+1$
defined by the equation $X_1^{p+1}+X_2^{p+1}+X_3^{p+1}=0$. There
exists a natural proper morphism $\pi: \mathcal{P}_{\mu} \to C$
with $\bbP^1$-fibers, and it is shown (cf.~\cite[Section 9.4]{lioort} and
Proposition~\ref{prop:explicitmoduli}) that 
$\mathcal{P}_{\mu}$ is isomorphic to the $\bbP^1$-bundle 
$\mathbb{P}_{C}(\mathcal{O}(-1)\oplus \mathcal{O}(1))$ 
over the Fermat curve $C$.
Moreover, the morphism $\pi$
has a section $s~:~C\isoto T\subseteq \calP_{\mu}$, 
cf.~Definition~\ref{def:T}. 
In particular, for each $k$-point $(X,\lambda)$ in the component
$\pr_0(\calP_\mu)$ of $\calS_{3,1}$ and a point 
$y \in \mathcal{P}_{\mu}(k)$ lying over $(X,\lambda)$, there exists a
unique 
pair $(t,u)$ where $t = (t_1:t_2:t_3) \in C(k)$ and $u = (u_1:u_2) \in
\pi^{-1}(t) \simeq \mathbb{P}^1_t(k)$ that characterises $y$.
Moreover, we have (cf.~Proposition~\ref{prop:sections}):
\begin{enumerate}
\item If $y \in T$ then $a(X) = 3$;
\item For any $t \in C(k)$, we have $t \in C(\mathbb{F}_{p^2})$ if and only if for any $y \in \pi ^{-1}(t)$ the corresponding threefold $X$ has $a(X) \geq 2$.
\item We have $a(X) = 1$ if and only if $y \notin T$ and $\pi (y) \notin C(\mathbb{F}_{p^2})$.
\end{enumerate}

We are now ready to state our first two main results, computing the mass for any principally polarised supersingular abelian threefold.

\begin{introtheorem}\label{introthm:a2} (Theorem~\ref{thm:massa2}) 
  Let $x = (X,\lambda)\in \mathcal{S}_{3,1}(k)$ with $a(X)\ge 2$, let $\mu\in P_1(E^3)$, and let $y\in \calP'_\mu(k)$ be such that $\mathrm{pr}_0(y) = [(X,\lambda)]$. 
  Write $y=(t,u)$ where $t=\pi(y)\in C(\F_{p^2})$ and $u\in \pi^{-1}(t) \simeq \mathbb{P}^1_t(k)$. Then 
\[
  \mathrm{Mass}(\Lambda_x)=\frac{L_p}{2^{10}\cdot 3^4\cdot 5\cdot 7}, 
\]
where 
\[
 L_p=
\begin{cases}
  (p-1)(p^2+1)(p^3-1) & \text{if } u\in
  \mathbb{P}_t^1(\mathbb{F}_{p^2}); \\ 
  (p-1)(p^3+1)(p^3-1)(p^4-p^2) & \text{if }
  u\in\mathbb{P}_t^1(\mathbb{F}_{p^4})\setminus
  \mathbb{P}_t^1(\mathbb{F}_{p^2}); \\ 
  2^{-e(p)}(p-1)(p^3+1)(p^3-1) p^2(p^4-1) & \text{ if
  } u \not\in 
  \mathbb{P}_t^1(\mathbb{F}_{p^4});
\end{cases} 
\]
where $e(p)=0$ if $p=2$ and $e(p)=1$ if $p>2$. 
\end{introtheorem}

\begin{introtheorem}\label{introthm:a1} (Theorem~\ref{thm:anumber1})
Let $x = (X,\lambda) \in \mathcal{S}_{3,1}(k)$ such that $a(X)=1$ and
$x\in \pr_0(\calP_\mu)$ for some $\mu \in P_1(E^3)$. 
Consider an element $y \in
\mathcal{P}_{\mu}(k)$ over $x$, 
which is characterised by the pair $(t,u)$ with $t
\in C(k)\setminus C(\mathbb{F}_{p^2})$ and $u \in
\mathbb{P}^1_t(k)$. Let $\calD_t$ be as in Definition \ref{def:D}, and let
$d(t)$ be as in Definition~\ref{def:dx}. 
Then 
\[
  \mathrm{Mass}(\Lambda_x)=\frac{p^3 L_p}{2^{10}\cdot 3^4\cdot 5\cdot 7}, 
\]
where 
\[
\begin{split}
L_p = \begin{cases}
2^{-e(p)}p^{2d(t)}(p^2-1)(p^4-1)(p^6-1) & \text{ if } u \notin \calD_t; \\
p^{2d(t)}(p-1)(p^4-1)(p^6-1) & \text{ if } t \notin C(\mathbb{F}_{p^6}) \text{ and } u \in \calD_t; \\
p^6(p^2-1)(p^3-1)(p^4-1) & \text{ if } t \in C(\mathbb{F}_{p^6}) \text{ and } u \in \calD_t.
\end{cases}
\end{split}
\]
\end{introtheorem}

The mass function on $\calS_{3,1}$ induces a stratification such that the mass function becomes constant on each stratum.
By Theorem~\ref{introthm:a2}, the locus of $\calS_{3,1}$ with
$a$-number $\ge 2$ decomposes into three strata: 
one stratum with $a$-number $3$ and two strata
with $a$-number $2$. On the locus with $a$-number $1$, the
stratification depends on
$p$. When $p\ne 2$, the $d$-invariant takes values in $\{3,4,5,6\}$
and $d(t)=3$ if and only if $t\in C(\F_{p^6})$. In this case, 
Theorem~\ref{introthm:a1} says that the mass function
depends only on the $d$-invariant and whether $u\in \calD_t$ or not, 
and hence there are eight strata. 
When $p=2$, the $d$-value $d(t)$ is always $3$ and
Theorem~\ref{introthm:a1} gives three strata. 

Our computations of the automorphism groups can be summarised as follows.

\begin{introtheorem}\label{introthm:Oort}
Let $x = (X,\lambda) \in \mathcal{S}_{3,1}(k)$ and $\mu \in P_1(E^3)$ so that $x\in \pr_0(\calP_\mu)$.
Consider an element $y \in \mathcal{P}_{\mu}$ over~$x$, which is characterised by the pair $(t,u)$ with $t \in C(k)$ and $u \in \mathbb{P}^1_t(k)$.
Let $\calD_t$ be as in Definition~\ref{def:D} and let $d(t)$ be as in Definition~\ref{def:dx}.
\begin{enumerate}
\item (Theorem~\ref{thm:gen_autgp}) 
Suppose that $a(X) =1$, so that $t \in C(k)\setminus C(\mathbb{F}_{p^2})$. Assume that $(t,u)\not \in \calD$, that is, $u\not\in \calD_t$. 
  \begin{enumerate}
  \item If $p=2$, then $\Aut(X,\lambda)\simeq C_2^3$.
  \item If $p\ge 5$, or $p=3$ and $d(t)=6$, then
$\Aut(X,\lambda)\simeq C_2$,
\end{enumerate}
where $C_n$ denotes the cyclic group of order $n$. 
\item (Theorem~\ref{thm:inD})
Suppose that $a(X) = 1$ and that $(t,u)\in \calD$ with $t\not \in C(\F_{p^6})$. 
\begin{enumerate}
\item If $p=2$, then $\Aut(X,\lambda)\simeq C_2^3 \times C_3$.
\item If $p=3$ and $d(t)=6$, then $\Aut(X,\lambda)\in
\{C_2,C_4\}$.
\item For $p\ge 5$, we have the following cases:
\begin{itemize}

\item [(i)] If $p\equiv -1 \pmod {4}$, then $\Aut(X,\lambda)\in
  \{C_2,C_4 \}$. 
\item [(ii)] If $p\equiv -1 \pmod {3}$, then $\Aut(X,\lambda)\in
  \{C_2,C_6\}$.
\item [(iii)] If $p\equiv 1 \pmod {12}$, then $\Aut(X,\lambda)\simeq C_2$.  
\end{itemize}
\end{enumerate}
\item (Proposition~\ref{prop:asympt})
  Let $\Lambda_{3,1}(C_2):=\{(X,\lambda)\in \Lambda_{3,1}:
\Aut(X,\lambda)\simeq C_2\}$ be the set of superspecial principally polarised abelian threefolds satisfying Oort's conjecture. Then 
\[
  \frac{\vert \Lambda_{3,1}(C_2)\vert }{\vert \Lambda_{3,1}\vert }\to 1 \quad \text{as\ 
$p\to \infty$}.
\]
\end{enumerate}
\end{introtheorem}

In particular, Part (1) of Theorem~\ref{introthm:Oort} shows that
Oort's conjecture is true precisely for $p \neq 2$. That is, every
generic principally polarised supersingular abelian threefold over $k$
of characteristic $\neq 2$ has automorphism group~$C_2$.

Schemes in this paper are assumed to be locally Noetherian unless
stated otherwise. \\

The organisation of the paper is as follows. Sections~\ref{sec:formulae} and~\ref{sec:sslocus} contain preliminaries, respectively on mass formulae and the structure of the supersingular locus $\mathcal{S}_{3,1}$. In particular, the strategy we will follow in later sections to obtain mass formulae is outlined at the end of Section~\ref{sec:formulae}. Sections~\ref{sec:a2} and~\ref{sec:a1} determine the mass formulae for supersingular abelian threefolds $X$, respectively with $a(X) = 2$ (cf.~Theorem~\ref{introthm:a2}) and $a(X)=1$ (cf.~Theorem~\ref{introthm:a1}). The automorphism groups, as well as the implications for Oort's conjecture, are studied in Section~\ref{sec:Aut} (cf.~Theorem~\ref{introthm:Oort}).  The Appendix contains results of independent interest, concerning a set-theoretic intersection arising in Section~\ref{sec:a1}. 

\section*{Acknowledgements}

Parts of this work were carried out when the first author visited the Academia Sinica, and when the first and third authors visited RIMS and Kyoto University. They would like to thank these institutes for their hospitality and excellent working conditions.
A part of this paper is contained in the second author's master's thesis written at Tohoku University; he thanks his advisor Nobuo Tsuzuki for enlightening comments, advice and encouragement. 
The authors are grateful to Ming-Lun Hsieh and Akio Tamagawa for
useful discussions, and for proving Propositions~\ref{prop:ML} and
\ref{prop:akio}, respectively. They would like to thank Tomoyoshi
Ibukiyama and Jiangwei Xue for useful discussions and helpful comments
on an earlier manuscript, and the anonymous referee for their comments which improved the exposition.
The second author is supported by JSPS grants 15J05073 and 19K14501.
The third author is partially supported by MoST grants
107-2115-M-001-001-MY2 and 109-2115-M-001-002-MY3. 

\section{Mass formulae for supersingular abelian
  varieties}\label{sec:formulae} 

\subsection{Set-up and notation}\label{ssec:not}\

Throughout the paper, let $p$ be a prime number, let $g$ be a positive integer, and let $k$ be an
algebraically closed field of characteristic $p$. The ground field for objects studied is $k$, unless
stated otherwise.

For a finite set $S$, write $\vert S\vert $ for the cardinality of $S$. 
Let $\alpha_p$ be the unique $\alpha$-group of order $p$ over $\Fp$;
it is defined to be the kernel of the Frobenius morphism on the additive group $\G_a$
over $\Fp$. 
For a matrix $A=(a_{ij}) \in \Mat_{m\times n}(k)$ and integer $r$, write
$A^{(p^r)}:=(a_{ij}^{p^r})$ for the image of $A$ under the $r$th
Frobenius map. 
Denote by $\wh \Z=\prod_{\ell} \Z_\ell$ the profinite completion of
$\Z$ and by $\A_f=\wh \Z\otimes_{\Z} \Q$ the finite adele ring of $\Q$. 

\begin{definition}\label{def:Sgpm}
For any integer $d\ge 1$, let $\calA_{g,d}$ denote the (coarse) moduli
space over $\Fpbar$ of $g$-dimensional polarised abelian varieties
$(X,\lambda)$ with polarisation degree $\deg \lambda=d^2$.
For any $m \geq 1$, let $\calS_{g,p^m}$ be the
supersingular locus of $\calA_{g,p^m}$, which consists of all
polarised supersingular abelian varieties in $\calA_{g,p^m}$.
Then $\mathcal{S}_{g,1}$ is the moduli space of 
$g$-dimensional principally polarised 
supersingular abelian varieties. Denote $\mathcal{S}_{g,p^*} =
\cup_{m\geq 1}\mathcal{S}_{g,p^m}$. 
\end{definition}

\begin{definition}\label{def:Lambdax}
(1) If $S$ is a finite set of objects with finite
    automorphism groups in a specified category, 
    then we define the \emph{mass} of $S$ 
    to be the weighted
    sum
\[
 \Mass(S):=\sum_{s\in S} \frac{1}{\vert \Aut(s)\vert }.      
\]
(2) For any $x = (X_0, \lambda_0) \in \mathcal{S}_{g,p^*}(k)$, we define 
\begin{equation}\label{eq:Lambdaxo}
\Lambda_{x} = \{ (X,\lambda) \in \mathcal{S}_{g,p^*}(k) : (X,\lambda)[p^{\infty}] \simeq (X_0, \lambda_0)[p^{\infty}] \},
\end{equation}
where $(X,\lambda)[p^{\infty}]$ denotes the polarised $p$-divisible
group associated to $(X,\lambda)$. Then $\Lambda_x$ is a finite set; see \cite[Theorem 2.1]{yu}.
The \emph{mass} of $\Lambda_x$ is defined as  
\[
\mathrm{Mass}(\Lambda_{x}) = \sum_{(X,\lambda) \in \Lambda_{x}} \frac{1}{\vert \mathrm{Aut}(X,\lambda)\vert}.
\]
\end{definition}

\subsection{Superspecial mass formulae}\label{ssec:sspmass}\

Recall that a superspecial abelian variety over $k$ is an abelian
variety isomorphic to a product of supersingular elliptic curves. 

\begin{definition}\label{def:Lambda}
Let $0 \leq c \leq \lfloor g/2 \rfloor$ be an integer. We
define $\Lambda_{g,p^c}$ to be the set of isomorphism classes of
$g$-dimensional superspecial polarised abelian varieties $(X,
\lambda)$ whose polarisation $\lambda$ satisfies $\ker(\lambda) \simeq
\alpha_p^{2c}$. Its mass is 
\[
\mathrm{Mass}(\Lambda_{g,p^c}) = \sum_{(X,\lambda)\in \Lambda_{g,p^c}}
\frac{1}{\vert \mathrm{Aut}(X,\lambda) \vert}. 
\]
\end{definition}

In particular, $\Mass(\Lambda_{g,p^c})$ is a special case of
$\Mass(\Lambda_x)$, cf.~Definition~\ref{def:Lambdax}. 
Note that the $p$-divisible group
of a superspecial abelian variety of given dimension is unique up to
isomorphism. Furthermore, the polarised $p$-divisible group associated
to any member in $\Lambda_{g,p^c}$ is unique up to isomorphism, 
cf.~\cite[Proposition 6.1]{lioort}.   
Thus, if $x = (X, \lambda)$ is any member in~$\Lambda_{g,p^c}$, then we have 
$\Lambda_x = \Lambda_{g,p^c}$. 

\begin{theorem}\label{thm:sspmass}
\begin{enumerate}
\item 
For any $g \ge 1$, we have 
\[
\mathrm{Mass}(\Lambda_{g,1}) = \frac{(-1)^{g(g+1)/2}}{2^g} \prod_{i=1}^{g} \zeta(1-2i) \cdot \prod_{i=1}^g (p^i + (-1)^i).
\]
\item 
For any $g \ge 1$ and $0 \leq c \leq \lfloor g/2 \rfloor$, we have
\[
\begin{split}
\mathrm{Mass}(\Lambda_{g,p^c}) = &
\frac{(-1)^{g(g+1)/2}}{2^g} \prod_{i=1}^{g} \zeta(1-2i) 
 \cdot \prod_{i=1}^{g-2c} (p^i + (-1)^i) \cdot \prod_{i=1}^c
 (p^{4i-2}-1) \\
& \cdot \frac{\prod_{i=1}^g
 (p^{2i}-1)}{\prod_{i=1}^{2c}(p^{2i}-1)\prod_{i=1}^{g-2c} (p^{2i}-1)}.   
\end{split}
\]
\end{enumerate}
\end{theorem}
\begin{proof}
  (1) See \cite[p.~159]{ekedahl} and \cite[Proposition
      9]{hashimotoibukiyama}. (2) This follows from \cite[Proposition
      3.5.2]{harashita} by the functional equation for $\zeta(s)$. See
      also 
      \cite{yu2} for a geometric proof in the case where $g=2c$.  
\end{proof}

Using the fact that $\zeta(-1)=-1/12, \zeta(-3)=1/120$ and
$\zeta(-5)=-1/(42\cdot 6)$, we obtain the following corollary.

\begin{corollary}\label{cor:sspmassg3}
Let $g=3$. 
\begin{enumerate}
\item
If $c=0$, then $\Lambda_{g,p^c} = \Lambda_{3,1}$ consists of all
principally polarised superspecial abelian threefolds, and 
\begin{equation}\label{eq:ppg3ssp}
\mathrm{Mass}(\Lambda_{3,1}) = \frac{(p-1)(p^2+1)(p^3-1)}{2^{10} \cdot 3^4 \cdot 5 \cdot 7}.
\end{equation}
\item
If $c=1$, then  $\Lambda_{g,p^c} = \Lambda_{3,p}$ consists of all
polarised superspecial abelian threefolds whose polarisation $\lambda$ has $\ker(\lambda) \simeq \alpha_p \times \alpha_p$, and 
\begin{equation}\label{eq:npg3ssp}
\mathrm{Mass}(\Lambda_{3,p}) = \frac{(p-1)(p^3+1)(p^3-1)}{2^{10} \cdot
  3^4 \cdot 5 \cdot 7}. 
\end{equation}
\end{enumerate}
\end{corollary}

\subsection{From superspecial to supersingular mass formulae}\

For a (not necessary principally) polarised supersingular 
abelian variety $x = (X _0, \lambda _0) $ over $k$, let $G_{x}$ be the
automorphism group scheme over $\mathbb{Z}$ associated to $x$ ; for
any commutative ring $R$, the group of its $R$-valued points is
defined by  
\begin{equation}\label{eq:aut}
G_{x}(R) = \{ g \in (\text{End}(X_0)\otimes _{\mathbb{Z}}R)^{\times} : g^T \lambda_0 g  = \lambda_0\}. 
\end{equation}

\begin{definition}\label{def:arithmass}
For a connected reductive group $G$ over $\Q$ with finite arithmetic
subgroups and an open  compact subgroup $U \subseteq G(\mathbb{A}_{f})$, we define its (arithmetic) mass $\mathrm{Mass}(G, U)$ by 
\begin{align*}
\mathrm{Mass}(G, U) = \sum
_{i=1}^h\frac{1}{\vert \Gamma_i \vert}, \quad \Gamma_i:=G(\mathbb{Q}) \cap
c_iUc_i^{-1},  
\end{align*}
where $\{c_1, \cdots , c_h\}$ is a set of representatives for the
double coset space $G(\mathbb{Q}) \backslash G(\mathbb{A}_{f}) \slash U$.
\end{definition}

\begin{proposition}\label{prop:geomarithmass}
For any object $ x = (X_0, \lambda_0) \in \mathcal{S}_{g,p^*}(k)$, 
there is a natural bijection of pointed sets
\begin{align*}
\Lambda_x \simeq G_{x}(\mathbb{Q}) \backslash G_{x}(\mathbb{A}_{f})
\slash G_{x}(\wh{\mathbb{Z}}).
\end{align*}
Moreover, if $(X,\lambda)$ is a member of $\Lambda_x$ which corresponds
to the class $[c]$ under the bijection, then  $\mathrm{Aut}(X,\lambda)
\simeq G_{x}(\mathbb{Q}) \cap cG_{x}(\wh{\mathbb{Z}})c^{-1}$. 
In particular, we have
\begin{align*}
{\rm Mass}(\Lambda_{x}) = {\rm Mass}(G_{x}, G_{x}(\wh{\mathbb{Z}})),
\end{align*}
cf. Definition \ref{def:Lambdax}.
\end{proposition}
\begin{proof}
 See \cite[Theorems 2.2 and 4.6]{yu3}. Also see 
\cite[Proposition 2.1]{yuyu} for a proof sketch. 
\end{proof}

\begin{definition}\label{def:mu}
Let $U_1,U_2$ be two open compact subgroups of
$G_{x}(\mathbb{A}_{f})$. Then we define 
\[
\mu(U_1/U_2) = \frac{[U_1 : U_1 \cap U_2]}{[U_2 : U_1 \cap U_2]}.
\]
\end{definition}

Interpreting the mass from Definition \ref{def:arithmass} as the volume of a fundamental domain, with notation as above, we have the following lemma.

\begin{lemma}\label{lem:massesU1U2}
Let $U_1,U_2$ be two open compact subgroups of $G_{x}(\mathbb{A}_{f})$. Then their (arithmetic) masses compare as
\[
\mathrm{Mass}(G_{x}, U_2) = \mu (U_1 / U_2){\rm Mass}(G_{x}, U_1). 
\]
\end{lemma}

\begin{lemma}\label{lem:minisog}
  Let $X$ be a supersingular abelian variety over $k$. Then there
  exists a pair $(Y,\varphi)$, where $Y$ is a superspecial abelian
  variety and $\varphi: Y\to X$ is an isogeny such that for any pair
  $(Y',\varphi')$ as above there exists a unique isogeny $\rho: Y'\to Y$
  such that $\varphi'=\varphi\circ \rho$. 

  Dually, there exists a pair
  $(Z,\gamma)$, where $Z$ is a superspecial abelian variety and
  $\gamma: X\to Z$ such that for any pair $(Z',\gamma')$ as above
  there exists a unique isogeny $\rho: Z\to Z'$ such that
  $\gamma'=\rho\circ \gamma$.  
\end{lemma}
\begin{proof}
  See \cite[Lemma 1.8]{lioort}; also see \cite[Corollary
  4.3]{yu:mrl2010} for an independent proof. The proof of \cite[Lemma
  1.8]{lioort} contains a gap; see Remark~\ref{countexample:miniso}
  for a counterexample to the argument. 
\end{proof}

\begin{definition}\label{def:minisog}
  Let $X$ be a supersingular abelian variety over $k$.
  We call the pair $(Y,\varphi:Y\to X)$ or the pair $(Z,\gamma:X\to
  Z)$ as in Lemma~\ref{lem:minisog} \emph{the minimal isogeny} of $X$.
\end{definition}

\begin{proposition}\label{prop:compmass}
Let $x = (X,\lambda)\in \mathcal{S}_{g,p^*}(k)$ and let 
$\varphi: \tilde{X} \to X$ be the minimal isogeny of $X$. Put
$\tilde{x} = (\tilde{X},\tilde{\lambda})$, where
$\tilde{\lambda}:=\varphi^* \lambda$. 
Let $(M,\langle\, , \rangle), (\tilde{M}, \langle\, , \rangle)$ denote the
quasi-polarised (contravariant) Dieudonn{\'e} module of $X, \tilde{X}$,
respectively. 
Then $\varphi$ induces an injective map $\varphi^*:
\End(X[p^\infty])\hookrightarrow \End(\tilde{X}[p^\infty])$, 
or equivalently $\varphi^*: \End(M)\hookrightarrow \End(\tilde{M})$,  
and we have
\begin{equation}\label{eq:compmassaut}
\begin{split}
\mathrm{Mass}(\Lambda_x) &=
[\mathrm{Aut}((\tilde{X},\tilde{\lambda})[p^{\infty}]):
\mathrm{Aut}((X,\lambda)[p^{\infty}])] 
\cdot \mathrm{Mass}(\Lambda_{\tilde{x}}) \\     
&=
[\mathrm{Aut}(\tilde{M}, \langle\, , \rangle):
\mathrm{Aut}(M,\langle\, , \rangle)] 
\cdot \mathrm{Mass}(\Lambda_{\tilde{x}}).      
\end{split}
\end{equation} 
Here the injective map $\varphi^*$ yields the inclusion 
map $\mathrm{Aut}(M,\langle\, , \rangle) \subseteq \mathrm{Aut}(\tilde{M}, \langle\, , \rangle)$. 
\end{proposition}

\begin{proof}
This may be regarded as a refinement of \cite[Theorem 2.7]{yu}. 
Through the isogeny $\varphi$, we may view $G_{\tilde{x}}(\wh{\mathbb{Z}})$ and $\varphi^*G_x(\wh{\mathbb{Z}})$ as open compact subgroups of the same group $G_{\tilde{x}}(\mathbb{A}_f)$.
Using Proposition \ref{prop:geomarithmass} 
and Lemma \ref{lem:massesU1U2}, we see that
\[
\begin{split}
\mathrm{Mass}(\Lambda_x) &= \mu(G_{\tilde{x}}(\wh{\mathbb{Z}}) / \varphi^* G_{x}(\wh{\mathbb{Z}})) \mathrm{Mass}(\Lambda_{\tilde{x}}) \\
&= \frac{[G_{\tilde{x}}(\wh{\mathbb{Z}}):G_{\tilde{x}}(\wh{\mathbb{Z}}) \cap \varphi^* G_{x}(\wh{\mathbb{Z}})]}{[\varphi^* G_{x}(\wh{\mathbb{Z}}):G_{\tilde{x}}(\wh{\mathbb{Z}}) \cap \varphi^* G_{x}(\wh{\mathbb{Z}})]} \mathrm{Mass}(\Lambda_{\tilde{x}}).
\end{split}
\]
Note that $G_{\tilde{x}}(\wh{\mathbb{Z}})$ and $\varphi^*
G_{x}(\wh{\mathbb{Z}})$ differ only at $p$. By \cite[Proposition
4.8]{yu:mrl2010}, every endomorphism of $X[p^\infty]$ lifts uniquely to
an endomorphism of $\tilde{X}[p^\infty]$. This shows the injectivity
of the map $\varphi^*: \End(X[p^\infty]) \to \End(\tilde{X}[p^\infty])$. 
Therefore, we have 
the inclusion $G_x(\mathbb{Z}_p) =
\mathrm{Aut}((X,\lambda)[p^{\infty}])\hookrightarrow 
G_{\tilde{x}}(\mathbb{Z}_p)=
\mathrm{Aut}((\tilde{X},\tilde{\lambda})[p^{\infty}])$ via
$\varphi^*$ and find the
first part of Equation \eqref{eq:compmassaut}. 

By Dieudonn{\'e} module theory, for any polarised supersingular 
abelian variety $(X,\lambda)$ with quasi-polarised  
Dieudonn{\'e} module 
$(M,\langle\, , \rangle)$, we may identify
$\mathrm{Aut}((X,\lambda)[p^{\infty}])$ 
with $\mathrm{Aut}(M,\langle\, , \rangle)$. This yields Equation
\eqref{eq:compmassaut}. 
\end{proof}

To summarise, the results of this section provide the following strategy for
obtaining a mass formula for any principally polarised supersingular
abelian variety: 
\begin{itemize}
\item[(a)] For any supersingular abelian variety $x = (X,\lambda)$,
  construct the minimal isogeny \newline
  $\varphi:(\tilde{X},\tilde{\lambda}) \to (X,\lambda)$ from a
  suitable superspecial abelian variety $\tilde{x} =
  (\tilde{X},\tilde{\lambda})$. 
\item[(b)]  Use Theorem~\ref{thm:sspmass} 
 (or Corollary~\ref{cor:sspmassg3} if $g=3$) to compute
  $\mathrm{Mass}(\Lambda_{\tilde{x}})$. 
\item[(c)] Compute the local index 
$[\mathrm{Aut}(\tilde{M},\langle , \rangle) : \mathrm{Aut}((M,\langle
, \rangle)]$, cf.~\eqref{eq:compmassaut}.
\item[(d)] Compute $\mathrm{Mass}(\Lambda_x)$, i.e., compare $\Mass(\Lambda_{\tilde x})$ and $\Mass(\Lambda_x)$ by applying Proposition~\ref{prop:compmass}.
\end{itemize}

We will carry out these steps, in particular Step~(c), in the next 
sections in the case where $g=3$. In the next section, we start by studying in detail the moduli space $\mathcal{S}_{3,1}$ of supersingular principally polarised abelian threefolds and the minimal isogenies (cf. Definition \ref{def:minisog}) between threefolds.

\section{Structure of the supersingular locus $\mathcal{S}_{3,1}$}
\label{sec:sslocus}

In this section we describe the supersingular locus $\mathcal{S}_{3,1}$. Its
structure will be used to determine minimal isogenies,
cf.~Proposition~\ref{prop:miniso}. Finer structures will be introduced
in order to compute the local index in Step (c) in the previous section.

\subsection{The supersingular locus $\boldsymbol{\mathcal{S}_{g,1}}$ and the
  mass function}\label{ssec:mod} \

To describe the moduli space $\mathcal{S}_{3,1}$ of 
supersingular principally
polarised abelian threefolds, we will use the framework of polarised
flag type quotients (for $g=3$) as developed by Li and Oort
\cite{lioort}, which we will briefly describe below (for any $g\ge 1$). 
Then we will introduce the stratification of $\calS_{g,1}$ 
induced by the mass values and its local analogue. 

For any abelian variety $X$, denote by $P(X)$ the set of isomorphism
classes of principal polarisations on $X$.
 
Let $E/\mathbb{F}_{p^2}$ be a supersingular elliptic curve whose
Frobenius endomorphism is $\pi_E = -p$ and denote $E_k = E
\otimes_{\F_{p^2}} k$. Since every polarisation on ${E^g_k}$ 
is defined over $\F_{p^2}$, 
we may identify $P({E_k^g})$ with $P(E^g)$.  
Recall that an $\alpha$-group of rank $r$ over an $\Fp$-scheme $S$  
is a finite flat group scheme over $S$ which is Zariski-locally 
isomorphic to $\alpha_p^r$. For a scheme $X$ over $S$,  put
$X^{(p)}:=X\times_{S,F_S} S$, where $F_S:S\to S$ denotes the absolute Frobenius morphism on $S$, and denote by $F_{X/S}:X\to X^{(p)}$ the
relative Frobenius morphism.

For each integer $i\ge 0$, let $P(E^g,i)$ be the set of isomorphism
classes of 
polarisations $\lambda$ on $E^g$ such that $\ker \lambda=E[\sfF^{i}]$
with $\sfF=F_{E/\F_{p^2}}$ and set $P^*(E^g):=P(E^g, g-1)$. 
The map $\lambda\mapsto p^{\lfloor
  (g-1)/2\rfloor} \lambda$ gives a bijection 
$P(E^g)\isoto P^*(E^g)$ if $g$ is odd and
$P(E^g,1)\isoto P^*(E^g)$ otherwise. 
Moreover, the map
$\lambda\mapsto (E^g_k, \lambda)$ gives a bijection 
$P(E^g)\isoto \Lambda_{g,1}$ when $g$ is odd and
$P(E^g,1)\isoto \Lambda_{g,p^c}$ when $g=2c$ is even. Thus, 
\begin{equation}
  \label{eq:P*Eg}
  P^*(E^g)\simeq 
  \begin{cases}
    \Lambda_{g,1}, & \text{if $g$ is odd}; \\
    \Lambda_{g,p^c}, & \text{if $g=2c$ is even.}
  \end{cases}
\end{equation}
It is known that $|\Lambda_{g,1}|=H_g(p,1)$ for any positive 
integer $g$ 
and $|\Lambda_{g,p^c}|=H_g(1,p)$ for any even positive integer
$g=2c$, where $H_{g}(p,1)$ (resp.~$H_g(1,p)$) is the class number of
principal genus (resp. the non-principal genus); see \cite{lioort} for
details.  

\begin{definition} (cf.~\cite[Section 3]{lioort})
\begin{enumerate}
\item Let $g\ge 1$ be an integer. 
  For any $\mu \in P^*(E^g)$, a \emph{$g$-dimensional polarised flag
  type quotient (PFTQ)} with respect to $\mu$ is a chain of
$g$-dimensional polarised
abelian schemes over a base $\F_{p^2}$-scheme $S$
\[ (Y_\bullet,\rho_\bullet):(Y_{g-1},\lambda_{g-1}) 
\xrightarrow{\rho_{g-1}} (Y_{g-2},\lambda_{g-2}) \xrightarrow{\rho_{g-2}} \cdots 
\xrightarrow{\rho_2} (Y_{1},\lambda_{1})\xrightarrow{\rho_1} (Y_0, \lambda_0),\]
such that: 
\begin{itemize}
\item [(i)] $(Y_{g-1},\lambda_{g-1}) = ({E^g}, \mu)\times_{\Spec \F_{p^2}} S$;
\item [(ii)] $\ker(\rho_i)$ is an $\alpha$-group of rank $i$ 
for $1\le i\le g-1$;
\item [(iii)] $\ker(\lambda_i) \subseteq \ker (\sfV^j \circ \sfF^{i-j})$
  for $0\le i\le g-1$ and $0\le j\le \lfloor i/2 \rfloor$, where
  $\sfF=F_{Y_i/S}: Y_i\to Y_i^{(p)}$  
  and $\sfV=V_{Y_i/S}:Y_i^{(p)}\to Y_i$ are the 
  relative Frobenius and Verschiebung morphisms, respectively.    
\end{itemize}
In particular, $\lambda_0$ is a principal polarisation on $Y_0$. 
An isomorphism of $g$-dimensional polarised flag type quotients is a
chain of isomorphisms $(\alpha_i)_{0\le i \le g-1}$ of polarised abelian
varieties such that $\alpha_{g-1}={\rm id}_{Y_{g-1}}$. 
\item A $g$-dimensional polarised flag
  type quotient $(Y_\bullet,\rho_\bullet)$ is said to be
\emph{rigid} if 
\[ \ker(Y_{g-1}\to Y_i)=\ker (Y_{g-1}\to Y_0)\cap
Y_{g-1}[\sfF^{g-1-i}],\quad  \text{for $1\le i \le g-1$}, \]
where $Y_{g-1}[\sfF^{g-1-i}]:=\ker (\sfF^{g-1-i}:{Y_{g-1}}\to 
Y_{g-1}^{(p^{g-1-i})})$.
\item Let $\mathcal{P}_{g,\mu}$ (resp.~$\calP'_{g,\mu}$) denote the moduli
space over $\F_{p^2}$ of 
$g$-dimensional (resp.~rigid) polarised flag
type quotients with respect to $\mu$.   
\end{enumerate}
\end{definition}

Clearly, each member $Y_i$ of $(Y_\bullet,\rho_\bullet)$ is a 
supersingular abelian variety. 

\begin{definition}\label{def:anumber}
For an abelian variety $X$ over $k$, its $a$-number is defined as 
\[
a(X) := \dim_k \mathrm{Hom}(\alpha_p,X).
\]
The $a$-number of a \dieu module $M$ over $k$ is defined as
$a(M) := \dim(M/(\sfF,\sfV)M)$. 
If $M$ is the Dieudonn{\'e} module of $X$, then $a(M) = a(X)$.
When $x \in \calP_{g,\mu}$ corresponds to a polarised flag type quotient
$(Y_{g-1},\lambda_{g-1}) \to \cdots \to (Y_1,\lambda_1) \to (Y_0, \lambda_0)$, 
we say that its $a$-number is $a(x) = a(Y_0)$. 
\end{definition}

According to \cite[Lemma 3.7]{lioort}, $\mathcal{P}_{g,\mu}$ is 
a projective scheme over $\mathbb{F}_{p^2}$ and 
$\calP'_{g,\mu}\subset \calP_{g,\mu}$ is an open subscheme. 
Thus, $\calP'_{g,\mu}$ a quasi-projective scheme over $\F_{p^2}$. 
The projection to the last member gives 
a proper $\overline{\mathbb{F}}_p$-morphism 
\begin{align*}
\mathrm{pr}_0 : \mathcal{P}_{g,\mu,\Fpbar} & \to \mathcal{S}_{g,1}, \\
(Y_\bullet,\rho_\bullet) & \mapsto (Y_0, \lambda_0).
\end{align*}

\begin{theorem}[Li-Oort]\
\begin{enumerate}
\item The natural morphism 
\begin{equation}\label{eq:moduli}
\mathrm{pr}_0: 
\coprod _{\mu \in P^*(E^g)}\mathcal{P'}_{g,\mu, \Fpbar} 
\rightarrow \mathcal{S}_{g,1}
\end{equation}
is quasi-finite and surjective. 
\item For every  $\mu\in P^*(E^g)$, the scheme $\calP'_{g,\mu}$ is
non-singular and geometrically irreducible of 
dimension $\lfloor g^2/4\rfloor$. Moreover, the
$a$-number $1$ locus $\calP'_{g,\mu}(a=1)$ is open and dense in
$\calP'_{g,\mu}$.
\item The morphism $\mathrm{pr}_0$ induces a surjective 
birational morphism 
\begin{equation}\label{eq:birational}
\mathrm{pr}_0: 
\coprod _{\mu \in P^*(E^g)}\mathcal{P'}_{g,\mu, \Fpbar}/G_\mu  
\rightarrow \mathcal{S}_{g,1},
\end{equation}
where $G_\mu:=\Aut(E^g, \mu)$ is the automorphism group of $(E^g,
\mu)$. Moreover, it induces an isomorphism on the $a$-number $1$ loci:
\begin{equation}\label{eq:a=1loci}
\mathrm{pr}_0: 
\coprod _{\mu \in P^*(E^g)}\mathcal{P'}_{g,\mu, \Fpbar}(a=1) /G_\mu  
\isoto \mathcal{S}_{g,1}(a=1).
\end{equation}
\item The supersingular locus $\calS_{g,1}$ is equidimensional of
dimension $\lfloor g^2/4\rfloor$. The $a$-number $1$ locus $\calS_{g,1}(a=1)$
is open and dense in $\calS_{g,1}$. It has 
\begin{equation}
  \label{eq:classnumber}
  \begin{cases}
    H_g(p,1), & \text{for odd integer $g$}; \\
    H_g(1,p), & \text{for even integer $g$}
  \end{cases}
\end{equation}
geometrically irreducible components. 
\end{enumerate}
\end{theorem}
\begin{proof}
  See \cite[Section 4]{lioort}.
\end{proof}

Note that $\calP_{3,\mu}'\subset \calP_{3,\mu}$ is dense, while for
general $g$ the open subscheme $\calP_{g,\mu}'\subset \calP_{g,\mu}$
is no longer dense, cf.~\cite[Section 9.6]{lioort}. 

\begin{definition}\
\begin{enumerate}
\item Let $k$ be an algebraically closed field of \ch $p>0$ and let
\[ \Mass: \calS_{g,1}(k)\to \Q, \quad x\mapsto
\Mass(x):=\Mass(\Lambda_x)  \]
be the mass function. 
For each mass value $r\in \Q$, i.e. $r=\Mass(x)$ for some point $x\in
\calS_{g,1}(k)$, define a subset
\begin{equation}
  \label{eq:massstratum}
  \calS_{g,1,r}:=\{x\in \calS_{g,1}(k): \Mass(x)=r \}. 
\end{equation}
Then we have a decomposition of the supersingular locus into subsets
\begin{equation}
  \label{eq:massstratification}
  \calS_{g,1}(k)=\coprod_{r} \calS_{g,1,r}, 
\end{equation}
where $r$ runs through all mass values. Each subset $\calS_{g,1,r}$ is
called \emph{the mass stratum with mass value $r$}, 
and the decomposition
\eqref{eq:massstratification} is called the 
\emph{mass stratification} of $\calS_{g,1}(k)$. 
\item For each $\mu\in P^*(E^g)$, consider the pull-back of the
mass function on $\calS_{g,1}(k)$ by $\mathrm{pr}_0$.
We obtain the mass function on
$\calP_{g,\mu}(k)$:
\[ \Mass: \calP_{g,\mu}(k)\to \Q, \quad y\mapsto
\Mass(y):=\Mass(\Lambda_{\mathrm{pr_0}(y)}).  \]
Similarly, we define the mass stratum $\calP_{g,\mu,r}$ for each mass
value $r\in \Q$ as in \eqref{eq:massstratum} and obtain a
decomposition of $\calP_{g,\mu}(k)$ into mass strata:
\begin{equation}
  \label{eq:massstratification_calP}
  \calP_{g,\mu}(k)=\coprod_{r} \calP_{g,\mu,r}, 
\end{equation}
called the \emph{mass stratification} of $\calP_{g,\mu}(k)$.
\end{enumerate}
\end{definition}

When $g=1$, the supersingular locus $\calS_{1,1}$ consists of one mass
stratum. 
When $g=2$, there are three mass strata: one stratum with $a$-number $2$ and
two strata with $a$-number $1$. 
Each mass stratum is a locally closed subset and the collection of 
mass strata satisfies the stratification property, 
namely, the closure of each stratum is the union of some strata 
cf.~\cite{yuyu}. When
$g=3$, we will see again from our computation 
that each mass stratum is a
locally closed subset on both $\calP_{3,\mu}$ and 
$\calS_{3,1}$. However,  
the collection of mass strata does not satisfy the stratification
property on $\calP_{3,\mu}$ (because the structure morphism $\pi: \mathcal{P}_{3,\mu} \to C$ constructed in Proposition~\ref{prop:explicitmoduli} admits a section $T$, which will be formally introduced in Definition~\ref{def:T}) 
but it does on its open dense subscheme 
$\calP'_{3,\mu}=\calP_{3,\mu}-T$. 
We expect that every mass stratum is a locally closed
subset for general $g$. The mass stratification encodes arithmetic information
(automorphism groups and endomorphism rings) of supersingular abelian
varieties. For example, we will see in Section~\ref{sec:Aut} that 
the automorphism groups of supersingular abelian threefolds 
jump only when the objects cross different mass strata. Since arithmetic
properties generally do not respect geometric properties, we are less optimistic that the collection of 
mass strata of $\calP'_{g,\mu}$ satisfies the stratification property.\\   

Now we introduce a local analogue of the mass stratification where
the underlying space $\calS_{g,1}$ is replaced with the moduli space of
supersingular $p$-divisible groups, namely, 
the supersingular Rapoport-Zink space. 

Fix a $g$-dimensional principally polarised superspecial abelian
variety $x_0=(X_0,\lambda_{X_0})$ over $\Fpbar$, and let
$\ul \bfX_0=
(\bfX_0,\lambda_{\bfX_0})=(X_0,\lambda_{X_0})[p^\infty]$ be the
associated principally polarised $p$-divisible group. 
Let $\calM_{\Fpbar}^0$ be the Rapoport-Zink space over $\Fpbar$ 
classifying principally polarised quasi-isogenies of $(\bfX_0,\lambda_{\bfX_0})$ of height $0$. 
For each $\Fpbar$-scheme $S$, $\calM^0_{\Fpbar}(S)$ is the set of
isomorphism classes of pairs $(\ul \bfX, \rho)_S$, where
\begin{itemize}
\item [(i)] $\ul \bfX=(\bfX,\lambda_{\bfX})$ is a principally polarised
  $p$-divisible group over $S$;
\item [(ii)] $\rho:\bfX_0 \to \bfX$ is a quasi-isogeny over $S$ such
  that $\rho^* \lambda_{\bfX}=\lambda_{\bfX_0}$.
\end{itemize}
Two pairs $(\ul \bfX_1, \rho_1)$ and $(\ul \bfX_2, \rho_2)$ 
are isomorphic if there exists an isomorphism $\alpha:\bfX_1\isoto
\bfX_2$ such that $\alpha\circ \rho_1=\rho_2$. One easily sees 
 $\alpha^* \lambda_{\bfX_2}=\lambda_{\bfX_1}$. The Rapoport-Zink space
 $\calM^0_{\Fpbar}$ is a scheme locally of finite type over $\Fpbar$,
   cf.~\cite[Theorem 3.25 and Corollary 2.29]{rapoport-zink}.

Let $G_{\ul \bfX_0}$ be the automorphism group scheme of $\ul \bfX_0$ 
over $\Zp$. The group $G_{\ul \bfX_0}(\Qp)$ of $\Qp$-valued points
consists of polarised quasi-self-isogenies of $\ul \bfX_0$ over $k$;
it is a locally compact topological group. 
Choose 
a Haar measure on $G_{\ul \bfX_0}(\Qp)$ with
volume one on the maximal open compact subgroup $G_{\ul
  \bfX_0}(\Zp)=\Aut(\ul \bfX_0)$. 
For each $k$-valued point $\bfx=(\ul
\bfX, \rho)\in \calM^0_{\Fpbar}(k)$, we may 
regard its automorphism group
$\Aut(\ul \bfX)$ as an open compact subgroup of $G_{\ul
  \bfX_0}(\Qp)$ by inclusion:
\[ \rho^*:\Aut(\ul \bfX)\embed G_{\ul \bfX_0}(\Qp), \quad h\mapsto 
   \rho^{-1}\circ h\circ \rho. \]

\begin{definition}
Let the notation be as above.
Define a function on $\calM^0_{\Fpbar}(k)$ by
\begin{equation}
  \label{eq:vfunction}
  v: \calM^0_{\Fpbar}(k) \to \Q, \quad \bfx=(\ul \bfX,\rho)\mapsto
  v(\bfx):=\mathrm{vol}(\rho^*(\Aut(\ul \bfX)))^{-1}. 
\end{equation}
For each $v$-value $r\in \Q$, that is, $r=v(\bfx)$ for some $\bfx\in
\calM^0_{\Fpbar}(k)$, consider the subset 
\[ \calM^0_r:=\{x\in \calM^0_{\Fpbar}(k): v(\bfx)=r\}, \]
for which the function $v$ takes value $r$, called \emph{the $v$-stratum
with $v$-value $r$}. The Rapoport-Zink space then decomposes in
subsets:
\[ \calM^0_{\Fpbar}(k)=\coprod_{r} \calM^0_r, \]
where $r$ runs through all $v$-values in $\Q$, called the
$v$-stratification of $\calM^0_{\Fpbar}(k)$. 
Observe that the collection of $v$-strata
is independent of the choice of the Haar measure on $G_{\ul
  \bfX_0}(\Qp)$ as the function $v'$ associated to a different Haar
measure is just a multiple of $v$ by a scalar.          
\end{definition}
 
Let 
\[ \wt \pi: \calM^0_{\Fpbar} \to \calS_{g,1} \]
be the Rapoport-Zink uniformisation morphism, cf.~\cite[6.13]{rapoport-zink}.
 
\begin{proposition}\label{massandv}
  The stratification of $\calM^0_{\Fpbar}(k)$ obtained 
  by the pull-back of the
  mass stratification of $\calS_{g,1}(k)$ by $\wt \pi$ 
  coincides with the $v$-stratification. 
\end{proposition}
\begin{proof}
  We compare the functions $v$ and $\wt \pi^* \Mass=\Mass \circ \wt
  \pi$. Let $\bfx=(\ul \bfX, \rho)$ be a $k$-valued point in
  $\calM^0_{\Fpbar}(k)$. Then $\bfx$ lifts to a pair 
  $((X,\lambda_X),\wt \rho)$ of a principally
  polarised supersingular abelian variety $(X,\lambda_X)$ and 
  a polarised
  quasi-isogeny $\wt \rho:(X_0,\lambda_{X_0})\to (X,\lambda_X)$. 
  By the construction of \cite[6.13]{rapoport-zink}, the
  map $\wt
  \pi$ sends $\bfx$ to $x:=(X,\lambda_X)$. Using Proposition \ref{prop:geomarithmass} 
and Lemma \ref{lem:massesU1U2}, we see that
\[
\begin{split}
\Mass(x)=\mathrm{Mass}(\Lambda_x) &=
\frac{\mathrm{vol}(G_{x_0}(\Zp))}{\mathrm{vol}(\wt \rho^*(G_{x}(\Zp)))} \mathrm{Mass}(\Lambda_{x_0}) \\
&=
\frac{\mathrm{Mass}(\Lambda_{x_0})}{\mathrm{vol}(\rho^*(\Aut(\bfX, 
\lambda_{\bfX})))}=\Mass(x_0)\cdot v(\bfx).
\end{split}
\]
Thus, $\wt \pi^* \Mass (\bfx)=\Mass(x_0) \cdot v(\bfx)$ for $\bfx\in
\calM^0_{\Fpbar}(k)$ and the assertion follows. 
\end{proof}

\subsection{The structure of $\calS_{3,1}$}\

Hereafter we will only treat the case where $g=3$. For brevity, we
write $\calP_{\mu}$ and $\calP'_{\mu}$ for 
$\calP_{3,\mu}$ and $\calP'_{3,\mu}$, respectively. 
Roughly speaking, Equation \eqref{eq:moduli} says that each
$\mathcal{P}_{\mu }$ approximates an irreducible component of the
supersingular locus $\mathcal{S}_{3,1}$. More precisely, one can show the
following structure results; for more details, we refer to
\cite[Sections 9.3-9.4]{lioort}. 
Let $C \subseteq \mathbb{P}^2$ be the Fermat curve  
defined by the equation $X_{1}^{p+1}+X_{2}^{p+1}+X_{3}^{p+1} = 0$. 

\begin{proposition}\label{prop:explicitmoduli}
The Fermat curve $C$ can be interpreted as the 
classifying space of isogenies $(Y_2, \lambda_2) \to (Y_1,\lambda_1)$ 
whose kernel is locally isomorphic to $\alpha_p^2$.
Moreover, there is an isomorphism $\mathcal{P}_{\mu} \simeq
 \mathbb{P}_{C}(\mathcal{O}(-1)\oplus \mathcal{O}(1))$ for which
the structure morphism $\pi : \mathbb{P}_{C}(\mathcal{O}(-1)\oplus
\mathcal{O}(1)) \to C$ corresponds to the forgetful map
$((Y_2,\lambda_2) \to (Y_1,\lambda_1) \to (Y_0, \lambda_0)) \mapsto
((Y_2,\lambda_2) \to (Y_1,\lambda_1))$.   
\end{proposition}

\begin{proof}
Let $M_2$ be the polarised contravariant Dieudonn{\'e} module of $Y_2$. 
Choosing an isogeny $\rho _2$ from ${E^3_k}$ 
such that $\ker (\rho_2) \simeq \alpha _p^2$ is equivalent to choosing
a surjection of Dieudonn{\'e} modules $M_2 \to k^2$. Since Frobenius
$\sfF$ and Verschiebung $\sfV$ act as zero on $k^2$, this is further
equivalent to choosing a one-dimensional subspace of the
three-dimensional (since $a(Y_2)=3$) $k$-vector space $M_2/(\sfF, \sfV)M_2$
which corresponds to a point $(t_1:t_2:t_3) \in \mathbb{P}^2 =
\mathbb{P}((M_2/(\sfF, \sfV)M_2)^{\ast})$.  

The polarisation $\lambda_2=p\mu$ descends to a polarisation $\lambda_1$ on
$Y_1$ through such $\rho_2$, and the condition $\ker (\lambda _1)
\subseteq Y_1[\sfF]$ is equivalent to the condition 
\begin{align*}
t_1^{p+1}+t_2^{p+1}+t_3^{p+1} = 0,
\end{align*}
which describes the Fermat curve $C$ of degree $p+1$ in $\mathbb{P}^2$. 
For precise computations, we refer to \cite{katsuraoort}. 

Let $M_1$ be the polarised Dieudonn{\'e} module of $Y_1$:
the polarisation $\lambda_1$ induces a quasi-polarisation
$D(\lambda_1) \colon M_1^{\vee} \to M_1$, and we regard $M_1^{\vee}$
as an submodule of $M_1$ under this injection. One has the inclusions
$M_1^\vee \subset \sfV M_2 \subset M_1$ as $\sfV M_2$ is self-dual with 
respect to the quasi-polarisation induced by $\lambda_1$ and 
$\sfV M_2=(\sfF,\sfV)M_2 \subset M_1$.  
Choosing a second isogeny $(Y_1,\lambda_1) \to (Y_0, \lambda_0)$ is equivalent to choosing a one-dimensional subspace of the two-dimensional vector space $M_1/M_1^{\vee}$. 
Thus each fibre of the structure morphism $\pi \colon 
\mathcal{P}_{\mu} \to C $ is isomorphic to
$\mathbb{P}((M_1/M_1^{\vee})^{\ast}) \simeq \mathbb{P}^1$ and this
fibration corresponds to a rank two vector bundle $\calV$ on $C$.  
The canonical one-dimensional space $(\sfF, \sfV)M_2/M_1^{\vee} \subseteq
M_1/M_1^{\vee}$ defines a section $s$ of 
$\pi \colon \mathcal{P}_{\mu} \to C $ and corresponds to a surjection 
$\mathcal{V} \to \mathcal{O}(-1)$. 
By the duality of polarisations, we see that $\mathcal{V}$ is an extension of $\mathcal{O}(-1)$ by $\mathcal{O}(1)$ and this extension splits. 
\end{proof}

Since the Fermat curve $C$ is a smooth plane curve of degree $p+1$,
its genus is
equal to $p(p-1)/2$.
Let $U_3(\F_p)\subseteq \GL_3(\F_{p^2})$ denote the unitary subgroup 
consisting of matrices $A$ such that $A^T A^{(p)}=\mathbb{I}_3$. 
We see that for each $A\in
U_3(\F_p)$ and $t\in C$, the  matrix multiplication 
$A\cdot t^T$ lies in $C$. This gives a left action of  $U_3(\F_p)$ on
the curve $C$. 
It is known that $\vert U_3(\F_p)\vert =p^3(p+1)(p^2-1)(p^3+1)$. 

A curve is $\mathbb{F}_{p^{2k}}$-maximal (resp.~minimal) if its Frobenius eigenvalues over $\mathbb{F}_{p^{2k}}$ all equal $-p^k$ (resp.~$p^k$).  From the well-understood behaviour of Frobenius eigenvalues under field extensions we then derive the following lemma.
 
\begin{lemma}\label{lem:Cmaxmim}
We have  $\vert C(\mathbb{F}_{p^2}) \vert = p^3 + 1$. 
Thus, it is $\F_{p^2}$-maximal
and hence $\F_{p^4}$-minimal. Moreover, 
we have $C(\mathbb{F}_{p^2}) = C(\mathbb{F}_{p^4})$. Furthermore, we
have
\begin{equation}
  \label{eq:Cpoints}
  \vert C(\F_{p^{2i}})\vert =
  \begin{cases}
    p^{2i}+p^{i+2}-p^{i+1}+1 & \text{if $i$ is odd;}\\
    p^{2i}-p^{i+2}+p^{i+1}+1 & \text{if $i$ is even.}\\   
  \end{cases}
\end{equation}
 \end{lemma}

\begin{proof} For each $t=(t_i)\in C(\F_{p^2})$,  
  let $s_i=t_i^{p+1}$. Then $s_i\in \Fp$ and $s_1+s_2+s_3=0$. So there
  are $p+1$ points $(s_i)$ in $\mathbb{P}^1(\Fp)$. 
  For each point $(s_i)$, there are $p+1$ (resp.~$(p+1)^2$) points
  $(t_i)$ over $(s_i)$
  if some of the $s_i$ are zero (resp.~otherwise);
  there are 3 points $(s_i)$ with $s_i=0$ for some $i$. 
  Thus, \[ \vert
  C(\mathbb{F}_{p^2}) \vert=(p+1-3)(p+1)^2+3(p+1)=p^3+1. \] 
  One checks that this means $C$ is $\F_{p^2}$-maximal.
  Hence, $C$ is $\F_{p^4}$-minimal and satisfies $\vert C(\mathbb{F}_{p^4}) \vert =p^3+1$. 
  Since $C$ is $\F_{p^{2i}}$-maximal (resp.~$\F_{p^{2i}}$-minimal) if $i$ is odd
  (resp.~even), the formula \eqref{eq:Cpoints} follows immediately.  
\end{proof}

\begin{lemma}\label{lem:CFp2}
Let $t = (t_1:t_2:t_3) \in C(k)$. Then $t \in C(\mathbb{F}_{p^2})$ if and only if $t_1, t_2, t_3$ are linearly dependent over $\mathbb{F}_{p^2}$.
\end{lemma}

\begin{proof}
See \cite[Lemma 2.1]{oort2}. Alternatively, we give the following
independent proof: 

The forward implication is immediate, so we will only show the reverse implication.
Assume $t_1, t_2, t_3$ are linearly dependent over $\mathbb{F}_{p^2}$. 
Then the vectors $(t_i,t_i^{p^2},t_i^{p^4})$ for $i=1,2,3$ are
$k$-linearly dependent. If $(t_i,t_i^{p^2},t_i^{p^4})$ for $i=2,3$ are
linearly independent, then there exist $a, b \in k$ such that $t_i =
at_i^{p^2} + bt_i^{p^4}$ for $i=1,2,3$. If they are linearly
dependent, then there exists $a'\in k$ such that $t_i^{p^2} =
a' t_i^{p^4}$ for $i=1,2,3$ and hence $t_i =
at_i^{p^2}$ with $a^{p^2}=a'$. 
Therefore, there exist $a, b \in k$ such that $t_i = at_i^{p^2} +
bt_i^{p^4}$ for $i=1,2,3$ in either case. 
Substituting this into the defining equation of $C$, we obtain
\[ a^{p+1}\sum_{i=1}^3  t_i^{p^2+p^3} + ab^p \sum_{i=1}^3
t_i^{p^2+p^5} + a^pb \sum_{i=1}^3 t_i^{p^3+p^4} + b^{p+1} \sum_{i=1}^3
t_i^{p^4+p^5} =0. \]
Again using the defining equation of $C$, we see that the first, third, and fourth terms vanish, so that also
$ab^p \sum_{i=1}^3 t_i^{p^2+p^5}= ab^p(\sum_{i=1}^3 t_i^{p^3+1})^{p^2} =0$.
If $a=0$ then the point $t = (t_1 : t_2: t_3)$ is defined over $\mathbb{F}_{p^4}$ and hence, by Lemma \ref{lem:Cmaxmim}, it is defined over $\mathbb{F}_{p^2}$. If $b=0$, then $t$ is defined over $\mathbb{F}_{p^2}$ as well.
So we may assume that $\sum_{i=1}^3 t_i^{p^3+1} = 0$.
Let $Z:=V(X_1^{p^3+1} + X_2^{p^3+1} + X_3^{p^3+1})$ be the Fermat curve of degree $p^3+1$. 
Then $t \in C \cap Z$.
The intersection number of $C$ and $Z$ is $(p+1)(p^3+1)$ and each point of $C(\mathbb{F}_{p^2})$ is in $C \cap Z$.
Since $\vert C(\mathbb{F}_{p^2}) \vert=p^3+1$ by Lemma \ref{lem:Cmaxmim}, it is enough to show that for each point $s \in C(\mathbb{F}_{p^2})$, the local multiplicity of $C$ and $Z$ at $s$ is $p+1$.
Since the unitary group $U_3(\mathbb{F}_p)$ acts transitively on $C(\mathbb{F}_{p^2})$, we may assume that $s = (\zeta : 0 :1)$ where $\zeta^{p+1}=-1$.
With local coordinates $v = X_1 - \zeta$ and $w = X_2$, the respective equations for $C$ and $Z$ at $y$ become $v^{p+1} + \zeta v^p + \zeta v + w^{p+1}$ and $v^{p^3+1} + \zeta v^{p^3} + \zeta^p v + w^{p^3+1}$. 
Now we may read off that the local multiplicity, i.e., the valuation of $v$ at $s$, is $p+1$, as required.
\end{proof}

We will denote $C^0 := C\setminus C(\mathbb{F}_{p^2})$.
Slightly abusively, we will tacitly switch between the notations $(t_1, t_2, t_3)$ and $(t_1:t_2:t_3)$.
For later use, we define the following:

\begin{definition}\label{def:Endt}
For $t = (t_1,t_2,t_3) \in k^3$ (viewed as a column vector), let 
\[
\mathrm{End}(t) = \{ A \in \mathrm{Mat}_3(\mathbb{F}_{p^2}) : A \cdot t \in k\cdot t \}.
\]
\end{definition}

\begin{lemma}\label{lem:Endt}
For any $t \in C^0(k)$, the $\mathbb{F}_{p^2}$-algebra $\mathrm{End}(t)$ is isomorphic to either $\mathbb{F}_{p^2}$ or $\mathbb{F}_{p^6}$.
\end{lemma}

\begin{proof}
For any $A \in \mathrm{End}(t)$, we have $A \cdot t = \alpha_A t$ for some $\alpha_A \in k$. The map
\[
\begin{split}
\mathrm{End}(t) & \to k \\
A & \mapsto \alpha_A
\end{split}
\]
is an $\mathbb{F}_{p^2}$-algebra homomorphism. It is injective, i.e., $A \cdot t = 0$ with $t = (t_1:t_2:t_3)$ implies that $A = 0$, since the $t_i$ are linearly independent over $\mathbb{F}_{p^2}$ by Lemma \ref{lem:CFp2}. 
Hence, $\mathrm{End}(t)$ is a finite field extension of $\mathbb{F}_{p^2}$. 
Since $\mathrm{End}(t) \subseteq \Mat_3(\mathbb{F}_{p^2}) = \mathrm{End}((\mathbb{F}_{p^2})^3)$, we may regard $(\mathbb{F}_{p^2})^3$ as a vector space over $\mathrm{End}(t)$. It follows that $[\mathrm{End}(t) : \mathbb{F}_{p^2}] \mid 3$, as required.
\end{proof}

\begin{lemma}\label{lem:CM}
We have 
\begin{equation}
  \label{eq:CM}
CM:=\{ t \in C^0(k) : \mathrm{End}(t) \simeq \mathbb{F}_{p^6} \} =
C^0(\mathbb{F}_{p^6}).   
\end{equation}
\end{lemma}
\begin{proof}
The containment $\{ t \in C^0(k) : \mathrm{End}(t) \simeq
\mathbb{F}_{p^6} \} \subseteq C^0(\mathbb{F}_{p^6})$ is immediate,
because $t$ is an eigenvector of a matrix in $\Mat_3(\F_{p^2})$ and
can be solved over the ground field $\F_{p^6}$. 
We will now prove the reverse containment.

For each $t \in C^0(\mathbb{F}_{p^6})$, we construct for each
element $\alpha\in \F_{p^6}$ a
matrix $A \in \mathrm{Mat}_3(\mathbb{F}_{p^6})$ as follows
\[
A=A_\alpha:= (t, t^{(p^2)}, t^{(p^4)})\cdot {\rm diag}(\alpha
,\alpha^{p^2}, \alpha^{p^4}  ) \cdot (t, t^{(p^2)}, t^{(p^4)})^{-1}.
\]
Since the $t_i$ are linearly independent over $\mathbb{F}_{p^2}$
by Lemma \ref{lem:CFp2}, 
the matrix $(t, t^{(p^2)}, t^{(p^4)})$ is invertible.
We check that
\begin{equation*}
  \label{eq:A_rational}
\begin{split}
  A^{(p^2)}&=(t^{(p^2)}, t^{(p^4)}, t)\cdot {\rm diag}(\alpha^{p^2}
  ,\alpha^{p^4},\alpha   ) \cdot (t^{(p^2)}, t^{(p^4)}, t)^{-1} \\
&= (t, t^{(p^2)}, t^{(p^4)})\cdot 
\begin{pmatrix}
  0 & 0 & 1\\1 & 0 & 0\\0 & 1 & 0
\end{pmatrix}
 \cdot {\rm diag}(\alpha^{p^2}
  ,\alpha^{p^4},\alpha) \cdot \begin{pmatrix}
  0 & 1 & 0\\0 & 0 & 1\\1 & 0 & 0
\end{pmatrix}\cdot (t, t^{(p^2)}, t^{(p^4)})^{-1} \\
  &=A,
\end{split}  
\end{equation*}
and hence $A\in \mathrm{Mat}_3(\mathbb{F}_{p^2})$. We also have that
$A_\alpha \cdot t = \alpha t$. Thus, the map $\alpha\in \F_{p^6}
\mapsto A_\alpha$ gives an isomorphism $\F_{p^6}\simeq \End(t)$, as
required. 
\end{proof}

\begin{remark} \
\begin{enumerate}
\item We can also show that $U_3(\Fp)$ acts transitively on $C^0(\F_{p^6}) = CM$. 
The action on $C(\F_{p^2})$ is also transitive, with stabilisers of size
  $p^3(p+1)(p^2-1)$; this gives another proof of the result $\vert
  C(\F_{p^2}) \vert =p^3+1$.
\item The proof of Lemma~\ref{lem:Endt} proves the following more
  general result. Let $F$ be any field contained in a field $K$ 
  and $t_1, t_2,\dots, t_n$ be a set of $F$-linearly independent
  elements in $K$. Put $t=(t_1,\dots, t_n)^T$ and $\End(t):=\{A\in
  \Mat_n(F) : A\cdot t\subseteq K\!\cdot t\}$. 
  Then $\End(t)$ is a finite field
  extension of $F$ of degree dividing $n$. 

  Furthermore, suppose that $t_1,\dots, t_n$ are contained a degree
  $n$ subextension $E$ of 
  $F$ in $K$. Then the $F$-basis $t_1,\dots, t_n$ of $E$ determines an
  $F$-algebra embedding $r: E \to \Mat_n(F)$ which is characterised
  by $r(a)\cdot t=at$ for every $a\in
  E$. Thus, $E\simeq \End(t)$ and $t$ is an eigenvector of a
  matrix in $\Mat_n(F)$. This is an abstract way of doing what is done
  explicitly in the the second part of the proof of
  Lemma~\ref{lem:CM}.                          
\end{enumerate}
\end{remark}

\begin{definition}\label{def:T}
The morphism $\pi:\calP_\mu\to C$ admits a section $s$ defined as follows. 
For a base scheme $S$, let $\rho_2: (Y_2,p\mu)\to (Y_1,\lambda_1)$ be an object in $C(S)$. Put $(Y_2^{(p)},\mu^{(p)}):=(Y,\mu)\times_{S,F_S}
S$, where $F_S:S\to S$ is the absolute Frobenius map. The relative
Frobenius morphism $\sfF:Y_2 \to Y_2^{(p)}$ gives rise to a morphism
of polarised abelian schemes $\sfF: (Y_2,p \mu)\to
(Y_2^{(p)},\mu^{(p)})$. Since $\ker(\rho_2)\subseteq \ker(\sfF)$, the
morphism factors through an isogeny $\rho_1:Y_1 \to Y_2^{(p)}$. As
$\rho_2^* \rho_1^* \mu^{(p)}=\sfF^* \mu^{(p)}=p\mu=\rho_2^* \lambda_1$,
we see that $\rho_1^* \mu^{(p)}=\lambda_1$ and thus obtain a polarised flag type quotient 
\[ 
\begin{CD}
  (Y_2,p\mu) @>\rho_2>> (Y_1,\lambda_1) @>{\rho_1}>> (Y_2^{(p)},\mu^{(p)}).
\end{CD}
\]    
This defines the section $s$, whose image will be denoted by $T$.   
\end{definition}

Recall the definition of the $a$-number from Definition~\ref{def:anumber}.
For an abelian threefold $X$ over~$k$, we have $a(X) \in \{1,2,3\}$.

\begin{proposition}\label{prop:sections}
Let the notation be as above.
\begin{enumerate}
\item We have $\calP_\mu'=\calP_\mu-T$. 
\item If $x \in T$ then we have $a(x) =3$.
\item For any $t \in C(k)$, we have $t \in C(\mathbb{F}_{p^2})$ if and only if $a(x) \geq 2$ for any $x \in \pi ^{-1}(t)$.
\item For any $x \in \calP_\mu(k)$, 
we have $a(x) = 1$ if and only if $x \notin T$ and $\pi (x) \notin C(\mathbb{F}_{p^2})$.
\end{enumerate}
\end{proposition}
\begin{proof}
  See \cite[Section 9.4]{lioort}.
\end{proof}

\subsection{Minimal isogenies}\label{ssec:miniso}\

Given a polarised flag type quotient $Y_2 = E_k^3
\xrightarrow{\rho_2} Y_1 \xrightarrow{\rho_1} Y_0 = X$, 
the composite map $\rho_1\circ \rho_2: (Y_2,\lambda_2) \to
(Y_0,\lambda_0) = (X, \lambda)$ 
is an isogeny from a superspecial abelian variety
$Y_2$. Thus, this isogeny factors through the minimal isogeny of
$(X,\lambda)$:
\[ (Y_2,\lambda_2) \xrightarrow{\rho_1\circ \rho_2}
(\tilde{X},\tilde{\lambda})  \xrightarrow{\varphi} 
(X, \lambda). \]
Since every member $(X,\lambda)\in \mathcal{S}_{3,1}(k)$ can be constructed from
a polarised flag type quotient $(Y_\bullet,\rho_\bullet)$, 
we can construct the minimal isogeny of $(X,\lambda)$ from
$(Y_\bullet,\rho_\bullet)$. 

To describe the minimal isogenies for supersingular abelian threefolds in more detail, in the following proposition we separate into three cases, based on the $a$-number of the threefold.

\begin{proposition}\label{prop:miniso} 
Let $(X,\lambda)$ be a supersingular principally polarised abelian
threefold over~$k$. Suppose that $(X,\lambda)$ lies in the image of $\calP_\mu'$ under the
map $\calP'_\mu\to \mathcal{S}_{3,1}$ for some $\mu\in P(E^3)$, so that there
is a unique PFTQ over $(X,\lambda)$. 
\begin{enumerate}
\item If $a(X) = 1$, then the associated 
polarised flag type quotient $(Y_2,\lambda_2) 
\xrightarrow{\rho_2} (Y_1, \lambda_1) \xrightarrow{\rho_1} (Y_0,
\lambda_0) = (X, \lambda)$ gives the minimal isogeny $\varphi :=
\rho_1 \circ \rho_2$ of degree~$p^3$. 
\item If $a(X) = 2$, then in the associated polarised flag type
  quotient $Y_2 = E_k^3 \to Y_1 \to Y_0 = X$ we have $a(Y_1)
  = 3$, so $Y_1$ is superspecial. Thus, the minimal isogeny is
  $\rho_1: (Y_1, \lambda_1) \to (X, \lambda)$ of degree~$p$, where
  $\rho_1^* \lambda = \lambda_1$ satisfies $\ker(\lambda_1) \simeq
  \alpha_p \times \alpha_p$. 
\item If $a(X) = 3$, then $X$ is superspecial. Thus, $X$ is $k$-isomorphic to $E_k^3$ and the minimal isogeny is the identity map. 
\end{enumerate}
\end{proposition}

\begin{proof}
\begin{enumerate}
\item Let $M_2, M_1, M_0$ denote the Dieudonn{\'e} modules of $Y_2,
  Y_1, Y_0 = X$, respectively. Then $a(M_2) = 3$.  
Suppose that $a(M_0) = 1$. By Proposition \ref{prop:sections}, this
corresponds to a point $t = (t_1:t_2:t_3) \not\in
C(\mathbb{F}_{p^2})$.  
We claim that $a(M_1) = 2$, which implies the statement. 
The Dieudonn{\'e} modules satisfy the following inclusions:
\begin{align*}
\begin{matrix}
M_2 &\supseteq &M_1 &\supseteq &M_0 &&&\\
&\rotatebox{-30}{$\supseteq$} &\rotatebox{90}{$\subseteq$} & &\rotatebox{90}{$\subseteq$}   &\rotatebox{-30}{$\supseteq$} &&\\
&&(\sfF, \sfV)M_2 &\supset &(\sfF, \sfV)M_1&= &(\sfF, \sfV)M_0&\\
&&&\rotatebox{-30}{$\supseteq$} &\rotatebox{90}{$\subseteq$} &\rotatebox{-30}{$\supseteq$} &\rotatebox{90}{$\subseteq$} &\rotatebox{-30}{$\supseteq$}\\
&&&&(\sfF,\sfV)^2M_2 &= &(\sfF,\sfV)^2M_1 &= &(\sfF,\sfV)^2M_0.
\end{matrix}
\end{align*}
All inclusions follow from the construction of flag type quotients.
For the equalities, we note the following:
Since $M_2$ is superspecial of genus three, we have $(\sfF, \sfV)M_2 =
\sfF M_2, (\sfF, \sfV)^2M_2=pM_2$, and 
\[
\dim (M_2/\sfF M_2) = \dim (\sfF M_2/pM_2) = 3.
\]
It follows from the definition of flag type quotients that
$\dim(M_1/\sfF M_2) = 1$, so $M_1/\sfF M_2$ is generated by one element,
namely the image of $t$ (abusively again denoted $t$). 
So $(\sfF, \sfV) M_1/pM_2$ is two-dimensional and generated by the two
elements $\sfF t$ and $\sfV t$, which are $k$-linearly
independent since $t \not\in C(\mathbb{F}_{p^2})$, by Lemma
\ref{lem:CFp2}.  
Using this, we see that 
\[
\dim(\sfF M_2/(\sfF, \sfV)M_1) = \dim(\sfF M_2/pM_2) - \dim((\sfF, \sfV)M_1/pM_2) = 1
\]
and $a(M_1) = \dim (M_1/(\sfF, \sfV)M_1) = 2$, as claimed.
It follows from $\dim (M_1/M_0)=1$ and $a(M_1)=2$ that $\dim (M_0/(\sfF,\sfV)M_1)=1$.
As we have assumed that $a(M_0) =\dim(M_0/(\sfF,\sfV)M_0)=1$, 
the latter implies the equality
$(\sfF, \sfV)M_1= (\sfF, \sfV)M_0$. 
Since $\dim (M_0/(\sfF,\sfV)M_1)=1$ and $\dim(M_0/pM_2)=3$, one has
$\dim (\sfF,\sfV)M_1/pM_2)=2$. 
Since $t_1,t_2,t_3$ are $\F_{p^2}$-linearly independent by 
Lemma~\ref{lem:CFp2}, the vectors $\sfF^2 t, p t$ and $\sfV^2 t$ in
$\sfF M_2/p\sfF M_2$ span a $3$-dimensional subspace and hence 
$\dim((\sfF,\sfV)^2M_1/p\sfF M_2)=3$. This shows the equality
$pM_2 = (\sfF, \sfV)^2M_1 = (\sfF,\sfV)^2M_0$.

Now put $\Phi:=1+\sfF \sfV^{-1}$. We have shown that $\sfV \Phi
M_0=(\sfF,\sfV)M_1$ is not superspecial and that $\Phi^2 M_0=M_2$ is
superspecial. Therefore, $M_2$ is the smallest 
superspecial \dieu module containing $M_0$.
This proves that $\rho_1\circ \rho_2: Y_2\to X$ is the minimal isogeny.
\item When $a(M_0)=2$, this corresponds to a point $t = (t_1:t_2:t_3)
  \in C(\mathbb{F}_{p^2})$. Using the notation from the previous item,
  we still have that $(\sfF,\sfV)M_1/pM_2$ is generated by $\sfF t$
  and $\sfV t$, but since the $t_i$ are $\mathbb{F}_{p^2}$-linearly
  dependent, we have $\dim((\sfF, \sfV)M_1/pM_2)=1$, so $a(M_1) = 3$.
  Since $\ker(\lambda_1) \subseteq Y_1[F]\simeq \alpha_p^3$, we have
  $\ker(\lambda_1) \simeq \alpha_p^2$, as claimed. 
\item The fact that $a(X) = 3$ if and only if $X$ is superspecial is
  due to Oort, \cite[Theorem 2]{oort}. 
\end{enumerate}
\end{proof}

\begin{remark}\label{countexample:miniso}
The proof of \cite[Lemma 1.8]{lioort} uses the claim: If $X$ is
a $g$-dimensional supersingular abelian variety with $a(X)<g$, 
and $X':=X/A(X)$,
where $A(X)$ is the maximal $\alpha$-subgroup of $X$, then
$a(X')>a(X)$. 

Now take $Y_1$ the abelian threefold as in
Proposition~\ref{prop:miniso}(1). 
We have computed $a(Y_1)=2$ and 
\[ 
\begin{split}
a(Y_1/A(Y_1))&=a((\sfF,\sfV)M_1)=\dim\,
(\sfF,\sfV)M_0/(\sfF,\sfV)^2M_1 \\
&=\dim\, M_0/(\sfF,\sfV)^2M_1-\dim\,
M_0/(\sfF,\sfV)M_0 =2.  
\end{split}
 \]
This gives a counterexample to the claim.   
\end{remark}

\section{The case $a(X) \ge 2$}\label{sec:a2}

Let $x=(X,\lambda)\in \mathcal{S}_{3,1}(k)$ with $a(X)=2$ and let $y\in \calP_\mu\simeq
\bbP^1_C(\calO(-1)\oplus \calO(1))$ be the point corresponding to the
PFTQ over it:
\[ (Y_2,\lambda_2)\xrightarrow{\rho_2} (Y_1,\lambda_1) 
\xrightarrow{\rho_1} (Y_0,\lambda_0)=(X,\lambda). \]
By Propositions~\ref{prop:sections} and \ref{prop:miniso}, 
$(Y_1,\lambda_1)$ corresponds to a point 
$t=(t_1,t_2,t_3)\in C(\F_{p^2})$ and $u \in
\bbP^1_t(k):=\pi^{-1}(t)$. Moreover, $\rho_1:(Y_1,\lambda_1)\to
(X,\lambda)$ is the minimal isogeny. Put $x_1~=~(Y_1,\lambda_1)$. Then
$\Lambda_{x_1}=\Lambda_{3,p}$ and by Corollary~\ref{cor:sspmassg3} 
and Proposition~\ref{prop:compmass} we have
\begin{equation}
  \label{eq:a2formula}
  \Mass(\Lambda_x)=\frac{(p-1)(p^3+1)(p^3-1)}{2^{10}\cdot 3^4 \cdot
  5\cdot 7}\cdot [\Aut(M_1,\<\, \>): \Aut(M,\<\,,\>)],  
\end{equation}
where $(M,\<\,,\>)\subseteq (M_1,\<\,, \>)$ are the quasi-polarised \dieu
modules associated to $(Y_1,\lambda_1)\to (X,\lambda)$. 

Let $M_1^{\vee}$ denote the dual lattice of $M_1$ with respect to
$\<\,,\>$. Then one has $M_1^{\vee}\subseteq M\subseteq M_1$ and 
$M/M_1^{\vee} \in \bbP(M_1/M_1^{\vee})=\bbP^1_t(k)$ is a one-dimensional 
$k$-subspace in $M_1/M_1^{\vee}$. Since the morphism $\rho_2$ is defined over
$\F_{p^2}$, the threefold $Y_1$ is endowed with the $\F_{p^2}$-structure $Y_1'$
with Frobenius $\pi_{Y_1'}=-p$. The induced $\F_{p^2}$-structure on
$\bbP^1_t$ is defined by the $\F_{p^2}$-vector space
$V_0:=M_1^\diamond/M^{{t,\diamond}}_1$, where $M_1^\diamond:=\{m\in M_1 :
\sfF m+\sfV m=0\}$ is the skeleton of $M_1$, cf. \cite[Section 5.7]{lioort}.  

Since $\ker(\lambda_1) \simeq
\alpha_p\times \alpha_p$, the quasi-polarised superspecial \dieu
module $(M_1,\<\,,\>)$ decomposes into a product of
a two-dimensional indecomposable superspecial \dieu module  and 
a one-dimensional such module. By \cite[Proposition 6.1]{lioort}, there is a
$W$-basis $e_1$, $e_2$, $e_3$, $f_1$, $f_2$, $f_3$ for $M_1$ such that $\sfF e_i=-\sfV
e_i=f_i$, $\sfF f_i=-\sfV f_i=-p e_i$. for $i=1,2,3$,  
\[ \<e_1,e_2\>=p^{-1}, \quad \<f_1,f_2\>=1, \quad \<e_3,f_3\>=1, \]  
and other pairings are zero. 
Then $M_1^{\vee}$ is spanned by $p e_1, p_2, e_3, f_1,f_2,f_3$ and
$M_1/M_1^{\vee}=\Span_k\{e_1,e_2\}$. Let $u=(u_1:u_2)\in \bbP^1_t(k)$ be
the projective coordinates of the point corresponding to $M/M_1^{\vee}$. That is,
$M/M_1^{\vee}$ is the one-dimensional subspace spanned by $u=u_1 \bar e_1+u_2
\bar e_2$, where $\bar e_i$ denotes the image of $e_i$ in $M_1/M_1^{\vee}$. 

If $u\in\ \bbP^1_t(\F_{p^2})$, then $a(M)=3$ and $\Mass(\Lambda_x)$ is
already computed in Corollary~\ref{cor:sspmassg3}. 
Suppose then that $u\not\in \bbP^1_t(\F_{p^2})$. 
In this case, $M_1$ (resp.~$M_1^{\vee}$) is the smallest (resp.~maximal) 
superspecial \dieu module containing (resp.~contained in) $M$. Thus, 
\[ \End(M)=\{g\in \End(M_1): g(M_1^{\vee})\subseteq M_1^{\vee},\  g(M) \subseteq
M\}. \] 
Consider the reduction map 
\[ m: \End(M_1)=\End(M_1^\diamond) \onto
\End(M_1^\diamond/M_1^{t,\diamond})=\End_{\F_{p^2}}(V_0)=\Mat_2(\F_{p^2}).
\]
It is clear that $\End(M)$ contains $\ker(m)$ and that $m$ induces a
surjective map 
\[ m: \End(M) \onto m(\End(M))=\{g\in \Mat_2(\F_{p^2}): g \cdot
u\subseteq k \cdot u\}. \]
Write $\End(u):=\{ g\in
  \Mat_2(\F_{p^2}): g \cdot u\subseteq k \cdot u\}$.
  
\begin{lemma}\label{lem:endu}
\begin{enumerate}
\item 
If $u\in \bbP^1_t(\F_{p^4})-\bbP^1_t(\F_{p^2})$, then $\End(u)\subseteq
    \Mat_2(\F_{p^2})$ is an
    $\F_{p^2}$-subalgebra which is isomorphic to $\F_{p^4}$. 
\item
If $u\in \bbP^1_t(k)-\bbP^1_t(\F_{p^4})$, then $\End(u)=\F_{p^2}$. 
\end{enumerate}
\end{lemma}
\begin{proof}
  This is a simpler version of Lemmas~\ref{lem:Endt} 
  and \ref{lem:CM} so we omit the proof; cf. also \cite[Section 3]{yuyu}.
\end{proof}

Put $\<\,,\>_1:=p\<\,, \>$. Then $\<\,,\>_1$ induces a non-degenerate
alternating pairing, again denoted $\<\,,\>_1:V_0\times V_0\to \F_{p^2}$. The reduction map $m$
then gives rise to the following map 
\begin{equation}
  \label{eq:redmapautM1}
  m: \Aut(M_1,\<\,,\>)=\Aut(M_1,\<\,,\>_1) \to \Aut(V_0,\<\,,\>_1)\simeq
  \SL_2(\F_{p^2}). 
\end{equation}

\begin{lemma}\label{lem:m}
  The map $m: \Aut(M_1,\<\,,\>)\to \Aut(V_0,\<\,,\>_1)$ is surjective.
\end{lemma}

\begin{proof}
  Since $Y_1$ is supersingular, we have that $\End(Y_1)\otimes \Zp\simeq
  \End(M_1)$ and that $G_{x_1}(\Zp)\simeq \Aut(M_1,\<\,, \>)$; 
  recall the notation from \eqref{eq:aut}. 
  The group scheme $G_{x_1}\otimes \Zp$ is a parahoric group scheme
  and in particular is smooth over $\Zp$. Thus, the map 
   $G_{x_1}(\Zp)\to G_{x_1}(\Fp)$ is surjective. Now
  $\Aut(V_0,\<\,,\>_1)=\Res_{\F_{p^2}/\Fp} \SL_2$ 
  viewed as an algebraic group over $\Fp$ is a
  reductive quotient of the special fibre $G_{x_1}\otimes \Fp$. 
  Therefore, the map  $G_{x_1}(\Fp)\to
  \Aut(V_0,\<\,,\>_1)=\SL_2(\F_{p^2})$ is also surjective. This proves
  the lemma.  
\end{proof}

We now prove the main result of this section.

\begin{theorem}\label{thm:massa2}
  Let $x=(X,\lambda)\in \mathcal{S}_{3,1}(k)$ with $a(X)\ge 2$ and let $y\in
  \calP'_\mu(k)$ be a lift of $x$ 
  for some $\mu\in P(E^3)$. Write $y=(t,u)$
  where $t=\pi(y)\in C(\F_{p^2})$ and $u\in
  \pi^{-1}(t)=\bbP^1_t(k)$. Then 
\begin{equation}
  \label{eq:massa2}
  \mathrm{Mass}(\Lambda_x)=\frac{L_p}{2^{10}\cdot 3^4\cdot 5\cdot 7}, 
\end{equation}
where 
\begin{equation}
  \label{eq:massa2_1}
 L_p=
\begin{cases}
  (p-1)(p^2+1)(p^3-1) & \text{if } u\in
  \mathbb{P}^1_t(\mathbb{F}_{p^2}); \\ 
  (p-1)(p^3+1)(p^3-1)(p^4-p^2) & \text{if }
  u\in\mathbb{P}^1_t(\mathbb{F}_{p^4})\setminus
  \mathbb{P}^1_t(\mathbb{F}_{p^2}); \\ 
  2^{-e(p)}(p-1)(p^3+1)(p^3-1) p^2(p^4-1) & \text{ if
  } u \not\in 
  \mathbb{P}^1_t(\mathbb{F}_{p^4});
\end{cases} 
\end{equation}
where $e(p)=0$ if $p=2$ and $e(p)=1$ if $p>2$. 
\end{theorem}
\begin{proof}
  By Lemma~\ref{lem:m}, 
\[ [\Aut(M_1,\<\ , \>): \Aut(M,\<\,,\>)]=[\SL_2(\F_{p^2}):
\SL_2(\F_{p^2})\cap \End(u)^\times]. \]
By Lemma~\ref{lem:endu}, 
\[ \SL_2(\F_{p^2})\cap \End(u)^\times = 
\begin{cases}
   \F_{p^4}^1 & \text{if $u\in\bbP^1_t(\mathbb{F}_{p^4})\setminus
   \mathbb{P}^1_t(\mathbb{F}_{p^2})$;}\\
   \{\pm 1\} & \text{if $u\not \in\mathbb{P}^1_t(\mathbb{F}_{p^4})$.}
\end{cases} \]
It follows that 
\[ [\Aut(M_1,\<\, \>): \Aut(M,\<\,,\>)]=
\begin{cases}
   p^2(p^2-1) & \text{if $u\in\bbP^1_t(\mathbb{F}_{p^4})\setminus
   \mathbb{P}^1_t(\mathbb{F}_{p^2})$;}\\
   \vert \PSL_2(\F_{p^2}) \vert & \text{if $u\not
   \in\mathbb{P}^1_t(\mathbb {F}_{p^4})$,}
\end{cases} \]
so the theorem follows from \eqref{eq:a2formula}. 
\end{proof}

\section{The case $a(X) = 1$}\label{sec:a1}

Suppose that $(X,\lambda)$ is a supersingular principally polarised
abelian threefold over $k$ with $a(X)=1$.  
By Proposition~\ref{prop:miniso}(1), there is a minimal isogeny
$\varphi: (Y_2,\mu) \to (X,\lambda)$, where $Y_2 = E_k^3$,
and where $\varphi^*\lambda = p\mu$ for $\mu \in P(E^3)$ a principal
polarisation. 
In this section we will compute the local index
\begin{equation}\label{eq:indexaut}
[\mathrm{Aut}((Y_2,\mu)[p^{\infty}]): \mathrm{Aut}((X,\lambda)[p^{\infty}])].
\end{equation}
Let $M$ and $M_2$ be the 
Dieudonn{\'e} modules of $X$ and $Y_2$,
respectively. Together with the induced (quasi-)polarisations, 
we have $(M, \langle , \rangle)$ and $(M_2,\langle , \rangle_2)$, where $\langle , \rangle_2 =
p\langle , \rangle$ is again a principal polarisation. 
(Note that $(M_2, \langle , \rangle_2)$ is the quasi-polarised Dieudonn{\'e} module associated to $(Y_2,\mu)$ and not to $(Y_2, p\mu)$, and that $pM_2 \subseteq M$ by the proof of Proposition \ref{prop:miniso}(1).)
The proof of Proposition \ref{prop:miniso}(1) also shows that every
automorphism of $M$ can be lifted to an automorphism of $M_2$, i.e.,
that $\mathrm{Aut}((M,\langle , \rangle)) \subseteq
\mathrm{Aut}((M_2,\langle, \rangle_2))$. 
Then equivalently to \eqref{eq:indexaut}, cf. Proposition
\ref{prop:compmass}, we will compute  
\begin{equation}\label{eq:indexsimple}
[\mathrm{Aut}((M_2,\langle, \rangle_2)):\mathrm{Aut}((M,\langle , \rangle))].
\end{equation}

\subsection{Determining $\boldsymbol{\mathrm{Aut}((M_2,\langle,
    \rangle_2))}$}\ 

Let $W = W(k)$ denote the ring of Witt vectors over $k$.
Choose a $W$-basis $e_1, e_2, e_3, f_1, f_2, f_3$ for $M_2$ such that 
\begin{equation}
  \label{eq:basis}
\sfF e_i =-\sfV e_i=f_i, \quad \sfF f_i =-\sfV f_i= -pe_i, \quad
\langle e_i, f_j \rangle_2 = \delta_{ij}, \quad \langle e_i,e_j
\rangle_2 = \langle f_i, f_j \rangle_2 = 0,   
\end{equation}
for all $i,j \in \{1,2,3\}$.

Let $D_p$ be the division quaternion algebra over $\mathbb{Q}_p$ and let $\mathcal{O}_{D_p}$ denote its maximal order. 
We also write $D_p = \mathbb{Q}_{p^2}[\Pi]$ and $\mathcal{O}_{D_p} = \mathbb{Z}_{p^2}[\Pi]$, where $\mathbb{Z}_{p^2} = W(\mathbb{F}_{p^2})$ and $\mathbb{Q}_{p^2} = \mathrm{Frac}W(\mathbb{F}_{p^2})$, and where $\Pi^2 = -p$ and $\Pi a = \overline{a} \Pi$ for any $a \in \mathbb{Q}_{p^2}$. 
Here $a \mapsto \overline{a}$ denotes the non-trivial automorphism of $\mathbb{Q}_{p^2}/\mathbb{Q}_p$. 
If we let $*$ denote the canonical involution of $D_p$, then $a^* = \overline{a}$ for any $a \in \mathbb{Q}_{p^2}$, and $\Pi^* = -\Pi$. 

\begin{lemma}\label{lem:endMtilde}
We have $\mathrm{End}(M_2) \simeq \mathrm{Mat}_3(\mathcal{O}_{D_p})$ and hence
$\mathrm{Aut}(M_2) \simeq \mathrm{GL}_3(\mathcal{O}_{D_p})$ (not taking the polarisation into account).
\end{lemma}

\begin{proof}
We have $\mathrm{End}(M_2) = \mathrm{End}_{\mathcal{O}_{D_p}}(M_2^{\Diamond})$, where $M_2^{\Diamond} := \{ m \in M_2 : \sfF m + \sfV m = 0\}$ denotes the skeleton of $M_2$; this is an $\mathcal{O}_{D_p}$-module where $\Pi$ acts by $\sfF$ and $\Pi^*$ acts by $\sfV$. 
Now the result follows by using the basis $e_1,e_2,e_3$ for
$\mathrm{Mat}_3(\mathcal{O}_{D_p})^{\mathrm {op}}$ (the opposite algebra);
we choose a convention where the matrices act on the left. We fix the
isomorphism $\mathrm{Mat}_3(\mathcal{O}_{D_p})^{\mathrm {op}}\simeq
\Mat_3(\calO_D)$ by sending $A$ to $A^*$. 
\end{proof}

We fix the identification $\End(M_2)=\Mat_3(\calO_D)$ by the
isomorphism chosen in Lemma~\ref{lem:endMtilde} with respect to the
basis in \eqref{eq:basis}.

\begin{lemma}\label{lem:autMtilde}
We have $\mathrm{Aut}(M_2,\langle, \rangle_2) \simeq \{ A \in
\mathrm{GL}_3(\mathcal{O}_{D_p}) : A^*A \simeq \mathbb{I}_3 \}$. 
\end{lemma}
\begin{proof}
It suffices to check that $\langle A \cdot e_i, e_j\rangle_2 = \langle
e_i, A^* \cdot e_j \rangle_2$ for any $A \in
\mathrm{Mat}_3(\mathcal{O}_{D_p})$ and any $i,j \in \{1,2,3\}$.  
Write $A = (a_{ij})$ and $A^* = (a'_{ij})$ with $a_{ij} = c_{ij} +
d_{ij}\Pi$ for $c_{ij},d_{ij} \in \mathbb{Z}_{p^2}$, and with $a'_{ij}
= a^*_{ji}$.  
Then
\[
\langle A \cdot e_i, e_j\rangle_2 = \langle \sum_k a_{ik} e_k, e_j\rangle_2 = \langle d_{ij} f_j , e_j\rangle_2 = -d_{ij}
\]
coincides with
\[
\langle e_i, A^* \cdot e_j \rangle_2 = \langle e_i, \sum_k a'_{jk}e_k \rangle_2 = \langle e_i, a'_{ji} e_i\rangle_2 = \langle e_i, \overline{c}_{ij}e_i - d_{ij}f_i\rangle_2 = -d_{ij},
\]
as required.
\end{proof}

\subsection{Endomorphisms and automorphisms modulo $\boldsymbol{pM_2}$}\

As was pointed out earlier, the proof of Proposition \ref{prop:miniso}(1) contains the important observation that $pM_2 \subseteq M$. 
This allows us to consider the endomorphisms and automorphisms of both $M_2$ and $M$ modulo $p$ (i.e., reducing modulo $pM_2$) and modulo $\Pi$.  In Definitions~\ref{def:AutEndM2mod} and~\ref{def:AutEndMmod}, we first define, and introduce notation for, all the endomorphism rings and automorphism groups we are considering.

\begin{definition}\label{def:AutEndM2mod}
Let $m_p$ denote the reduction-modulo-$p$ map and $m_{\Pi}$ the reduction-modulo-$\Pi$ map. 
By Lemma~\ref{lem:endMtilde}, for $M_2$ we have 
\begin{equation}\label{eq:redM2}
\mathrm{End}(M_2) \simeq \mathrm{Mat}_3(\mathcal{O}_{D_p}) \xrightarrow{m_p} \mathrm{Mat}_3(\mathbb{F}_{p^2}[\Pi]) \xrightarrow{m_{\Pi}} \mathrm{Mat}_3(\mathbb{F}_{p^2}).
\end{equation}
On the level of automorphisms (respecting the polarisation) we get
\begin{equation}\label{eq:redAutM2}
\mathrm{Aut}(M_2,\langle, \rangle_2)  \xrightarrow{m_p} G_{(M_2,\langle, \rangle_2)} \xrightarrow{m_{\Pi}} \overline{G}_{(M_2,\langle, \rangle_2)},
\end{equation}
where
\begin{equation}\label{eq:defGM2}
G_{(M_2,\langle, \rangle_2)} := \{ A + B \Pi \in \mathrm{GL}_3(\mathbb{F}_{p^2}[\Pi]) : A\overline{A}^T = \mathbb{I}_3, B^T \overline{A} = \overline{A}^T B \},
\end{equation}
(here, $B^T$ denotes the transpose of the matrix $B$), and where
\begin{equation}\label{eq:defbarGM2}
\overline{G}_{(M_2,\langle, \rangle_2)} := \{ A \in \mathrm{GL}_3(\mathbb{F}_{p^2}) : A^*A = \mathbb{I}_3  \}.
\end{equation}
\end{definition}

\begin{definition}\label{def:AutEndMmod}
For $M$ we have $\mathrm{End}(M) = \{ g \in \mathrm{End}(M_2) : g(M) \subseteq M \}$ and $\mathrm{Aut}(M) = \{ g \in \mathrm{Aut}(M_2) : g(M) = M \}$, and
\begin{equation}\label{eq:AutM}
\mathrm{Aut}(M,\langle, \rangle) = \{ g \in \mathrm{Aut}(M_2,\langle, \rangle_2) : g(M) = M \}.
\end{equation}
Under the same maps $m_p$ and $m_{\Pi}$, we find
\begin{equation}\label{eq:redMp}
E_M := m_p(\mathrm{End}(M)) = \{ A \in \mathrm{Mat}_3(\mathbb{F}_{p^2}[\Pi]) : A \cdot M/pM_2\subseteq M/pM_2 \}
\end{equation}
and $\overline{E}_M := m_{\Pi}(E_M) \subseteq \mathrm{Mat}_3(\mathbb{F}_{p^2})$. These fit in the diagram
\begin{equation}\label{eq:diagram}
\begin{tikzcd}
\
\mathrm{End}(M) \rar{}\dar{m_p} & \mathrm{End}(M_2) = \mathrm{Mat}_3(\mathcal{O}_{D_p}) \dar{m_p} \\
E_M \rar{}\dar{m_{\Pi}} & \mathrm{Mat}_3(\mathbb{F}_{p^2}[\Pi])) \dar{m_{\Pi}} \\
\overline{E}_M \rar{} & \mathrm{Mat}_3(\mathbb{F}_{p^2})
\end{tikzcd}
\end{equation}
in which all horizontal maps are inclusion maps and the left vertical maps are the surjective reduction maps.

On the level of automorphisms, we let
\begin{equation}\label{eq:defGM}
G_{M} := m_p(\mathrm{Aut}(M)) = \{ A \in \mathrm{GL}_3(\mathbb{F}_{p^2}[\Pi]) : A \cdot M/pM_2 \subseteq M/pM_2  \}
\end{equation} 
and $\overline{G}_M := m_{\Pi}(G_{M})$. For the polarised versions, since $\varphi^* \lambda = p\mu$, we obtain
\begin{equation}\label{eq:defGMpol}
G_{(M, \langle, \rangle)} := \{ g \in G_{(M_2,\langle, \rangle_2)} : g(M/pM_2) \subseteq M/pM_2 \}
\end{equation} 
and 
\begin{equation}\label{eq:defGbarMpol}
\overline{G}_{(M, \langle, \rangle)} := \{ g \in \overline{G}_{(M_2,\langle, \rangle_2)} : g(M/\Pi M_2) \subseteq M/\Pi M_2 \}.
\end{equation} 
\end{definition}

Denote the group of three-by-three symmetric matrices over
$\mathbb{F}_{p^2}$ by $S_3(\mathbb{F}_{p^2})$; this group has
cardinality 
$p^{12}$ (since it a six-dimensional $\mathbb{F}_{p^2}$-vector
space). Also recall that the group $U_3(\mathbb{F}_p)$ of three-by-three unitary matrices with
entries in $\mathbb{F}_{p^2}$ has cardinality  $p^3(p+1)(p^2-1)(p^3+1)$.

\begin{lemma}\label{lem:sizeGM2pol}
In Equation \eqref{eq:defGM2} we have $A \in U_3(\mathbb{F}_p)$ and
$B^T\overline{A} \in S_3(\mathbb{F}_{p^2})$.  Hence, 
\begin{equation}\label{eq:sizeGM2pol}
\vert G_{(M_2,\langle, \rangle_2)} \vert = \vert
U_3(\mathbb{F}_p) \vert \cdot \vert S_3(\mathbb{F}_{p^2}) \vert =
p^{15}(p+1)(p^2-1)(p^3+1). 
\end{equation}
\end{lemma}

\begin{remark}\label{rem:indexmodp}
Now we note, cf. \eqref{eq:indexsimple}, that
\begin{equation}\label{eq:indexmodp}
[\mathrm{Aut}((M_2,\langle, \rangle_2)):\mathrm{Aut}((M,\langle ,
\rangle))] = [G_{(M_2,\langle, \rangle_2)} : G_{(M, \langle , \rangle)}]. 
\end{equation}
In light of Lemma \ref{lem:sizeGM2pol}, it now suffices to compute $[G_{(M_2,\langle, \rangle_2)} : G_{(M, \langle , \rangle)}]$. This will take up the remainder of this section.
\end{remark}

We start by studying the unpolarised automorphisms $G_{M_2}$.
Thus, let $g=(a_{ij} + b_{ij}\Pi )_{1 \leq i,j \leq 3} \in {\rm GL}_{3}(\mathbb{F}_{p^2}(\Pi))$ be an (unpolarised) automorphism of  $M_2/pM_2$.
If we take $\bar{e}_{1}, \bar{e}_{2}, \bar{e}_{3}, \bar{f}_{1}, \bar{f}_{2}, \bar{f}_{3}$ (i.e., the reductions of $e_1, \ldots, f_3$ in the previous subsection) as a basis of $M_{2}/pM_{2}$ in this order, $g$ can be expressed by a matrix of the form
\begin{align}\label{eq:matrixg}
g =  \begin{pmatrix}
  A &  0 \\
  B & A^{(p)}
 \end{pmatrix},
\end{align}
where $A=(a_{ij})_{1 \leq i,j \leq 3}$,$B=(b_{ij})_{1 \leq i,j \leq 3}$, and $A^{(p)} = (a_{ij}^p)_{1 \leq i,j \leq 3}$. 

Recall from Propositions~\ref{prop:sections} and~\ref{prop:miniso}(1) that the polarised flag type quotient $Y_2 \to Y_1 \to X$ corresponds to a point $t = (t_1:t_2:t_3) \in C^0(k)$ such that $M_1/\sfF M_2$ is generated by $t_1\bar{e}_1 + t_2 \bar{e}_2 + t_3 \bar{e}_3$, where $M_1$ is the Dieudonn{\'e} module of $Y_1$, and a point $u = (u_1: u_2) \in
\bbP^1_t(k):=\pi^{-1}(t)$. 
We choose a new basis for $M_2/pM_2$ as follows:
\begin{align*}
\bar{E}_{1}:= \sum _{i=1,2,3}t_{i}\bar{e}_{i}, ~
\bar{E}_{2}:= \sum _{i=1,2,3}t_{i}^{p}\bar{e}_{i}, ~
\bar{E}_{3}:= \sum _{i=1,2,3}t_{i}^{p^{-1}}\bar{e}_{i}, \\
\bar{F}_{1}:= \sum _{i=1,2,3}t_{i}\bar{f}_{i}, ~
\bar{F}_{2}:= \sum _{i=1,2,3}t_{i}^{p}\bar{f}_{i}, ~
\bar{F}_{3}:= \sum _{i=1,2,3}t_{i}^{p^{-1}}\bar{f}_{i}.
\end{align*}
(This is a basis by Lemma \ref{lem:CFp2}.) 
Using this basis, $g$ is expressed as
\begin{align}\label{eq:newmatrixg}
g =  \begin{pmatrix}
  \mathbb{T}^{-1} A \mathbb{T} & 0 \\
  \mathbb{T}^{-1} B \mathbb{T} & \mathbb{T}^{-1} A^{(p)} \mathbb{T}
 \end{pmatrix},
\end{align}
where
\begin{align}\label{eq:T}
\mathbb{T} := 
 \begin{pmatrix}
  t_{1} & t_{1}^p & t_{1}^{p^{-1}} \\
  t_{2} & t_{2}^p & t_{2}^{p^{-1}} \\
  t_{3} & t_{3}^p & t_{3}^{p^{-1}} 
 \end{pmatrix}.
\end{align}

Now we determine the group $G_M \subseteq \mathrm{GL}_3(\mathbb{F}_{p^2}[\Pi])$ of elements preserving $M/pM_2$. Any such element will also preserve $M_1/pM_2$. We prove the following proposition.

\begin{proposition}\label{prop:GM}
Let $g \in \mathrm{GL}_3(\mathbb{F}_{p^2}[\Pi])$ be an automorphism of $M_{2}/pM_{2}$, expressed as in \eqref{eq:matrixg}.
Then $g \in G_M$ (i.e., $g$ preserves $M/pM_2$) if and only if the following hold:
\begin{itemize}
\item[(a)] We have $A \cdot t = \alpha t$ for some $\alpha \in k$, i.e., $A \in \mathrm{End}(t)$.
\item[(b)] The $(1,1)$-component of the matrix $\mathbb{T}^{-1} B \mathbb{T}$ is $u_2u_1^{-1}(\alpha - \alpha^{p^3})$. 
\end{itemize}
\end{proposition}

\begin{proof}
For an $A \in \mathrm{End}(t)$ (see Definition \ref{def:Endt}) with eigenvalue $\alpha$, it holds by definition that
\begin{align}\label{eq:comp}
\mathbb{T}^{-1} A \mathbb{T} = 
 \begin{pmatrix}
  \alpha & \ast & \ast \\
              & \ast & \ast \\
              & \ast & \ast 
 \end{pmatrix}, 
\mathbb{T}^{-1} A^{(p)} \mathbb{T} = 
 \begin{pmatrix}
  \ast &  &  \\
  \ast & \alpha^{p} &  \\
   \ast &  & \alpha^{p^{-1}}  
 \end{pmatrix}.
\end{align}
As ${\rm det}(A) = \alpha^{1+p^2+p^{-2}}$ and ${\rm det}(A^{(p)}) = {\rm det}(A)^p$, we see that 
\begin{align}\label{eq:comp2}
\mathbb{T}^{-1} A^{(p)} \mathbb{T} = 
 \begin{pmatrix}
  \alpha^{p^3} &  &  \\
  \ast & \alpha^{p} &  \\
   \ast &  & \alpha^{p^{-1}}  
 \end{pmatrix}.
\end{align}
By Proposition \ref{prop:miniso}(1), the quotient $M_{1}/pM_2$ is a
two dimensional $k$-vector space generated by $\bar{E}_{1}$ and
$\bar{F}_{1}$.  
As $M_1^{\vee} = (\sfF,\sfV)M_1 = pM_2$, we find that $M/pM_2 \subseteq M_{1}/pM_2$ is a one-dimensional $k$-vector space.  
Take $u_1, u_2 \in k$ so that $M/pM_2$ is generated by the image of $u_{1}\bar{E}_{1} + u_{2}\bar{F}_{1}$.
As $M \neq pM_{2}$, we see that $u_{1} \neq 0$.

We see that if $g \in \mathrm{GL}_3(\mathbb{F}_{p^2}[\Pi])$ preserves $M_{1}/(\sfF,\sfV)M_{2}$, then it induces an automorphism of $M_{1}/(\sfF,\sfV)M_{1} = M_1/pM_2$ which is expressed as 
$\left(\begin{smallmatrix}
\alpha & \\ \ast & \alpha^{p^3}
\end{smallmatrix}\right)$
by \eqref{eq:newmatrixg}, \eqref{eq:comp}, and \eqref{eq:comp2}.
Moreover, $g$ also preserves $M/(\sfF,\sfV)M_{1} = M/pM_2$ if and only if the column vector 
$\left( \begin{smallmatrix}
\alpha & \\ \ast & \alpha^{p^3}
\end{smallmatrix}\right) \left(\begin{smallmatrix} u_1 \\ u_2\end{smallmatrix}\right)$
is in the subspace spanned by $\left(\begin{smallmatrix} u_1 \\ u_2 \end{smallmatrix}\right)$. 
This is equivalent to the entry $\ast$ being equal to $u_2u_1^{-1}(\alpha - \alpha^{p^3})$.
\end{proof}

\begin{remark}\label{rem:choiceoftandu}
\begin{enumerate}
\item It follows from the construction of polarised flag type quotients that
for $(X,\lambda)$ with $a(X) =1$ and a choice $\mu \in P(E^3)$
together with an identification $(\wt X,\wt \lambda)= (
E^3_k,p\mu)$,   
there exists a unique pair $(t,u)$ where $t = (t_1:t_2:t_3) \in
C^0(k)$ and $u= (u_1:u_2) \in \mathbb{P}^1(k)$ as in the proof of
Proposition \ref{prop:GM}. For the rest of the section, we will work
with these $(t,u)$. 
\item The coordinates $(t,u)$ in (1) also give rise to 
a trivialisation $C^0\times \bbP\simeq \calP_{C^0}$, where 
$\calP_{C^0}:=\calP_\mu\times_C C^0$, as follows. 
By Proposition~\ref{prop:explicitmoduli}, points in $\calP_{C^0}$
correspond to pairs $(\ol M_1,\ol M)$: here $\ol M_1\subseteq \ol M_2$
is a four-dimensional subspace generated by the subspace $\sfV \ol M_2$ 
and $\ol E_1=t_1
\bar e_1+t_2 \bar e_2+t_3 \bar e_3$ with $(t_1:t_2:t_3)\in C^0$, and
$\ol M\subseteq \ol M_2$ is a three-dimensional subspace with $\ol
M_1^{\bot} \subseteq \ol  M \subseteq \ol M_1$, where $\ol
M_1^{\bot}$ is the orthogonal complement of $\ol M_1$ with respect to
$\<\,, \>_2$. The two-dimensional vector spaces $\ol M_1/\ol M_1^{\bot}$
for $t\in C^0$ form a rank two vector bundle $\calV=\calO(1)\oplus \calO(-1)|_{C^0}$ over $C^0$.  
As shown in the proof of Proposition \ref{prop:GM}, the images of $\ol E_1$ and $\ol F_1$ in $\ol M_1/\ol M_1^\bot$ (again denoted by
$\ol E_1$ and $\ol F_1$ for simplicity) form a basis, and give rise to two
global sections $\wt E_1$ and $\wt F_1$ of $\calV$ respectively 
(note that both $\ol E_1$ and $\ol F_1$ are vector-valued functions 
in $t_1$, $t_2$, and $t_3$). Then the desired trivialisation 
$C^0\times \bbP \isoto 
\calP_{C^0}\simeq \bbP(\calV)$ is given by 
$(t,(u_1:u_2))\mapsto [u_1 \wt E_1(t)+u_2 \wt F_1(t)]$. Since $M_2$ is
the \dieu module of $E^3_k$, the vector space $\ol M_2$ has an
$\F_{p^2}$-structure, so we see that this trivialisation is defined over
$\F_{p^2}$. 

Now let $t\in C^0(k)$ and $u=(0:1)$. The corresponding subspace $\ol M$ is
generated by $\ol F_1$ and $\ol M_1^{\bot}=(\sfF,\sfV)\ol M_1$. Therefore, we have $\ol M=\sfV \ol M_2$, which corresponds a point in $T$. 
It follows that under the above trivialisation, 
$T|_{C^0} \simeq C^0\times \{\infty\}$. 
\end{enumerate}   
\end{remark}

The following lemma follows from Lemma \ref{lem:Endt}, Lemma \ref{lem:CM}, and Proposition \ref{prop:GM}. It describes the \emph{polarised} elements $g \in G_{(M_2,\langle, \rangle_2)}$ that preserve $M_1/pM_2$: for such $g$ of the form \eqref{eq:matrixg}, Proposition \ref{prop:GM}(1) implies that $A \in \mathrm{End}(t)$, while Definition \ref{def:AutEndM2mod}\eqref{eq:defGM2} implies that $A$ is unitary.

\begin{lemma}\label{lem:EndtintU}
Let $t = (t_{1}:t_{2}:t_{3}) \in C^0(k)$.
\begin{enumerate}
\item When $t \notin C(\mathbb{F}_{p^6})$, we have
\begin{align*}
\mathrm{End}(t) \cap U_3(\mathbb{F}_p) \simeq \{ \alpha \in \mathbb{F}_{p^2} : \alpha^{p+1} = 1\}.
\end{align*}
\item When $t \in C(\mathbb{F}_{p^6})$, we have
\begin{align*}
\mathrm{End}(t) \cap U_3(\mathbb{F}_p) \simeq \{ \alpha \in \mathbb{F}_{p^6} : \alpha^{p^3+1} = 1 \} .
\end{align*}
\end{enumerate}
\end{lemma}

\begin{proof}
\begin{enumerate}
\item This follows since a diagonal matrix $\alpha \mathbb{I}_3$ with $\alpha \in \mathbb{F}_{p^2}$ is unitary if and only if $\alpha^{p+1} = 1$.
\item Take any $A \in \mathrm{End}(t) \cap U_3(\mathbb{F}_p)$.
The eigenvalues of $A^{(p)T}$ are $\alpha^{p}, \alpha^{p^3}, \alpha^{p^5}$ where $\alpha$ is the eigenvalue of $A$.
As $A$ is unitary, $\alpha^{-1}$ is also an eigenvalue, so we have $\alpha^{-1} \in \{ \alpha^{p}, \alpha^{p^3}, \alpha^{p^5} \}$.
In each case, we have $\alpha^{p^3+1} = 1$. 

For the converse, choose any $\alpha \in \mathbb{F}_{p^6}$ such that $\alpha^{p^3+1} = 1$.
By the proof of Lemma~\ref{lem:Endt}, the corresponding $A \in \mathrm{End}(t)$ is given by 
\begin{align*}
A = (t, t^{(p^2)}, t^{(p^4)}){\rm diag}(\alpha, \alpha^{p^2}, \alpha^{p^4})(t, t^{(p^2)}, t^{(p^4)})^{-1}.
\end{align*}
We compute that 
\begin{align*}
AA^{(p)T} = (t, t^{(p^2)}, t^{(p^4)})
\begin{pmatrix}
 & s^{-1} & \\
 & & s^{-p^2} \\
s^{-p} & &
\end{pmatrix}
(t^{(p)}, t^{(p^3)}, t^{(p^5)})^T
\end{align*}
where $s = t_{1}^{p^3+1}+t_{2}^{p^3+1}+t_{3}^{p^3+1}$.
That is, $AA^{(p)T}$ is independent of $\alpha$.
By the case $\alpha = 1$, we have $AA^{(p)T} = 1$. 
\end{enumerate}\end{proof}

Suppose now that we have $g \in G_{(M_2,\langle, \rangle_2)}$ of the form \eqref{eq:matrixg} preserving $M_1/pM_2$, i.e., we have $A \in  \mathrm{End}(t) \cap U_3(\mathbb{F}_p)$ by Lemma \ref{lem:EndtintU}. We now determine the conditions on $B$ so that $g$ also preserves $M/pM_2$, i.e., so that $g \in G_{(M,\langle , \rangle)}$. By \eqref{eq:defGM2}, $B$ satisfies a symmetric condition.

Let $S_3(\mathbb{F}_{p^2})A$ (for $A \in  \mathrm{End}(t) \cap U_3(\mathbb{F}_p)$ as above) be the $\mathbb{F}_{p^2}$-vector space consisting of matrices of the form $SA$ for some $S \in S_3(\mathbb{F}_{p^2})$.
Define a homomorphism of $\mathbb{F}_{p^2}$-vector spaces
\begin{equation}\label{eq:psi}
\begin{split}
\psi _{t, A} : & S_3(\mathbb{F}_{p^2})A \to k \\
& SA \mapsto \text{ the $(1,1)$-component of } \mathbb{T}^{-1} SA \mathbb{T}.
\end{split}
\end{equation}
Similarly define a homomorphism
\begin{equation}\label{eq:psix}
\begin{split}
\psi _{t} : & S_3(\mathbb{F}_{p^2}) \to k \\
& S \mapsto \text{ the $(1,1)$-component of } \mathbb{T}^{-1} S \mathbb{T}.
\end{split}
\end{equation}
Using these notations, we have the following proposition. 

\begin{proposition}\label{prop:GM2intGM}
The group $G_{(M, \langle , \rangle)}$ consists of the matrices of the form 
\[
\begin{pmatrix}
A & 0 \\
SA & A^{(p)}
\end{pmatrix}
\]
satisfying the following conditions:
\begin{enumerate}
\item $A \in \mathrm{End}(t) \cap U_3(\mathbb{F}_p)$ with eigenvalue $\alpha$;
\item $S \in S_3(\mathbb{F}_{p^2})$ is a symmetric matrix; and
\item $\psi _{t, A}(SA) = u_{2}u_{1}^{-1}(\alpha - \alpha^{p^3})$.
\end{enumerate}
The third condition is equivalent to 
\begin{itemize}
\item[(3')] $\psi _{t} (S) = u_{2}u_{1}^{-1}(1 - \alpha^{p^3 -1})$.
\end{itemize}
\end{proposition}

\begin{proof}
It follows from \eqref{eq:defGMpol} and Proposition \ref{prop:GM} that for $A \in \mathrm{End}(t) \cap U_3(\mathbb{F}_p)$ with eigenvalue $\alpha$, the matrix 
$
\begin{pmatrix}
A & 0 \\
B & A^{(p)} 
\end{pmatrix}
$
is an element of $G_{(M,\langle,\rangle_2)} \cap G_{(M,\langle, \rangle)}$ if and only if $BA^{-1}$ is a symmetric matrix and the $(1,1)$-component of the matrix $\mathbb{T}^{-1} B \mathbb{T}$ is $u_2u_1^{-1}(\alpha - \alpha^{p^3})$. 
The latter condition amounts to Condition (3) (and (3')) by noticing that since $\mathbb{T}^{-1} A \mathbb{T}$ is of the form 
\[
 \begin{pmatrix}
  \alpha & \ast & \ast \\
              & \ast & \ast \\
              & \ast & \ast 
 \end{pmatrix}
 \] 
where $\alpha$ is the eigenvalue of $A$, we have a commutative diagram
\begin{equation}\label{eq:diagpsi}
\begin{tikzcd}
\
S_3(\mathbb{F}_{p^2}) \rar{\psi_x}\dar{\cdot A} & k \dar{\cdot \alpha} \\
S_3(\mathbb{F}_{p^2})A \rar{\psi_{t,A}} & k
\
\end{tikzcd},
\end{equation}
where the left vertical arrow is multiplying $A$ from the right and the right vertical arrow is multiplying with $\alpha$.
\end{proof}

The following corollary follows immediately from Proposition \ref{prop:GM2intGM} and summarises the results in this subsection.

\begin{corollary}\label{cor:GM2intGMsize}
We have
\begin{equation}\label{eq:GM2intGMsize}
\vert G_{(M, \langle , \rangle)} \vert = \vert \{ A \in \mathrm{End}(t) \cap U_3(\mathbb{F}_p) : u_{2}u_{1}^{-1}(1 - \alpha^{p^3 -1}) \in \mathrm{Im}(\psi_t)  \} \vert \cdot \vert \ker(\psi_t) \vert.
\end{equation}
\end{corollary}

\subsection{Analysing $\boldsymbol{\mathrm{Im}(\psi_t)}$ and $\boldsymbol{\ker(\psi_t)}$}\

In the following subsection, we will make Corollary \ref{cor:GM2intGMsize} more explicit by analysing the image and kernel of the homomorphism $\psi_t$.

\begin{definition}\label{def:dx}
In the notation as above, we set
\begin{equation}\label{eq:defdx}
d(t) := \dim_{\mathbb{F}_{p^2}}({\rm Im}(\psi _{t})).
\end{equation}
\end{definition}
As $\dim_{\mathbb{F}_{p^2}}(S_3(\mathbb{F}_{p^2})) = 6$, we see that $d(t) \leq 6$, and that 
\begin{equation}\label{eq:sizeker}
\vert \ker(\psi _{t}) \vert = p^{2(6-d(t))}.
\end{equation}
We prove the following precise result about the values of $d(t)$.

\begin{proposition}\label{prop:dx}
We have $3 \leq d(t) \leq 6$.
When $p = 2$, we have $d(t) = 3$. Let $v = (t_1^2,t_2^2, t_3^2,t_1t_2, t_1t_3, t_2t_3)$ and let 
\[
\Delta = \left\{ \det\left(v^T, (v^{(p^2)})^T, (v^{(p^4)})^T, \ldots, (v^{(p^{10})})^T\right) = 0 \right\}.
\]
When $p \neq 2$, we have:
\begin{equation}\label{eq:dtvalues}
\begin{split}
d(t) = 3  \qquad \text{ if and only if } &  \qquad t \in C^0(\mathbb{F}_{p^6}); \\
d(t) = 4   \qquad \text{ if and only if } &  \qquad t \in C^0(\mathbb{F}_{p^8}); \\
d(t) = 5   \qquad\text{ if and only if } &  \qquad t \in \Delta \cap C^0 \setminus \left( C^0(\mathbb{F}_{p^6}) \amalg C^0(\mathbb{F}_{p^8})  \right);\\
d(t) = 6   \qquad \text{ if and only if } &  \qquad t \not\in \Delta \cap C^0.\\
\end{split}
\end{equation}
\end{proposition}

\begin{proof}
Since $t \in C^0(k)$, we see that $t_i \neq 0$, and without loss of generality we assume that $t_3 = 1$.
For $1 \leq i, j \leq 3$, let $I_{ij} $ be the three-by-three matrix whose $(i, j)$-component is one and where all other entries are zero.
Then $I_{11}, I_{22}, I_{33}, I_{12}+ I_{21}, I_{13} + I_{31}, I_{23} + I_{32} $ is a basis for $S_3(\mathbb{F}_{p^2})$ over $\mathbb{F}_{p^2}$.
We set 
\begin{equation}\label{eq:wi}
\begin{split}
&w_1 = \psi_t(I_{11}), w_2 = \psi_t(I_{22}), w_3 = \psi_t(I_{33}), \\
&w_4 = \psi_t(I_{12}+I_{21}), w_5 = \psi_t(I_{13}+I_{31}), w_6 = \psi_t(I_{23}+I_{32}). 
\end{split}
\end{equation}
\begin{lemma}\label{lem:wi}
The $w_i$ in \eqref{eq:wi} satisfy the following relations:
\begin{flalign*}
&w_1=t_1^2w_3 , \qquad w_2=t_2^2w_3 , \\
&w_4 = 2t_1t_2w_3, \\
&w_5 = 2t_1w_3, \qquad w_6 = 2t_2w_3,
\end{flalign*}
and $w_3$ is not zero.
\end{lemma}
\begin{proof}[Proof of lemma]
The inverse matrix of $\mathbb{T}$ is
\[
\mathbb{T}^{-1}=\det(\mathbb{T})^{-1}
\begin{pmatrix}
t_2^p-t_2^{p^{-1}} & t_1^{p^{-1}}-t_1^p & t_1^pt_2^{p^{-1}}-t_1^{p^{-1}}t_2^p \\
t_2^{p^{-1}}-t_2 & t_1-t_1^{p^{-1}} & t_1^{p^{-1}}t_2-t_1t_2^{p^{-1}} \\
t_2^p-t_2 & t_1-t_1^p & t_1^pt_2-t_1t_2^p 
\end{pmatrix}.
\]
Since for any matrices $M=(m_{ij})$, $N=(n_{ij})$ and $L=(l_{ij})$ the $(1, 1)$-component of $MNL$ is given by $\sum_{i, j}m_{1i}n_{ij}l_{j1}$, we have
\begin{align*}
w_1=\det(\mathbb{T})^{-1}(t_2^p-t_2^{p^{-1}})t_1; \\
w_2=\det(\mathbb{T})^{-1}(t_1^{p^{-1}}-t_1^p)t_2.
\end{align*}
Furthermore, $w_3$ is given by
\begin{align*}
w_3
&= \det(\mathbb{T})^{-1}(t_1^pt_2^{p^{-1}}-t_1^{p^{-1}}t_2^p) \\
&=\det(\mathbb{T})^{-1}t_1^{-1}(t_1^{p+1}t_2^{p^{-1}}-t_1^{p^{-1}+1}t_2^p) \\
&=\det(\mathbb{T})^{-1}t_1^{-1}(t_2^p-t_2^{p^{-1}}).
\end{align*}
For the last equality, we used equations $t_1^{p+1}+t_2^{p+1}+1=0$ and $t_1^{p^{-1}+1}+t_2^{p^{-1}+1}+1=0$.
Similarly, we see that $w_3=\det(\mathbb{T})^{-1}t_2^{-1}(t_1^{p^{-1}}-t_1^p)$.
These computations imply the first two relations of the assertion, and since $t_1, t_2 \not \in \mathbb{F}_{p^2}$, we see that $w_3$ is not zero.
Furthermore, we compute that
\begin{align*}
w_4
&=\det(\mathbb{T})^{-1}((t_2^p-t_2^{p^{-1}})t_2+(t_1^{p^{-1}}-t_1^p)t_1) \\
&=\det(\mathbb{T})^{-1}(t_2^{p+1}-t_2^{p^{-1}+1}+t_1^{p^{-1}+1}-t_1^{p+1}) \\
&=2\det(\mathbb{T})^{-1}t_2(t_2^p-t_2^{p^{-1}}); \\
w_5
&=\det(\mathbb{T})^{-1} ((t_2^p-t_2^{p^{-1}})+(t_1^pt_2^{p^{-1}}-t_1^{p^{-1}}t_2^p)t_1) \\
&=\det(\mathbb{T})^{-1}(t_2^p-t_2^{p^{-1}}+t_1^{p+1}t_2^{p^{-1}}-t_1^{p^{-1}+1}t_2^p) \\
&=2\det(\mathbb{T})^{-1}(t_2^p-t_2^{p^{-1}}).
\end{align*}
Similarly, we see that $w_6=2\det(\mathbb{T})^{-1}(t_1^{p^{-1}}-t_1^p)$, so we obtain the remaining relations.
\end{proof}
When $p \neq 2$, we see from Lemma \ref{lem:wi} that
\begin{align*}
d(t) = \dim_{\mathbb{F}_{p^2}}\langle w_1, w_2, w_3, w_4, w_5, w_6 \rangle = \dim _{\mathbb{F}_{p^2}}\langle 1, t_1 , t_2, t_1t_2, t_1^2, t_2^2 \rangle.
\end{align*}
In particular, this implies that
\begin{align*}
d(t) \geq \dim _{\mathbb{F}_{p^2}}\langle w_3, w_5, w_6 \rangle = \dim _{\mathbb{F}_{p^2}}\langle 1, t_1 , t_2 \rangle =3.
\end{align*}
When $p=2$, by Lemma \ref{lem:CFp2} and Lemma \ref{lem:wi}, we see that $d(t) =3$.
So assume $p \neq 2$, and consider \eqref{eq:dtvalues}.

By construction (since $t_3 = 1$), we have $t \in \Delta$ if and only if  $\dim _{\mathbb{F}_{p^2}}\langle 1, t_1 , t_2, t_1t_2, t_1^2, t_2^2 \rangle \leq 5$. 
Hence we see that $t \in \Delta \cap C^0$ if and only if $d(t) \leq 5$, which gives the required statement for $d(t) =6$.
Also note that if $d(t) \leq 5$ then there exists some conic $Q/\mathbb{F}_{p^2}$ with equation $a_1+a_2t_1+a_3t_2 + a_4t_1t_2 + a_5t_1^2 + a_6t_2^2 = 0$ such that $t \in C^0 \cap Q$.
Similarly if $d(t) \leq 4$ then there exist two independent conics
$Q_1, Q_2$ such that $t \in C^0 \cap Q_1 \cap Q_2$. In this case,
$Q_1$ and $Q_2$ do not have a common component (even defined over
$\Fpbar$). 
Otherwise, the intersection $Q_1\cap Q_2$ must be a line $L$ defined
over $\F_{p^2}$ (because we require $Q_1\neq Q_2$)
and $Q_1=L\cup L_1$ for another line $L_1$ defined over $\F_{p^2}$. 
This implies that $t\in L$ or $t\in L_1$, a contradiction by 
Lemma~\ref{lem:CFp2}.  
If $d(t) \leq 3$ there exist three independent conics $Q_1, Q_2, Q_3$
such that $t \in C^0 \cap Q_1 \cap Q_2 \cap Q_3$. 
 
If $t \in C^0(\mathbb{F}_{p^{2a}})$ then $d(t) \leq a$, i.e., if $2
\leq \deg_{\mathbb{F}_{p^2}}(t) \leq a$ then $d(t) \leq a$, for any
value of $a$. This shows in particular that if $t\in C^0(\F_{p^6})$,
then $d(t)=3$, cf.~Lemma~\ref{lem:CFp2}.  
Conversely, since $\vert Q_1 \cap Q_2 \vert \le 4$ 
by B{\'e}zout's theorem we see that if $d(t) \leq 4$ then
$\deg_{\mathbb{F}_{p^2}}(t) \leq 4$.  
That is, then $t \in C^0(\mathbb{F}_{p^8}) \cup
C^0(\mathbb{F}_{p^6})$; note that by Lemma~\ref{lem:Cmaxmim} we have
$C^0(\mathbb{F}_{p^4}) = \emptyset$. 
If $d(t)=3$, then the $\F_{p^2}$-subspace
$\<1,t_1,t_2,t_1^2,t_2^2,t_1t_2\>$ is equal to the $\F_{p^2}$-subspace
$U$ spanned by $1,t_1,t_2$. Since $t_1U\subseteq U$ and $t_2 U\subseteq
U$, the algebra $\F_{p^2}[t_1,t_2]=U$ has dimension three and
$\deg_{\F_{p^2}}(t)=3$. This implies that $d(t) = 3$ if and only if $t
\in C^0(\mathbb{F}_{p^6})$ and hence $d(t) = 4$ if and only if $t \in
C^0(\mathbb{F}_{p^8})$.  
The statement for $d(t) =5$ now follows.
\end{proof}

\begin{remark}\label{rem:d3deg3}
  We provide another proof of the implication $d(t)=3\implies
  \deg_{\F_{p^2}} (t)=3$, since this information may also be useful.
  Suppose $P_1,P_2,P_3,P_4\in \bbP^2(K)$, where $K$ is a field, are four
  distinct points not on the same line. Then the conics passing
  through them form a $\bbP^1$-family. To see this, suppose $Q$ is
  represented by $F(t)=0$, where $F(t)=a_1t_1^2+a_2t_2^2+a_3t_3^2 +
  a_4t_1t_2 + a_5t_1 t_2 + a_6t_1 t_3$. By assumption $P_1, P_2, P_3$ are not
  on the same line. Choose a coordinate for
  $\bbP^2$ over $K$ such that $P_1=(1:0:0), P_2=(0:1:0)$ and
  $P_3=(0:0:1)$. Then $a_1=a_2=a_3=0$. The point
  $P_4=(\alpha_1:\alpha_2:\alpha_3)$ satisfies 
  $(\alpha_1\alpha_2,\alpha_1\alpha_3,\alpha_2 \alpha_3)\neq
  (0,0,0)$. Thus, $F(P_4)=0$ gives a non-trivial linear relation among $a_4,a_5$, and $a_6$. 

  Suppose now $t\in C^0 \cap Q_1\cap Q_2 \cap Q_3$ with
  $\F_{p^2}$-linear independent conics $Q_1, Q_2, Q_3$. It suffices to
  prove $\vert Q_1\cap Q_2\cap Q_3\vert \le 3$. If $\vert Q_1\cap Q_2\vert \le 3$, then we are
  done. So suppose that $Q_1\cap Q_2=\{P_1,P_2,P_3,P_4\}$. If $Q_3$
  contains these four points, then $Q_3$ is a linear combination of
  $Q_1$ and $Q_2$ over some extension of $\F_{p^2}$ and by descent an
  $\F_{p^2}$-linear combination of  $Q_1$ and $Q_2$,
  contradiction. Thus, we have shown that $\vert Q_1\cap Q_2\cap Q_3\vert \le 3$. 
\end{remark}

\begin{definition}\label{def:D}
Let $\mathcal{P}_{C^0} \simeq C^0\times \mathbb{P}^1$ be the 
fibre $\mathbb{P}_{C}(\mathcal{O}(-1)\oplus
\mathcal{O}(1))\times_{C} C^0$ over $C^0$,
cf. Remark~\ref{rem:choiceoftandu}.  
For each $S \in S_3(\mathbb{F}_{p^2})$, we define a morphism 
$f_S: C^0 \to \mathcal{P}_{C^0}$  via 
the map $C^0 \ni t=(t_1:t_2:t_3) 
\mapsto
(t^{(p)}, (1:\psi_t(S)^p)) \in C^0\times \mathbb{P}^1$. 
Observe from the computation in the proof of
Proposition~\ref{prop:dx} that $\psi_t(S)$ is a polynomial function in 
$t_1^{p^{-1}},t_2^{p^{-1}}, t_3^{p^{-1}}$, and hence that $\psi_t(S)^p$ is a
polynomial function in $t_1,t_2, t_3$.   
The image of $f_S$ defines a Cartier divisor $\calD_S\subseteq \calP_{C^0}$, 
and we let $\calD$ be the horizontal divisor
\begin{align*}
\calD = \sum _{S \in S_3(\mathbb{F}_{p^2})}\calD_S.
\end{align*}
For $t \in C^0(k)$, 
let $\calD_t = \pi ^{-1}(t) \cap \calD$. That is, $(u_1:u_2) \in \calD_t$ if and only if $u_2u_1^{-1} \in \mathrm{Im}(\psi_t)$.
\end{definition}

\begin{lemma}\label{lem:imDt}
Let $t = (t_1:t_2:t_3) \in C^0(k)$.
\begin{enumerate}
\item If $t \not\in C^0(\mathbb{F}_{p^6})$, then 
         \begin{align*}
         \{ \alpha \in \mathbb{F}_{p^2}^{\times} : u_2u_1^{-1}(1-\alpha^{p^3-1}) \in \mathrm{Im}(\psi_t) \} =
        \begin{cases}
            \mathbb{F}_{p^2}^{\times} & \text{ if } (u_1:u_2) \in \calD_t; \\
            \mathbb{F}_{p}^{\times} & \text{ otherwise.}
         \end{cases}
         \end{align*}
\item If $t \in C^0(\mathbb{F}_{p^6})$, then
         \begin{align*}
         \{ \alpha \in \mathbb{F}_{p^6}^{\times} : u_2 u_1^{-1}(1-\alpha^{p^3-1}) \in \mathrm{Im}(\psi_t) \} =
        \begin{cases}
            \mathbb{F}_{p^6}^{\times} & \text{ if } (u_1:u_2) \in \calD_t; \\
            \mathbb{F}_{p^3}^{\times} & \text{ otherwise}.
         \end{cases}
         \end{align*}
\end{enumerate}
\end{lemma}

\begin{proof}
\begin{enumerate}
\item First we note that $\mathbb{F}_{p}^{\times} \subseteq \{ \alpha \in \mathbb{F}_{p^2}^{\times} : u_2u_1^{-1}(1-\alpha^{p^3-1}) \in \mathrm{Im}(\psi_t) \}$.
Since $\mathrm{Im}(\psi_t)$ is an $\mathbb{F}_{p^2}$-vector space, we have that if $(u_1:u_2) \in \calD_t$, i.e., if $u_2u_1^{-1} \in \mathrm{Im}(\psi_t)$, then $u_2u_1^{-1}(1-\alpha^{p^3-1}) \in \mathrm{Im}(\psi_t)$ for any $\alpha \in \mathbb{F}_{p^2}^{\times}$.
Conversely if  $u_2u_1^{-1}(1-\alpha^{p^3-1}) \in \mathrm{Im}(\psi_t)$ for some $\alpha \in \mathbb{F}_{p^2} \setminus \mathbb{F}_p$, then $u_2u_1^{-1} \in \mathrm{Im}(\psi_t)$.
\item If $t \in C^0(\mathbb{F}_{p^6})$, then $\mathrm{Im}(\psi_t) \subseteq \mathbb{F}_{p^6}$.
Since $\dim_{\mathbb{F}_{p^2}}(\mathbb{F}_{p^6}) = 3$ and $d(t) \geq 3$ by Proposition~\ref{prop:dx}, we must have that $\mathrm{Im}(\psi_t) = \mathbb{F}_{p^6}$.
The proof now follows from a similar argument as in (1).
\end{enumerate}
\end{proof}

\begin{corollary}\label{cor:GbarMpol}
We have
\begin{equation*}
\begin{split}
& \{ A \in \mathrm{End}(t) \cap U_3(\mathbb{F}_p) : u_{2}u_{1}^{-1}(1 - \alpha^{p^3 -1}) \in \mathrm{Im}(\psi_t) \} \simeq  \\
 & \begin{cases}
\{ \alpha \in \mathbb{F}_{p} : \alpha^{p+1} = 1\} & \text{ if } t \notin C^0(\mathbb{F}_{p^6}) \text{ and } u \notin \calD_t; \\
\{ \alpha \in \mathbb{F}_{p^2} : \alpha^{p+1} = 1\} & \text{ if } t \notin C^0(\mathbb{F}_{p^6}) \text{ and } u \in \calD_t; \\
\{ \alpha \in \mathbb{F}_{p^3} : \alpha^{p^3+1} = 1\} & \text{ if } t \in C^0(\mathbb{F}_{p^6}) \text{ and } u \notin \calD_t; \\
\{ \alpha \in \mathbb{F}_{p^6} : \alpha^{p^3+1} = 1\} & \text{ if } t \in C^0(\mathbb{F}_{p^6}) \text{ and } u \in \calD_t.
\end{cases}
\end{split}
\end{equation*}
\end{corollary}

\begin{proof}
This follows from combining Lemma \ref{lem:EndtintU} with  Lemma \ref{lem:imDt}.
\end{proof}

\subsection{Determining $\boldsymbol{[\mathrm{Aut}((M_2,\langle, \rangle_2)):\mathrm{Aut}((M,\langle , \rangle))]}$}\

By Corollary \ref{cor:GM2intGMsize}, Equation \eqref{eq:sizeker}, and the results in the previous subsection, in particular Corollary \ref{cor:GbarMpol}, we immediately obtain the following result.

\begin{lemma}\label{lem:GM2intGMexplicit}
Define $e(p) = 0$ if $p=2$ and $e(p) = 1$ if $p > 2$. Then
\begin{equation}\label{eq:GM2intGMexplicit}
\vert G_{(M, \langle , \rangle)} \vert = \begin{cases}
2^{e(p)}p^{2(6-d(t))} & \text{ if } u \notin \calD_t; \\
(p+1)p^{2(6-d(t))} & \text{ if } t \notin C^0(\mathbb{F}_{p^6}) \text{ and } u \in \calD_t; \\
(p^3+1)p^{6} & \text{ if } t \in C^0(\mathbb{F}_{p^6}) \text{ and } u \in \calD_t.
\end{cases}
\end{equation}
\end{lemma}

Recall that $d(t) = 3$ when $t \in C^0(\mathbb{F}_{p^6})$. 
Combining Lemma \ref{lem:GM2intGMexplicit} with Lemma \ref{lem:sizeGM2pol}, and using Remark \ref{rem:indexmodp}, we conclude the following.

\begin{corollary}\label{cor:muanumber1}
We have
\begin{equation}\label{eq:muanumber1}
\begin{split}
[\mathrm{Aut}((M_2,\langle, \rangle_2)):\mathrm{Aut}((M,\langle , \rangle))] = [G_{(M_2,\langle, \rangle_2)} : G_{(M, \langle , \rangle)}] = \\
 \begin{cases}
2^{-e(p)}p^{3+2d(t)}(p+1)(p^2-1)(p^3+1) & \text{ if } u \notin \calD_t; \\
p^{3+2d(t)}(p^2-1)(p^3+1) & \text{ if } t \notin C^0(\mathbb{F}_{p^6}) \text{ and } u \in \calD_t; \\
p^{9}(p+1)(p^2-1) & \text{ if } t \in C^0(\mathbb{F}_{p^6}) \text{ and } u \in \calD_t.
\end{cases}
\end{split}
\end{equation}
\end{corollary}

Now Corollary \ref{cor:sspmassg3}(1) and Corollary \ref{cor:muanumber1} yield the main result of this section, i.e., the mass formula for a supersingular principally polarised abelian threefold $x = (X, \lambda)$ of $a$-number $1$, cf. Theorem \ref{introthm:a1}.

\begin{theorem}\label{thm:anumber1}
Let $x = (X,\lambda) \in \mathcal{S}_{3,1}$ such that $a(X)=1$. For $\mu \in P^1(E^3)$, consider the associated polarised flag type quotient $(Y_2,\mu) \to (Y_1, \lambda_1) \to (X, \lambda)$ which is characterised by the pair $(t,u)$ with $t = (t_1:t_2:t_3)\in C^0(k)$ and $u = (u_1:u_2) \in \mathbb{P}^1(k)$. Let $(M_2, \langle, \rangle_2)$ and $(M, \langle, \rangle)$ be the respective polarised Dieudonn{\'e} modules of $Y_2$ and $X$, let $\calD_t$ be as in Definition \ref{def:D}, and let $d(t)$ be as in Definition \ref{def:dx}.
Then 
\begin{equation}\label{eq:anumber1}
\begin{split}
& \mathrm{Mass}(\Lambda_x) = \mathrm{Mass}(\Lambda_{3,1}) \cdot [\mathrm{Aut}((M_2,\langle, \rangle_2)):\mathrm{Aut}((M,\langle , \rangle))]  =\\
 \frac{p^{3}}{2^{10}\cdot 3^4 \cdot 5 \cdot 7} & \begin{cases}
2^{-e(p)}p^{2d(t)}(p^2-1)(p^4-1)(p^6-1) & \text{ if } u \notin \calD_t; \\
p^{2d(t)}(p-1)(p^4-1)(p^6-1) & \text{ if } t \notin C^0(\mathbb{F}_{p^6}) \text{ and } u \in \calD_t; \\
p^6(p^2-1)(p^3-1)(p^4-1) & \text{ if } t \in C^0(\mathbb{F}_{p^6}) \text{ and } u \in \calD_t.
\end{cases}
\end{split}
\end{equation}
\end{theorem}

\section{The automorphism groups}
\label{sec:Aut}

In this section we discuss the automorphism groups of principally
polarised abelian threefolds $(X,\lambda)$ over an algebraically closed 
field $k\supseteq \Fp$ with $a(X)=1$. 
We shall first focus on an open dense locus in $\calP_\mu(a=1)$ (the $a$-number one locus in $\calP_\mu$) in Subsection \ref{sec:notinD} and then discuss a few other cases in Subsections \ref{sec:outsideC6} and \ref{sec:ssp}.
To get started, we record some preliminaries in the next subsection.

\subsection{Arithmetic properties of definite quaternion algebras over $\boldsymbol{\Q}$}\label{ssec:prelim}\

Let $C_n$ denote the cyclic group of order $n\ge 1$. Let
$B_{p,\infty}$ denote 
the definite quaternion $\Q$-algebra ramified exactly at
$\{\infty, p\}$. The class number $h(B_{p,\infty})$ of $B_{p,\infty}$
was determined by Deuring, Eichler and Igusa (cf.~\cite{igusa}) as follows: 
\begin{equation}
  \label{eq:hB}
  h(B_{p,\infty})=\frac{p-1}{12}+\frac{1}{3}\left (1-\left
    (\frac{-3}{p}\right) \right ) +\frac{1}{4}\left (1-\left
    (\frac{-4}{p}\right) \right ),
\end{equation}
where $(\cdot/p)$ is the Legendre symbol. 
If $h(B_{p,\infty})=1$, then the type number of $B_{p,\infty}$ is one
and hence all maximal orders are conjugate. 
It follows from \eqref{eq:hB} that 
\begin{equation}
  \label{eq:h=1}
  h(B_{p,\infty})=1 \iff p\in \{2,3,5,7,13\}.
\end{equation}
If $p=2$, the quaternion algebra $B_{2,\infty}\simeq
\left(\frac{-1,-1}{\Q}\right)$ is generated by $i,j$ with relations 
$i^2=j^2=-1$ and $k:=ij=-ji$, and the $\Z$-lattice 
\begin{equation}
  \label{eq:O2inf}
  O_{2,\infty}:=\Span_{\Z} \left \{ 1,i,j, \frac{1+i+j+k}{2} \right \}
\end{equation}
is a maximal order of $B_{2,\infty}$. Moreover, 
\begin{equation}
      \label{eq:E24}
O^\times_{2,\infty}=\left \{\pm 1, \pm i, \pm j, \pm k, \frac{\pm 1 \pm i \pm j \pm
    k}{2}\right \}=:E_{24},  
\end{equation}
and one has $E_{24}\simeq \SL_2(\F_{3})$ and $E_{24}/\{\pm 1\}\simeq
A_4$. 

If $p=3$, the quaternion algebra $B_{3,\infty}\simeq
\left(\frac{-1,-3}{\Q}\right)$ is generated by $i,j$ with relations 
$i^2=-1, j^2=-3$ and $k:=ij=-ji$, and the $\Z$-lattice 
\begin{equation}
  \label{eq:O3inf}
  O_{3,\infty}:=\Span_{\Z} \left \{ 1,i,\frac{1+j}{2}, \frac{i(1+j)}{2} \right \}
\end{equation}
is a maximal order of $B_{3,\infty}$. Moreover, 
\begin{equation}
      \label{eq:T12}
O^\times_{3,\infty}=\<i,\zeta_6\>=:T_{12},  \quad \zeta_6=(1+j)/2,
\end{equation}
and one has $T_{12}\simeq C_4 \rtimes C_3$ and $T_{12}/\{\pm 1\}\simeq
D_3$, the dihedral group of order six. 

If $p \ge 5$, then $O^\times \in \{C_2,C_4,C_6\}$ for any maximal order
$O$ in $B_{p,\infty}$ \cite[V Proposition 3.1, p.~145]{vigneras}. 
Fix a maximal order $O$ in $B_{p,\infty}$ and
let $h(O,C_{2n})$ be the number of right $O$-ideal classes $[I]$ with
$O_\ell(I)^\times \simeq C_{2n}$, where $O_\ell(I)$ is the left order
of $I$. Then (see \cite{igusa})
\begin{equation}
  \label{eq:hC46}
  h(O,C_4)=\frac{1}{2}\left (1-\left
    (\frac{-4}{p}\right) \right ) \quad \text{and}\quad  
    h(O,C_6)=\frac{1}{2}\left (1-\left
    (\frac{-3}{p}\right) \right ). 
\end{equation}

\begin{lemma}\label{lm:UnO} \ 
\begin{enumerate}
\item Let $Q$ be a definite quaternion $\Q$-algebra and $O$ a
  $\Z$-order in $Q$,
  and let $n\ge 1$ be a
  positive integer. Then the integral 
  quaternion hermitian group $U(n,O)=\{A\in \Mat_n(O) : A\cdot A^*=\bbI_n\}$ is
  equal to the permutation unit group $\diag(O^\times, \dots,
  O^\times) \cdot S_n$. 
\item Let $O$ be a maximal order in $B_{2,\infty}$. 
 Let $m_2:U(n,O)\to \GL_n(O)\to \GL_n(O/2O)$ be the reduction-modulo-$2$ map. Then $\ker(m_2)=\diag(\{\pm 1\},\dots, \{\pm 1\})\simeq C_2^n$.   
\end{enumerate}
\end{lemma}

\begin{proof}
\begin{enumerate}
\item Note that $O$ is stable under the involution $*$ since $x^*=\Tr
  x-x$ and $\Tr x\in \Z$ for any $x$ in $O$. Let $A=(a_{ij})\in
  U(n,O)$. Then since $A A^*=\bbI_n$, we have $\sum_k a_{ik} a_{ik}^*=1$ for any $1\le i\le n$. Since   $a_{ik} a_{ik}^*=0$ or $1$, for any $1 \leq i \leq n$, there is only one integer $1\le k\le n$ such that $a_{ik}\neq 0$ and $a_{ik}\in
  O^\times$. 
  On the other hand, since $A^* A=\bbI_n$, for any $1\le k\leq n$, there is a only one integer 
  $1\le i\le n$ such that $a_{ik}\neq 0$ and $a_{ik}\in
  O^\times$. Thus, $A\in \diag(O^\times, \dots,
  O^\times) \cdot S_n$. Checking the reverse containment $\diag(O^\times, \dots, O^\times) \cdot S_n \subseteq U(n,O)$ is straightforward.
\item By \eqref{eq:h=1}, we may assume that $O=O_{2,\infty}$. Since the diagonal entries of elements in $\ker(m_2)$ are all not zero,
    by part (1) we find $\ker(m_2) \subseteq \diag (O^\times, \dots,
    O^\times)$. Therefore, it suffices to show that the kernel of the reduction-modulo-$2$ map
    $m_2: O^\times \to (O/2O)^\times$ is isomorphic to $C_2$. Using
    \eqref{eq:E24} and $2O=\{a_1+a_2 i+a_3 j+a_4 k : a_i\in \Z,\  a_1\equiv a_2 \equiv
    a_3\equiv a_4  \pmod 2 \}$, one checks that indeed $\ker(m_2)=\{\pm 1\} \subseteq O^\times$. 
\end{enumerate}
\end{proof}

\begin{lemma}\label{lm:Vp}
  Let $D_p$ be the quaternion division $\Qp$-algebra and $O_p$ its
  maximal order. Let $n \geq 1$ be a positive integer. Let $\Pi$ be a uniformiser of $O_p$, and put
  $V_p:=1+\Pi \Mat_n(O_p)\subseteq \GL_n(O_p)$. 
  If $p\ge 5$, then the torsion subgroup
  $(V_p)_{\mathrm{tors}}$ of $V_p$ is trivial.   
\end{lemma}

\begin{remark}
 Before giving the proof, let us note that $p\geq 5$ is best
  possible. Indeed, when $p=3$, we have
  \[ D_3=\left(\frac{-1,-3}{\Q_3} \right ), \qquad O_3=\Z_3[i, (1+j)/2]=\Z_3[i, j],
    \qquad \Pi=j.\]
Thus, we find the torsion element $-(1+j)/2\in 1+\Pi O_p$. 
\end{remark}
\begin{proof}[Proof of Lemma~\ref{lm:Vp}]
For simplicity, write $(\Pi)$ for the two-sided ideal in $\Mat_n(O_p)$
generated by $\Pi$.
  We must show that any $\alpha\in   (V_p)_{\mathrm{tor}}$ must equal $1$. 
  Since $V_p$ is a pro-$p$ group, we have $\alpha^{p^r}=1$ for some
$r\geq 1$. By induction, we may assume that $\alpha^p=1$. Suppose that
$\alpha\neq 1$ and write
$\alpha=1+\Pi\beta$ for some nonzero $\beta\in\Mat_n(O_p)
$. Necessarily, $\beta\not\in (\Pi)$, for
otherwise $\alpha\equiv 1\pmod{p}$, which implies that $\alpha=1$ by a lemma of Serre  \cite[p.~207]{mumford:av}. Since $p\geq 5$ and $p\mid \binom{p}{i}$ for all $1\leq
i\leq p-1$, we find
\begin{equation}
  \label{eq:1}
1=\sum_{i=0}^p \binom{p}{i}(\Pi\beta)^i\equiv 1+p\Pi\beta
  \pmod{\Pi^4}.  
\end{equation}
This implies that $\beta\in (\Pi)$, which leads to a contradiction. 
\end{proof}

\subsection{The region outside the divisor $\boldsymbol{\calD}$}
\label{sec:notinD}\

Recall from Subsection \ref{ssec:mod} that $E$ is a supersingular elliptic curve over $\mathbb{F}_{p^2}$ such that $\pi_E = -p$.
Let $\mu_{\mathrm{can}}\in P(E^3)$ 
be the threefold self-product of the canonical principal 
polarisation on $E$; this is also called the canonical polarisation on $E^3$.

\begin{theorem}\label{thm:gen_autgp}
  Let $x = (X,\lambda) \in \mathcal{S}_{3,1}(k)$ with $a(X)=1$. 
  For $\mu \in P(E^3)$, consider
  the associated polarised flag type quotient $(Y_2,p \mu) \to (Y_1,
  \lambda_1) \to (X, \lambda)$ which is characterised by the pair
  $(t,u)$ with $t = (t_1:t_2:t_3)\in C^0(k)$ and $u = (u_1:u_2) \in
  \mathbb{P}^1(k)$. Let $(M_2, \langle, \rangle_2)$ and $(M, \langle,
  \rangle)$ be the respective polarised Dieudonn{\'e} modules of 
  $(Y_2,\mu)$ and $(X,\lambda)$, let $\calD_t$ be as in Definition
  \ref{def:D} and let $d(t)$
  be as in Definition \ref{def:dx}. Assume that $(t,u)\not \in \calD$,
  that is, $u\not\in \calD_t$. 
  \begin{enumerate}
  \item If $p=2$, then $\Aut(X,\lambda)\simeq C_2^3$.
  \item If $p\ge 5$, or $p=3$ and $d(t)=6$, then
$\Aut(X,\lambda)\simeq C_2$.
\end{enumerate} 
\end{theorem}

\begin{proof}
  By Proposition~\ref{prop:miniso}, $(Y_2,p\mu)\to (X,\lambda)$ is the minimal isogeny. Therefore, 
  \begin{equation}
    \label{eq:AutXpol}
     \Aut(X,\lambda)=\{h\in \Aut(Y_2,\mu): m_p(h)\in G_{(M,\<\,,\>)} \}.
  \end{equation}
By Proposition~\ref{prop:GM2intGM}, we have an exact sequence 
\begin{equation}\label{eq:GMpolmodPi}
1\to \ker(\psi_t)  \xrightarrow{} G_{(M,\langle, \rangle)}
\xrightarrow{m_{\Pi}} \overline{G}_{(M,\langle, \rangle)} \to 1.
\end{equation} 
\begin{enumerate}
\item A direct calculation using the mass formula (cf. Corollary~\ref{cor:sspmassg3} and Lemma~\ref{lm:UnO}) shows \[ \Mass(\Lambda_{3,1})=\frac{1}{2^{10}\cdot 3^4} =\frac{1}{24^3\cdot 3!}=\frac{1}{\vert \Aut(E^3,\mu_{\rm can})\vert },\]
and hence $\vert \Lambda_{3,1}\vert =1$. 
    Thus, we may
    assume that $(Y_2,\mu)=(E^3,\mu_{\mathrm{can}})$, and  
    we have 
    $\Aut(Y_2,\mu)=\diag(O^\times,O^\times,O^\times)\cdot S_3$ by
    Lemma~\ref{lm:UnO} with $O=\End(E)$. As $u\not\in \calD_t$, Corollary~\ref{cor:GbarMpol} yields $\ol G_{(M,\<\,,\>)}=\{\pm 1\}=1$. 
    We see from the proof of Proposition~\ref{prop:dx} that $\ker(\psi_t)$ is the
    $\F_{p^2}$-subspace generated by $I_{12}+I_{21}$,
    $I_{13}+I_{31}$ and $I_{23}+I_{32}$ (in the notation of that proof). Therefore, 
    \begin{equation}
      \label{eq:GMp2}
    G_{(M,\<\,,\>)}=\left\{  
  \begin{pmatrix}
  \bbI_3 & 0 \\
   S & \bbI_3
  \end{pmatrix}: S=(s_{ij})\in S_3(\F_{p^2}), s_{ii}=0 \ \forall 1\le
   i\le 3\right\}.  
    \end{equation}
   Let $h\in \Aut(X,\lambda)\subseteq \diag(O^\times,O^\times,O^\times)\cdot S_3$. Since $m_2(h)$ has non-zero diagonal entries, $h\in \diag(O^\times,O^\times,O^\times)$. One deduces $m_2(h)=1$ from \eqref{eq:GMp2}. 
   Thus, $h\in \ker(m_2)=C_2^3$, by Lemma~\ref{lm:UnO}. On the other hand, $\ker(m_2)\subseteq \Aut(X,\lambda)$ from \eqref{eq:AutXpol}. This proves (1). 
\item Assume $p\ge 5$. As $u\not\in \calD_t$, Corollary~\ref{cor:GbarMpol} implies that $\ol G_{(M,\<\,,\>)}=\{\pm 1\}$.
Lemma~\ref{lm:Vp} implies that the map $m_\Pi:\Aut(X,\lambda)\to \ol
G_{(M,\<\,,\>)}$ is injective, because $\ker(m_{\Pi})$ is contained in
$(V_p)_{\rm tors}$. Thus, $\Aut(X,\lambda)\simeq C_2$. 
Now assume $p=3$ and $d(t)=6$. In this case $G_{(M,\<\,,\>)}=\{\pm 1\}$ follows from \eqref{eq:GMpolmodPi} and
Corollary~\ref{cor:GbarMpol}. By a lemma of Serre \cite[p.~207]{mumford:av}, 
the map $m_3:\Aut(X,\lambda)\to G_{(M,\<\,,\>)}$ is injective and 
hence $\Aut(X,\lambda)\simeq C_2$.
\end{enumerate}
\end{proof}

\begin{corollary}\label{cor:size_Lx}
  Let the notation and assumptions be as in
  Theorem~\ref{thm:gen_autgp}. 
  \begin{enumerate}
  \item If $p=2$, then $\vert \Lambda_x\vert =4$.
  \item If $p=3$ and $d(t)=6$, then $\vert \Lambda_x\vert =3^{11}\cdot 13$.
  \item If $p\ge 5$, then
\begin{equation}
  \label{eq:Lx_p>3}
  \vert \Lambda_x\vert =\frac{p^{3+2d(t)} (p^2-1)(p^4-1)(p^6-1)}
  {2^{10}\cdot 3^4 \cdot 5 \cdot 7}. 
\end{equation}
\end{enumerate}
\end{corollary}
\begin{proof}
  All statements follow from Theorems \ref{thm:anumber1} and
  \ref{thm:gen_autgp}. For $p=2$, we have $\Aut(X,\lambda)\simeq
  C_2^3$ for each $(X,\lambda)\in \Lambda_x$ and hence
\begin{equation}
  \label{eq:Lx_p=2}
  \vert \Lambda_x\vert =\frac{2^3\cdot 2^9  \cdot 3\cdot (3\cdot
  5)\cdot (3^2\cdot 7)}
  {2^{10}\cdot 3^4 \cdot 5 \cdot 7}=4.
\end{equation}
For $p=3$ and $d(t)=6$, we have $\Aut(X,\lambda)\simeq C_2$ for each
$(X,\lambda)\in \Lambda_x$ and hence 
\begin{equation}
  \label{eq:Lx_p=3}
  \vert \Lambda_x\vert =\frac{3^{3+2d(t)} \cdot 2^3\cdot (2^4\cdot
  5)\cdot (2^3\cdot 7\cdot 13)}{2^{10}\cdot 3^4 \cdot 5 \cdot 7}=
  3^{2d(t)-1}\cdot 13=3^{11}\cdot 13.  
\end{equation}
The same arguement gives \eqref{eq:Lx_p>3} for $p\ge 5$. 
\end{proof}

A $g$-dimensional principally polarised supersingular 
abelian variety $(X,\lambda)$ 
over $k$ is said to be \emph{generic} if the moduli point $\Spec k \to
\mathcal{S}_{g,1}$ factors through a generic point of $\mathcal{S}_{g,1}$.
Recall that the supersingular locus $\mathcal{S}_{g,1}\subseteq \calA_{g,1}\otimes
\Fpbar$ is a scheme of finite type over $\Fpbar$ which is defined over
$\Fp$. Moreover, every geometrically irreducible component of
$\mathcal{S}_{g,1}$ is defined over $\F_{p^2}$, cf.~\cite[Section 2.2]{yu:fod_ss}.
  
Oort's conjecture \cite[Problem 4]{edixhoven-moonen-oort} 
asserts that for any integer $g\ge 2$ and any prime number
$p$, every generic $g$-dimensional principally polarised supersingular
abelian variety $(X,\lambda)$
over $k$ of characteristic $p$ has automorphism group $\{\pm 1\}$. 
Oort's conjecture fails with counterexamples in $(g,p)=(2,2)$ or
$(g,p)=(3,2)$; see \cite{oort2,ibukiyama}. 

For fixed $g\ge 2$ and prime number $p$, consider the refined Oort 
conjecture:  
\begin{itemize}
\item[$(\mathrm{O})_{g,p}$:] Every generic $g$-dimensional
principally polarised supersingular abelian variety $(X,\lambda)$ 
over $k$ of characteristic $p$ has automorphism group $\{\pm 1\}$.  
\end{itemize}

\begin{corollary}\label{cor:Oortconj}
  Let $(X,\lambda)$ be a generic principally polarised supersingular
  abelian threefold
  over $k$ of characteristic $p>0$. Then 
\[ \Aut(X,\lambda)\simeq 
\begin{cases}
  C_2^3 & \text{for $p=2$;} \\
  C_2  & \text{for $p\ge 3$.} 
\end{cases} \]
\end{corollary} 
\begin{proof}
  This follows immediately from Theorem~\ref{thm:gen_autgp}. 
\end{proof}

In other words, Oort's Conjecture $(\mathrm{O})_{3,p}$ holds precisely when
$p\neq 2$. 
 
\begin{remark}\label{rem:Autp=23}
\begin{enumerate}
\item It is shown \cite[Theorem 5.6, p.~270]{oort2} 
  that if $(X,\lambda)$ is a principally polarised
  supersingular abelian threefold
  over $k$ of characteristic $2$, then
  $\Aut(X,\lambda)\supseteq C_2^3$. By Corollary~\ref{cor:Oortconj}, the
  smallest group $C_2^3$ also appears as $\Aut(X,\lambda)$ for  some
  $(X,\lambda)$. We have seen that the unique member $(E^3,\mu_{\rm can})$ in $\Lambda_{3,1}$ has automorphism group $E_{24}^3\rtimes S_3$ (of order $2^{10}\cdot 3^4$). 
   We expect that $2^{10}\cdot 3^4$ is the maximal order of automorphism groups of \emph{all} principally
  polarised abelian threefolds over $k$ of any characteristic (including zero).
\item According to Hashimoto's result \cite{hashimoto:g=3}, we have $\vert \Lambda_{3,1}\vert =2$ for $p=3$. In this case, we have two isomorphism classes, represented by $(E^3,\mu_{\mathrm{can}})$ and $(E^3,\mu)$. Using Lemma~\ref{lm:UnO}, we compute $\vert\Aut(E^3,\mu_{\mathrm{can}})\vert=2^7\cdot 3^4$ and
    conclude $\vert\Aut(E^3,\mu)\vert=2^7\cdot 3^4$ from the mass formula
    $\Mass(\Lambda_{3,1})=1/(2^6\cdot 3^4)$.
\end{enumerate}
\end{remark}

\subsection{The region where $\boldsymbol{t\not\in C(\F_{p^6})}$ and $\boldsymbol{(t,u)\in \calD}$. }\label{sec:outsideC6}\

In this subsection we consider the region $(t,u)\in
  \calD$ and assume that $t\not\in C(\F_{p^6})$. This extends the region considered in Subsection~\ref{sec:notinD}.

\begin{lemma}\label{lm:C_p+1}
Let $(X,\lambda)\in \mathcal{S}_{3,1}(k)$ with $a(X)=1$. If $p\ge 3$ and
$\Aut(X,\lambda)\subseteq C_{p+1}$, then $\Aut(X,\lambda)\subseteq
\{C_2,C_4,C_6\}$.     
\end{lemma}
\begin{proof}
  Suppose that $\Aut(X,\lambda)=C_{2d}$ with $2d\vert (p+1)$. Then we have a
  ring homomorphism $\Z[C_{2d}]\to \End(X)$ which maps $C_{2d}$
  bijectively to $\Aut(X,\lambda)$. 
The $\Q$-algebra homomorphism
\[ \Q[C_{2d}]=\prod_{d'\vert 2d} \Q[\zeta_{d'}] \to \End^0(X)=\Mat_3(B_{p,\infty}) \]
factors through an injective $\Q$-algebra homomorphism
\[ \prod_{i=1}^r \Q[\zeta_{d_i}]\embed \End^0(X)=\Mat_3(B_{p,\infty}), \]
where $\{d_i \vert 2d\} \subseteq \{d' \vert 2d\}$. 
Since the composition gives an embedding
$C_{2d}\embed \Aut(X)$, the integers $\{d_i\}$ satisfy $\mathrm{lcm}(d_1,\dots,
d_r)=2d$. 
Since $p\nmid 2d$, the algebra $\Zp[C_{2d}]$ is {\'e}tale over $\Zp$ and
is the maximal order in $\Qp[C_{2d}]$. This gives rise to
an embedding  $\prod_{i=1}^r \Z[\zeta_{d_i}]\otimes
\Zp \embed \End(X)\otimes \Zp \simeq \End(X[p^\infty])$. Thus, the decomposition 
$X[p^\infty]=H_1\times \dots
\times H_r$ into a product of supersingular $p$-divisible groups 
shows  $a(X)\ge r$ and hence $r=1$. Therefore,
there is a $\Q$-algebra embedding of $\Q(\zeta_{2d})$ into
$\Mat_{3}(B_{p,\infty})$. 
This implies that $\varphi(2d)\vert 6$ (where $\varphi$ denotes Euler's totien function) and hence $2d\in \{2,4,6,14,18\}$.

If $2d=14$, then $p\equiv -1 \pmod 7$ and $\ord(p)=2$ in $(\Z/7\Z)^\times$. This gives rise to an embedding $\Z[\zeta_{14}]\otimes \Zp=\Z_{p^2}\times
\Z_{p^2}\times \Z_{p^2} \embed \End(X[p^\infty])$ and hence $a(X)=3$, a contradiction.   
If $2d=18$,
then $p\equiv -1 \pmod 9$ and $\ord(p)=2$ in $(\Z/9\Z)^\times$. 
Similarly, we
get an embedding $\Z[\zeta_{18}]\otimes \Zp=\Z_{p^2}\times
\Z_{p^2}\times \Z_{p^2}\embed \End(X[p^\infty])$ and $a(X)=3$, again a contradiction.   
\end{proof}

Recall that $\F_{p^2}^1:=\{\alpha\in \F_{p^2}^\times: \alpha^{p+1}=1\}\simeq
  C_{p+1}$ denotes the group of norm one elements in~$\F_{p^2}^\times$.

\begin{theorem}\label{thm:inD}
  Let the notation be as in Theorem~\ref{thm:gen_autgp}. Assume that $(t,u)\in \calD$ and $t\not \in C(\F_{p^6})$. 
\begin{enumerate}
\item If $p=2$, then $\Aut(X,\lambda)\simeq C_2^3 \times C_3$.
\item If $p=3$ and $d(t)=6$, then $\Aut(X,\lambda)\in
\{C_2,C_4\}$.
\item For $p\ge 5$, we have the following cases:
\begin{itemize}
\item [(i)] If $p\equiv -1 \pmod {4}$, then $\Aut(X,\lambda)\in
  \{C_2,C_4\}$. 
\item [(ii)] If $p\equiv -1 \pmod {3}$, then $\Aut(X,\lambda)\in
  \{C_2,C_6\}$.
\item [(iii)] If $p\equiv 1 \pmod {12}$, then $\Aut(X,\lambda)\simeq C_2$.  
\end{itemize}
\end{enumerate}
\end{theorem}
\begin{proof}
\begin{enumerate}
\item As in Theorem~\ref{thm:gen_autgp}(1), we may assume that $(Y_2,\mu)=(E^3,\mu_{\mathrm{can}})$, and by Lemma~\ref{lm:UnO} we have $\Aut(Y_2,\mu)=\diag(O^\times,O^\times,O^\times)\cdot S_3$. Then 
\[ 
\begin{split}
  \Aut(X,\lambda)& =\left \{h\in \Aut(Y_2, \mu): m_2(h)=
\begin{pmatrix}
  a & & \\
  & a & \\
  & & a
\end{pmatrix}, a\in \F_4^1\right \} \\
&=\left \{h\in \diag(O^\times,O^\times,O^\times): m_2(h)=
\begin{pmatrix}
  a & & \\
  & a & \\
  & & a
\end{pmatrix}, a\in \F_4^1 \right \} \\
&=\left \{
\begin{pmatrix}
  \pm w^j & & \\
  & \pm w^j & \\
  & & \pm w^j
\end{pmatrix}: 0\le j \le 5 \right \}\simeq C_2^3\times C_3, 
\end{split} \]
where $w=(1+i+j+k)/2$ satisfies $w^6=1$. 
\item In this case, $\ol G_{(M,\<\,,\>)}=\F_{9}^1\simeq C_4$ by Corollary~\ref{cor:GbarMpol}. The proof then
    follows from the fact that the
    reduction-modulo-$3$ map is injective.
\item In this case, $\ol G_{(M,\<\,,\>)}=\F_{p^2}^1\simeq C_{p+1}$ by Corollary~\ref{cor:GbarMpol}. It
    follows from Lemma~\ref{lm:Vp} that $\Aut(X,\lambda)$ can be
    identified with a subgroup of $\ol G_{(M,\<\,,\>)}\simeq C_{p+1}$ as $p\ge 5$. 
    By Lemma~\ref{lm:C_p+1}, $\Aut(X,\lambda)\in
    \{C_2,C_4,C_6\}$. The assertions for (i), (ii), (iii) and (iv) 
   follow from this assertion.  
\end{enumerate}
\end{proof}

Write $\calD_\mu$ for $\calD\subseteq \calP_{\mu}(a=1)$ 
to emphasise its dependence on $\mu\in P(E^3)$. 
Recall that $\Psi_\mu: \calP_{\mu}\to \mathcal{S}_{3,1}$ is the map  $(Y_\bullet,\rho_\bullet)\mapsto (Y_0,\lambda_0)$.  
Put $\calD_{\mu, C(\F_{p^6})^c}:=\{(t,u)\in \calD_\mu: t\not\in C(\F_{p^6})\}$. 

Let $\Lambda_1$ denote the set of $\F_{p^2}$-isomorphism classes of
supersingular elliptic curves $E'$ over $\F_{p^2}$ with Frobenius
endomorphism $\pi_{E'}=-p$. This set is in bijection with the set
$\Cl(B_{p,\infty})$ of right $O$-ideal classes for a fixed maximal
order $O$ in $B_{p,\infty}$; see~\cite{deuring} (also cf.\cite[Theorem 2.1]{xueyu}).

\begin{proposition}\label{prop:C246}\  
\begin{enumerate}
\item If $p=3$ and $d(t)=6$, then for all $(X,\lambda)\in
\Psi_{\mu}(\calD_{\mu, C(\F_{p^6})^c})$ with $\mu = \mu_{\mathrm{can}}$, one
has $\Aut(X,\lambda)\simeq  
C_4$.
\item If $p\ge 5$ and $p\equiv 3 \pmod 4$, then there exists $\mu\in P(E^3)$
such that for all $(X,\lambda)\in 
\Psi_\mu(\calD_{\mu, C(\F_{p^6})^c})$ one has $\Aut(X,\lambda)\simeq 
C_4$.
\item If $p\ge 5$ and $p\equiv 2 \pmod 3$, then there exists $\mu\in P(E^3)$
such that for all $(X,\lambda)\in
\Psi_\mu(\calD_{\mu, C(\F_{p^6})^c})$ one has $\Aut(X,\lambda)\simeq 
C_6$.
\item If $p\ge 11$, then there exists $\mu\in P(E^3)$
such that for all $(X,\lambda)\in
\Psi_\mu(\calD_{\mu, C(\F_{p^6})^c})$ one has $\Aut(X,\lambda)\simeq 
C_2$.
\end{enumerate}
\end{proposition}

\begin{proof}
We use the results from Subsection~\ref{ssec:prelim}. 
If $p=3$, then $O^\times=\Aut(E)=\<i,\zeta_6\>$.
If $p\ge 5$ and $p\equiv 2 \pmod 3$ (resp. $p\equiv 3 \pmod 4$), there
exists a unique supersingular elliptic curve $E'$ in $\Lambda_1$ 
such that  $O^\times:=\Aut(E')\simeq C_6$ (resp.~$C_4$). 
If $p\ge 11$, then there exists a supersingular elliptic
curve $E'$ in $\Lambda_1$ such that  $O^\times:=\Aut(E')\simeq C_2$. Note that if $p\ge 11$ then either $h(B_{p,\infty})\ge 2$ or
$p\equiv 1 \pmod {12}$.  
For cases (2), (3), and (4) we choose a polarisation $\mu\in P(E^3)$ such that $(E^3,\mu)\simeq (E'^3,\mu'_{\mathrm{can}})$, where $\mu'_{\mathrm{can}}$ is the canonical polarisation on $E'^3$ as before. (In case (1) $\mu = \mu_{\mathrm{can}}$ is the unique choice of polarisation.)
Then using the same argument as in Theorem~\ref{thm:inD}, the automorphism group 
$\Aut (X,\lambda)$ for $(X,\lambda)\in \Psi_\mu(\calD_{\mu, C(\F_{p^6})^c})$ consists of elements of the form $\diag(a,a,a)$ with $a\in O^\times$ satisfying
 $m_3(a)\in \F_4^1$ if $p=3$ (resp. $m_\Pi(a)\in \F_{p^2}^1$ if $p\ge 5$). 
If $p=3$, we have $m_3(\<i\>)=C_4$. 
If $p\equiv 3 \pmod 4$, we have $m_\Pi(\<i\>)=C_4$. 
If $p\equiv 2 \pmod 3$, we have $m_\Pi(\<\zeta_6\>)=C_6$. 
Thus, $\Aut(X,\lambda)\simeq C_4$ for $p\equiv 3 \pmod 4$ and
$\Aut(X,\lambda)\simeq C_6$ 
for $p\equiv 2 \pmod 3$. In case (4), we have $\Aut(X,\lambda)\simeq C_2$.  
\end{proof}

\begin{remark}
\begin{enumerate}
\item Given Proposition~\ref{prop:C246}, it remains to check whether the
group $C_2$ also appears as $\Aut(X,\lambda)$ in the region $\Psi_\mu(\calD_{\mu, C(\F_{p^6})^c})$ for some $\mu\in P(E^3)$ when $p~=~3,5,7$.  
\item We assume the condition $d(t)=6$ when $p=3$ in Theorems 
\ref{thm:gen_autgp} and \ref{thm:inD}.  
It remains to determine which other automorphism groups occur if this condition is dropped.        
\end{enumerate}
\end{remark}

\subsection{The superspecial case}
\label{sec:ssp}\

As we have seen in the previous subsection, to investigate the
automorphism groups in some special region of $\calP_{\mu}(a=1)$, the
knowledge of automorphism groups arising from the superspecial locus 
$\Lambda_{3,1}$ also
plays an important role.   
In this subsection, we discuss only preliminary results on the 
automorphism groups of members in  $\Lambda_{3,1}$. 
A complete list of all possible automorphism groups requires much more
work; see Question (2) below. 

We briefly recall some results. For $p=2$, we have $\vert\Lambda_{3,1}\vert=1$ and the unique isomorphism class represented by $(X,\lambda)$ has automorphism group $E_{24}^3\rtimes S_3$. For $p=3$, we have $\vert\Lambda_{3,1}\vert=2$ 
by Hashimoto's result. In this
case, the two isomorphism classes are represented by $(E^3,\mu_{\rm can})$ and
$(E^3,\mu)$, respectively, and we have $\Aut(E^3,\mu_{\rm can})=T_{12}^3\rtimes S_3$
so $\vert\Aut(E^3,\mu)\vert=2^7\cdot 3^4$, cf. Remark~\ref{rem:Autp=23}.
For $p\ge 5$, the following non-abelian groups occur:
\[
\begin{cases}
  C_2^3 \rtimes S_3 & \text{for $p\equiv 1 \pmod {12}$};\\
  C_4^3 \rtimes S_3 & \text{for $p\equiv 3\pmod {4}$};\\
  C_6^3 \rtimes S_3 & \text{for $p\equiv 2\pmod {6}$},
\end{cases}
\] cf.~Lemma~\ref{lm:UnO}.    

Unlike the $a$-number one case, it is more difficult to construct a member
$(X,\lambda)$ in $\Lambda_{3,1}$ such that $\Aut(X,\lambda)\simeq
C_2$. However, it is expected that when $p$ goes to infinity, most members of $\Lambda_{g,1}$ have automorphism group $C_2$. The following
result confirms this expectation for $g=3$, 
based on Hashimoto's result~\cite{hashimoto:g=3}.

\begin{proposition}\label{prop:asympt}
  Let $\Lambda_{3,1}(C_2):=\{(X,\lambda)\in \Lambda_{3,1}:
\Aut(X,\lambda)\simeq C_2\}$. Then 
\begin{equation}
  \label{eq:asympt}
  \frac{\vert \Lambda_{3,1}(C_2)\vert }{\vert \Lambda_{3,1}\vert }\to 1 \quad \text{as\ 
$p\to \infty$}.
\end{equation}
\end{proposition}
\begin{proof}
  Put $h_2(p):=\vert\Lambda_{3,1}(C_2)\vert$.
  By \cite[Main Theorem]{hashimoto:g=3}, the main term of
  $h(p):=\vert\Lambda_{3,1}\vert$ is $H_1(p):=(p-1)(p^2+1)(p^3-1)/(2^9\cdot
  3^4\cdot 5\cdot 7)$ and the error term $\varepsilon(p)$ is
  $O(p^5)$. Observe that $\Mass(\Lambda_{3,1})=H_1(p)/2$. If
  $(X,\lambda)\not\in \Lambda_{3,1}(C_2)$, then $\vert\Aut(X,\lambda)\vert\ge 4$. This gives the inequality
\[ 
\Mass(\Lambda_{3,1})\le \frac{h_2(p)}{2}+\frac{h(p)-h_2(p)}{4}= \frac{h_2(p)}{4}+\frac{H_1(p)+\varepsilon(p)}{4}. 
\]
From $\Mass(\Lambda_{3,1})=H_1(p)/2$ one deduces that $h_2(p)\ge H_1(p)-\varepsilon(p)$. Since
\[ 
\frac{H_1(p)-\varepsilon(p)}{H_1(p)+\varepsilon(p)}\le
{\frac{\vert\Lambda_{3,1}(C_2)\vert}{\vert\Lambda_{3,1}\vert}}\le 1 \quad  \text{and} \quad \frac{H_1(p)-\varepsilon(p)}{H_1(p)+\varepsilon(p)}\to 1 \quad
\text{as  $p\to \infty$},
\]
we get the assertion \eqref{eq:asympt}.    
\end{proof}

We end the paper with some open problems.

\begin{questions*}
\begin{enumerate}
\item Let $X$ be a principally polarisable supersingular abelian variety over
$k$, and let $P(X)$ be the set of isomorphism classes of principal
polarisations on $X$. The mass of $P(X)$ is defined as
\begin{equation}
  \label{eq:massPX}
  \Mass(P(X)):=\sum_{\lambda\in P(X)} \frac{1}{\vert \Aut(X,\lambda)\vert }.
\end{equation}
One would like to find a mass formula for $\Mass(P(X))$ and understand the relationship between the sets $P(X)$ and $\Lambda_{(X,\lambda)}$ for a polarisation
$\lambda\in P(X)$ when $\dim(X)~=~3$ .
Ibukiyama \cite{ibukiyama} studied $P(X)$ for
$\dim(X)=2$. He gave a mass formula for $\Mass(P(X))$ and also showed that $P(X)$
is in bijection with
the set $\Lambda_{(X,\lambda)}$ for any principal polarisation $\lambda$ on
$X$. Note that not every supersingular abelian threefold is
principally polarisable: by \cite[Theorem 10.5, p.~71]{lioort} we see that the
supersingular locus $\calS_{3,d}\subseteq \calA_{3,d}\otimes \Fpbar$ is three-dimensional if $d$ is divisible by a high power of $p$, while $\dim(\mathcal{S}_{3,1})=2$.  
\item In order to study the automorphism groups of $(X,\lambda)$ with
$a(X)=2$, we also need to study the automorphism groups arising from
the non-principal genus $\Lambda_{3,p}$; see
Proposition~\ref{prop:miniso}. Do we have an asymptotic result similar to Proposition~\ref{prop:asympt} for $\Lambda_{3,p}$?  
What are the possible automorphism groups arising from $\Lambda_{3,1}$ or from $\Lambda_{3,p}$? 
We refer to Ibukiyama-Katsura-Oort~\cite{Ibukiyama-Katsura-Oort-1986}, 
Katsura-Oort~\cite{katsuraoort:compos87} and 
Ibukiyama~\cite{ibukiyama:autgp1989} for
detailed investigations for the principal genus case $\Lambda_{2,1}$
and the non-principal genus case $\Lambda_{2,p}$.
Observe that there are natural maps $\Lambda_{2,1}\times
\Lambda_{1,1}\to \Lambda_{3,1}$ and  $\Lambda_{2,p}\times
\Lambda_{1,1}\to \Lambda_{3,p}$. Following the references mentioned above, these maps already produce many automorphism groups of members of $\Lambda_{3,1}$ and  $\Lambda_{3,p}$. 
\item We say two polarised abelian varieties $(X_1,\lambda_1)$ and
$(X_2,\lambda_2)$ are isogenous, denoted $(X_1,\lambda_1)\sim
(X_2,\lambda_2)$, if there exists a quasi-isogeny
$\varphi:X_1\to X_2$ such that $\varphi^* \lambda_2=\lambda_1$. Let
$x=(X_0,\lambda_0)\in \calA_{g,1}(k)$ be a geometric point. Define 
\begin{equation}
  \label{eq:newLambda_x}
  \Lambda_x:=\{(X,\lambda)\in \calA_{g,1}(k): (X,\lambda) \sim
  (X_0,\lambda_0) \text{\ and \ }
  (X,\lambda)[p^\infty]\simeq (X_0,\lambda_0)[p^\infty] \}.
\end{equation}
Using the foliation structure on Newton strata due to
Oort~\cite{oort:foliation}, one can
show that the set $\Lambda_x$ is finite. Note that any two principally
polarised supersingular abelian varieties over~$k$ are isogenous,
cf.~\cite[Corollary 10.3]{yu:thesis}. Thus, the definition of 
$\Lambda_x$ in \eqref{eq:newLambda_x}
coincides that of $\Lambda_x$ in \eqref{eq:Lambdaxo} when $x\in \mathcal{S}_{g,1}$. That is, a mass function 
\begin{equation}
  \label{eq:massfcn}
  \Mass: \calA_{g,1}(k) \to \Q, \quad x\mapsto \Mass(\Lambda_x)
\end{equation}
extends the mass function $\Mass(x):=\Mass(\Lambda_x)$ defined on
$\mathcal{S}_{g,1}(k)$ as before. One would like to compute or 
study the properties of such a mass function on
$\calA_{g,1}(k)$, starting in low genus $g$. 
This problem may require developing more explicit
descriptions of the foliation structure on Newton strata, or employing analogues of the Rapoport-Zink space which was introduced in Subsection~\ref{ssec:mod}.
\end{enumerate}  
\end{questions*}

\section{Appendix: The intersection $C\cap \Delta$}
\label{sec:CcapD}

Let $C\subseteq \bbP^2$ be the Fermat curve defined by 
the equation $X_1^{p+1}+X_2^{p+1}+X_3^{p+1}=0$ and $\Delta\subseteq
\bbP^2$ the curve defined in Proposition~\ref{prop:dx}. 

In Section~\ref{sec:a1} we have seen the inclusion 
\[ C(\F_{p^2})\coprod C^0(\F_{p^6}) \coprod C^0(\F_{p^8})\coprod
C^0(\F_{p^{10}}) \subseteq C\cap \Delta \]
for $p>2$. In this (independent) section we study the complement of this inclusion. 

\subsection{Bounds for the degrees}\label{sec:bound}\

Let $\calQ$ denote the set of all conics (including degenerate ones)
$Q\subseteq \bbP^2$ defined over $\F_{p^2}$. Then
$\Delta=\cup_{Q\in \calQ} Q$. If $t\in C\cap \Delta$, then $t\in C\cap
Q$ for some $Q\in \calQ$ and hence
$\deg_{\F_{p^2}}(t):=[\F_{p^2}(t):\F_{p^2}]\le 2(p+1)$. We need the 
following well-known result.

\begin{theorem}[Kummer's Theorem]\label{thm:Kummer}
  Let $K$ be any field and $n\ge 1$ an integer and $a\in K^\times$. If
  $(n, {\rm char\,} K)=1$, and $\mu_n(K^\sep)\subseteq K$, and the element $a \pmod{(K^\times)^n}$ in $K^\times/(K^\times)^n$ has order $n$, then $[K(a^{1/n}):K]=n$. 
\end{theorem}

The authors are grateful to Ming-Lun Hsieh for providing the following proposition. 

\begin{proposition}\label{prop:ML}
  There exist a conic $Q\in \calQ$ and a point $t\in C\cap Q$ such
  that $\deg_{\F_{p^2}}(t)=(p+1)$.
\end{proposition}
\begin{proof}
  Choose a generator $u_1$ of $\F_{p^2}^\times$ such that
  $u_1^p+u_1=-a\neq 0$. Put $u:=a^{-1} u_1$ and let $\alpha$ be a
  $p+1$-th root of $u$. As $a\in \Fp^\times$, we have $u^p+u=-1$. 
  Since the element $u \pmod{(\F_{p^2}^\times)^{p+1}}$ in
  $\F_{p^2}^\times/(\F_{p^2}^\times)^{p+1}=\F_{p^2}^\times/(\F_{p}^\times)$
  has order $p+1$, one has $[\F_{p^2}(\alpha):\F_{p^2}]=p+1$ by
  Kummer's Theorem.
  Let  
\[ Q: X_1 X_2=u X_3^2 \quad \text{ and } \quad
t:=(\alpha:u \alpha^{-1}:1). 
\]
 One sees $t\in C$ as
  $\alpha^{p+1}+(u \alpha^{-1})^{p+1}+1=u+u^{p+1}\cdot
  u^{-1}+1=0$. So $t\in C\cap Q$ and $\deg_{\F_{p^2}}(t)=p+1$. 
\end{proof}

The following result, due to Akio Tamagawa, says that the upper bound $2(p+1)$ for $\deg_{\F_{p^2}}(t)$ in $C\cap \Delta$ can be realised.

\begin{proposition}\label{prop:akio}
  There exist a conic $Q\in \calQ$ and a point $t\in C\cap Q$ such
  that $\deg_{\F_{p^2}}(t)=2(p+1)$.
\end{proposition}

\begin{proof}[Construction] 
We first consider the case $p=2$. Let $\zeta$ be a primitive fifth
roof of unity in $\ol \F_2$. Since $(\Z/5\Z)^\times\simeq \< 2\!\!\mod
5\>$, we have $\F_2(\zeta)=\F_{2^4}$. One computes that 
$(1+\zeta)^3=1+\zeta+\zeta^2+\zeta^3\neq 1$ and
$(1+\zeta)^5=\zeta+\zeta^4 \neq 1$. 
Therefore $1+\zeta$ generates
the cyclic group $\F_{2^4}^\times \simeq C_{15}$.  
Choose $x,y,z\in \ol \F_2$ such that $x=1$,
$y^3=\zeta$ and $z^3=1+\zeta$, and put $t:=(x:y:z)$; 
we have $1+\zeta+(1+\zeta)=0$.
Since $\F_2(z)$ contains
$\F_2(\zeta)=\F_{2^4}$, we have $\F_2(z)=\F_{2^4}(z)$.  
Since $\<1+\zeta\>=\F_{2^4}^\times$, by Kummer's Theorem,
$\F_2(z)=\F_{2^4}(z)=\F_{2^{12}}$ and hence $\deg_{\F_{4}}(t)=6=2(p+1)$.
Since $x,y\in \F_{2^4}$, there exist $a,b,c\in \F_{2^2}$ such that $a
x^2+bxy+c y^2=0$. Let $Q\subseteq \bbP^2$ be the (degenerate) conic
defined by the equation $a X_1^2+b X_1 X_2+c X_2^2$. Then the point
$t\in C\cap Q$ satisfies the desired property.  

Assume now that $p>2$. We would like to find solutions
$t=(x:y:z)$ with $x\in \F_{p^{4(p+1)}}^\times$, $y\in
\F_{p^4}^\times\setminus \F_{p^2}^\times$,
and $z\in \F_{p^2}^\times$ satisfying the desired properties.

Let 
\[
f:\F_{p^4}^\times\to \F_{p^4}^\times/ (\F_{p^4}^\times)^{2(p+1)}
\] 
be the natural projection; one has  $\F_{p^4}^\times/
(\F_{p^4}^\times)^{2(p+1)} \simeq C_{2(p+1)}$ as $p\neq 2$. Consider the following three sets:
\begin{equation}
  \label{eq:XYZ}
  \begin{split}
&Z:=\{z^{p+1}: z\in \F_{p^2}^\times \}\simeq \F_{p}^\times; \\
&Y:=\{y^{p+1}: y\in \F_{p^4}^\times \}\setminus Z; \\
&X:=\{\xi\in \F_{p^4}^\times: \text{$f(\xi)$ generates the cyclic group
$C_{2(p+1)}$}\,  \}.    
  \end{split}
\end{equation}

The sets $Y$ and $Z$ are equipped with an $\F_p^\times$-action and we have
\begin{equation}
  \label{eq:sizeXYZ}
  \vert Z\vert =p-1, \quad \vert Y\vert =p^2(p-1), \quad \vert X\vert =(p^4-1)\cdot
\frac{\varphi(2(p+1))}{2(p+1)}.
\end{equation}
Let $g$ be the composition 
\[ 
\begin{CD}
  g: \F_{p^4}^\times @>{N}>>  \F_{p^2}^\times @>{\rm proj.}>>
  \F_{p^2}^\times/(\F_{p}^\times)^2\simeq C_{2(p+1)}, 
\end{CD}
\]
where $N(\alpha)=\alpha^{p^2+1}$ is the norm map. The map $f$ can be
identified with $g$ by a suitable choice of the generators. Since
the image $g(\F_{p}^\times)$ is trivial, the image $f(\F_{p}^\times)$
is also trivial. 
Thus, $X$ is also equipped with an $\F_{p}^\times$-action and hence $-X=X$. 

We would like to find
\begin{equation}
  \label{eq:etazetaxi}
  \eta + \zeta = \xi
\end{equation}
for some $\eta\in Y$, $\zeta\in Z$ and $\xi\in -X=X$.

Note that $X$, $Y$ and $Z$ are mutually disjoint: that $Y\cap Z=\emptyset$ follows by definition, and $X\cap Z=\emptyset$ follows from the fact that $\Fp^\times \subseteq \ker(f)$. Since $f((\F_{p^4}^\times)^{p+1})$ is the $2$-torsion subgroup
of $\F_{p^4}^\times/ (\F_{p^4}^\times)^{2(p+1)}\simeq C_{2(p+1)}$ and $f(Y)\subseteq
f((\F_{p^4}^\times)^{p+1})$, the image $f(Y)$ contains no 
generator of $C_{2(p+1)}$. Therefore, we also have $Y\cap X=\emptyset$.

We are working on the space $\bbP:=\F_{p^4}^\times/
\F_{p}^\times \simeq \bbP^3(\F_{p})$. The images of $X,Y$ and $Z$ in $\bbP$
are written as $\ol X$, $\ol Y$ and $\ol Z$, respectively. So
$\ol Z=\{\bar \zeta\}$ and 
\[ \vert \ol Z\vert =1, \quad \vert \ol Y\vert =p^2, \quad \vert \ol X\vert =(p^2+1)\cdot
\frac{\varphi(2(p+1))}{2}. \]
 
For each point $\bar \eta\in \ol Y$ ($\bar \eta\neq \bar \zeta$),
denote by $L_{\bar \eta}\subseteq \bbP$ the line joining the points 
$\bar \eta$ and $\bar \zeta$. To solve \eqref{eq:etazetaxi}, 
it suffices to prove that
\begin{equation}
  \label{eq:LintX}
  \left(\bigcup_{\bar \eta\in \ol Y} L_{\bar \eta} \right )\cap \ol X\neq
  \emptyset. 
\end{equation}
This is because if $\bar \xi\in L_{\bar \eta}\cap \ol X$ for some $\bar
\eta\in \ol Y$, then we have $a \eta+b\zeta=c\xi$ with $a,b,c\in
\F_p^\times$ and hence $\eta'+\zeta'=\xi'$ with $\eta'\in Y, \zeta'\in Z$
and $\xi'\in X$. 

\begin{lemma}\label{lem:Lbareta}
  For any two distinct points $\bar \eta_1$ and $\bar \eta_2$ of\, $\ol
  Y$, one has $L_{\bar \eta_1}\cap L_{\bar \eta_2}=\{\bar \zeta\}.$
\end{lemma}
\begin{proof}
  Suppose that $L_{\bar \eta_1}\cap L_{\bar \eta_2}\supsetneq \{\bar
  \zeta\}$. Then $L_{\bar \eta_1}= L_{\bar \eta_2}$ and $\bar
  \eta_2\in L_{\bar \eta_1}$. Therefore, $-\eta_2=a \eta_1+b \zeta$
  for $a,b\in \F_p^\times$ and hence we have
\[ \eta_2 + \eta_1'+ \zeta' =0 \]
for some $ \eta_1'\in Y$ and $\zeta'\in Z$. Now write
\[ 
\eta_2=(y_2)^{p+1}, \quad \eta_1'=(y_1')^{p+1}, \quad
\zeta'=(z')^{p+1}, 
\]
with $y_2,y_1'\in \F_{p^4}^\times\setminus \F_{p^2}^\times$ and $z'\in
\F_{p^2}^\times$. That is, we get a point $(y_2:y_1':z')\in
C(\F_{p^4})$. Since $C(\F_{p^4}) =C(\F_{p^2})$ by
Lemma~\ref{lem:Cmaxmim}, 
we have $y_2,y_1'\in
\F_{p^2}$,  contradiction. 
\end{proof}
 
By Lemma~\ref{lem:Lbareta},
\[ \bigcup_{\bar \eta\in \ol Y} L_{\bar \eta}=\{\bar \zeta\} \amalg \coprod_{\bar \eta\in \ol Y} L_{\bar \eta}-\{\bar \zeta\}, \]
and hence 
\[ \vert \bigcup_{\bar \eta\in \ol Y} L_{\bar \eta}\vert =1+\vert \ol Y \vert \cdot
p=p^3+1,\quad \text{and} \quad \vert \bbP-\bigcup_{\bar \eta\in \ol Y} L_{\bar
    \eta}\vert =p^2+p. \]
To show \eqref{eq:etazetaxi}, we check the inequality
\begin{equation}
  \label{eq:olX}
  \vert \ol X \vert=(p^2+1)\cdot
\frac{\varphi(2(p+1))}{2}> p^2+p
\end{equation}
for all $p\neq 2$. If $p=3$, then $\vert \ol X\vert =20>12$ holds. For $p\ge 5$, by the
inequality $\varphi(n)\ge \sqrt{n/2}$, it suffices to show 
\[ (p^2+1)\cdot
\frac{\sqrt{p+1}}{2}> p^2+p. \]
This follows from 
\[ (p^2+1)^2 (p+1)-4(p^2+p)^2=(p+1)(p^4-4p^3-2p^2+1)>0 \]
for $p\ge 5$. Therefore, the inequality \eqref{eq:olX} holds and
we have found $\eta, \zeta, \xi$ as in \eqref{eq:etazetaxi}.

Now write
\[ \zeta=z^{p+1} \ \  (\text{ for } z\in \F_{p^2}^\times), \quad \eta=y^{p+1}\ \ (\text{ for } y\in
\F_{p^4}^\times\setminus \F_{p^2}^\times). \]
Choose an element $x\in\Fpbar$ such that $x^{p+1}=-\xi\in
\F_{p^4}^\times$. Since the element $\xi \pmod{(\F_{p^4}^\times)^{p+1}}$ is a
generator in $\F_{p^4}^\times/(\F_{p^4}^\times)^{p+1}$, by Kummer's Theorem we have
\begin{equation}
  \label{eq:p+1}
  [\F_{p^4}(x):\F_{p^4}]=p+1.
\end{equation}
We claim that $\xi\not\in \F_{p^2}^\times$. 
Suppose for contradiction that $\xi\in
\F_{p^2}^\times$. Then 
\[ f(\xi)=g(\xi)\in
g(\F_{p^2}^\times)=(\F_{p^2}^\times)^2/(\F_{p}^\times)^2\subsetneq
\F_{p^2}^\times/(\F_{p}^\times)^2\simeq C_{2(p+1)}. \]     
Therefore, $f(\xi)$ cannot be a generator of $C_{2(p+1)}$,
contradiction. So since $\xi\in \F_{p^4}^\times\setminus
\F_{p^2}^\times$, we have $\F_{p^2}(x)\supset
\F_{p^2}(\xi)=\F_{p^4}$. This shows that
\[ \F_{p^2}(x)=\F_{p^4}(x), \quad \text{and}\quad
[\F_{p^2}(x):\F_{p^2}]=2(p+1) \]
by \eqref{eq:p+1}. Put $t:=(x:y:z)=(x/z:y/z:1)\in C(\Fpbar)$. Then we get
\begin{equation}
  \label{eq:2(p+1)}
  [\F_{p^2}(t):\F_{p^2}]=2(p+1).
\end{equation}
Since $y/z\in \F_{p^4}^\times\setminus
\F_{p^2}^\times$, there exist $b,c\in \F_{p^2}$ such that 
\[ \left ( \frac{y}{z} \right )^2+ b\left ( \frac{y}{z} \right )+c=0,
\quad \text{or} \quad y^2+byz+c z^2=0. \]
Let $Q\in \calQ$ be the (degenerate) conic defined by the equation 
$X_2^2+b X_2 X_3+cX_3^2=0$. Then $t\in C\cap Q$ and
$\deg_{\F_{p^2}}(t)=2(p+1)$. 
This completes the construction.
\end{proof} 

\subsection{Estimate of  $\vert C\cap \Delta\vert $}\label{sec:sizeCD}\

In this subsection, points in $C$ will mean geometric points and
$C\cap \Delta$ will mean the set-theoretic intersection. 
Define
\[ \calZ:=\{(t,Q)\in C \times \calQ: t\in \calQ\} \]
and consider the following natural maps:
\[
\begin{tikzcd}
\ & \mathcal{Z} \arrow{dl}{\pi}\arrow{dr}{q} & \ \\
C & \ & \mathcal{Q}
\end{tikzcd}
\]
The degree of the map $q$ is $2(p+1)$. For each $Q\in \calQ$, the
fibre over $Q$ has size
\[ 2(p+1)-\varepsilon_Q, \]
where $\varepsilon_Q=\sum_{r\ge 2}\varepsilon_{Q,r}$ with 
\[ \varepsilon_{Q,r}=\#\{t\in C\cap Q : \mathrm{mult}_{C\cap \Delta}(t)=r \}\cdot (r-1). \]
Thus, $\vert \calZ \vert =2(p+1)(p^{10}+p^{8}+p^6+p^4+p^2+1)-\varepsilon$, where
\begin{equation}
  \label{eq:error}
  \varepsilon:=\sum_{Q\in \calQ} \varepsilon_Q 
\end{equation}
is the error term coming from intersection multiplicities.

\begin{proposition}\label{prop:CcapD}
  We have  
  $\vert C\cap \Delta\vert =p^{11}+o(p^{11})-\varepsilon$ as
  $p \to \infty$, 
  where $\varepsilon$ is defined in \eqref{eq:error}.
\end{proposition}

\begin{remark}
We expect that $\varepsilon=o(p^{11})$. Then we would have $\vert C\cap \Delta\vert =p^{11}+o(p^{11})$ as $p \to \infty$. 
\end{remark}

\begin{proof}
For any integer $i\ge 1$, define 
\[ C_i:=\{t \in C(\Fpbar): \deg_{\F_{p^2}}(t)=i \}. \]
By Lemma~\ref{lem:Cmaxmim}, we have
\[ \vert C_1\vert =\vert C(\F_{p^2})\vert =p^3+1, \quad
\vert C_3\vert =\vert C^0(\F_{p^6})\vert =p^6+p^5-p^4-p^3, \]
\[ \vert C_4\vert =\vert C^0(\F_{p^8})\vert =p^8-p^6+p^5-p^3, \quad
\vert C_5\vert =\vert C^0(\F_{p^{10}})\vert =p^{10}+p^7-p^6-p^3. \] 
Let $\F_{p^2}[X_1,X_2,X_3]_2\subseteq \F_{p^2}[X_1,X_2,X_3]$ denote the subspace of homogeneous
polynomials of degree two.
For each point $t=(t_1:t_2:t_3)\in C$, 
the fibre $\pi^{-1}(t)$ is the set $\left(W_t-\{0\}\right)/\F_{p^2}^\times$, where
\[ W_t:=\{F\in \F_{p^2}[X_1,X_2,X_3]_2: F(t)=0\}. \]
They fit into the following exact sequence
\[
\begin{CD}
  0 @>>> W_t @>>> \F_{p^2}[X_1,X_2,X_3]_2 @>{\mathrm{ev}_t}>>
  \F_{p^2}\<t_1^2,t_2^2,t_3^2, t_1t_2,t_1t_3, t_2t_3\> @>>> 0. 
\end{CD}
\]
It follows that $\dim(W_t)=6-d(t)$ and  $\pi^{-1}(t)\simeq
\bbP^{5-d(t)}(\F_{p^2})$, where we redefine $d(t)$ as the dimension of
$\F_{p^2}\<t_1^2,t_2^2,t_3^2, t_1t_2,t_1t_3, t_2t_3\>$ -- even for $p=2$. Therefore, 
the numbers of fibres over $C_i$ for $i=1,3,4,5$ are
\[ (p^{8}+p^6+p^4+p^2+1), \quad (p^4+p^2+1), \quad (p^2+1), \quad 1, \]
respectively. Then the number of points in $\calZ$ over the union of $C_i$
for $i=1,3,4,5$ is given by
\[ 
\begin{split}
  A&:=(p^3+1)(p^{8}+p^6+p^4+p^2+1)+ (p^6+p^5-p^4-p^3)(p^4+p^2+1) \\
  &\quad \  +(p^8-p^6+p^5-p^3)(p^2+1)+(p^{10}+p^7-p^6-p^3) \\
  &=p^{11}+3p^{10}+2p^9+p^8+3p^7-p^6+p^{5}-2p^3+p^2+1.  
\end{split}\]
Thus, 
\[ 
\begin{split}
B&:=\#\{(t,Q)\in \calZ: \deg_{\F_{p^2}}(t)>5\}=\vert \calZ\vert -A \\  
&=p^{11}-p^{10}+p^8-p^7+3p^6+p^5+2p^{4}+4p^3+p^2+2p+1-\varepsilon.
\end{split}
 \]
Finally, 
\begin{equation}\label{eq:formulaCcapD}
\begin{split}
 \vert C\cap \Delta\vert &=\vert \mathrm{Im}(\pi)\vert =\vert C_1\vert +\vert C_3\vert + \vert C_4\vert +\vert C_5\vert +B \\   
&=p^{11}+2p^8+2p^6+3p^5+p^4+2p^3+p^2+2p+2-\varepsilon.
\end{split}  
\end{equation}
\end{proof}

\providecommand{\bysame}{\leavevmode\hbox to3em{\hrulefill}\thinspace}
\providecommand{\href}[2]{#2}

\end{document}